\documentclass[11pt]{amsart}

\usepackage[a4paper,margin=1in]{geometry}
\usepackage{amsmath,amssymb,amsthm,mathrsfs}
\usepackage{booktabs,array,enumitem,placeins}
\usepackage[T1]{fontenc}
\usepackage{lmodern,microtype,esint}
\usepackage[hidelinks]{hyperref}

\hypersetup{
bookmarksdepth=2,
pdftitle={On Critical Dimensions for Compactness in the Boundary Yamabe Problem, I},
pdfsubject={Scalar-flat, minimal-boundary, and negative-mean-curvature constructions in all higher dimensions},
pdfkeywords={boundary Yamabe problem, noncompactness, blow-up, critical dimension, manifolds with boundary}
}

\allowdisplaybreaks
\numberwithin{equation}{section}
\newtheorem{theorem}{Theorem}[section]
\newtheorem{maintheorem}{Theorem}
\newtheorem{maincorollary}{Corollary}
\newtheorem{proposition}[theorem]{Proposition}
\newtheorem{lemma}[theorem]{Lemma}
\newtheorem{corollary}[theorem]{Corollary}

\theoremstyle{remark}
\newtheorem{remark}[theorem]{Remark}

\newcommand{\R}{\mathbb R}
\newcommand{\Q}{\mathbb Q}
\newcommand{\eps}{\varepsilon}
\newcommand{\dd}{\,\mathrm d}
\newcommand{\pa}{\partial}
\newcommand{\tr}{\operatorname{tr}}
\newcommand{\diag}{\operatorname{diag}}
\newcommand{\divg}{\operatorname{div}}
\newcommand{\cB}{\mathcal B}
\newcommand{\cL}{\mathcal L}
\newcommand{\norm}[1]{\lVert#1\rVert}
\newcommand{\inner}[2]{\left\langle #1,#2\right\rangle}
\newcommand{\Sph}{\mathbb S}
\newcommand{\av}{\fint_{\Sph^{m-1}}}
\newcommand{\Id}{\operatorname{Id}}

\makeatletter
\def\l@section{\@tocline{1}{.15em \@plus\p@}{0pt}{}{\normalfont}}
\def\l@subsection{\@tocline{2}{0pt}{1.5em}{}{\normalfont}}
\makeatother

\title[Critical Dimensions for the Boundary Yamabe Problem, I]{On Critical Dimensions for Compactness\\ in the Boundary Yamabe Problem, I}

\author{Liuwei Gong}
\address[Liuwei Gong]{Department of Mathematics, Chinese University of Hong Kong, Shatin, NT, Hong Kong}
\email{lwgong@math.cuhk.edu.hk}

\author{Seunghyeok Kim}
\address[Seunghyeok Kim]{Department of Mathematics and Research Institute for Natural Sciences, College of Natural Sciences, Hanyang University, 222 Wangsimni-ro Seongdong-gu, Seoul 04763, Republic of Korea}
\email{shkim0401@hanyang.ac.kr}
\email{shkim0401@gmail.com}

\author{Monica Musso}
\address[Monica Musso]{Department of Mathematical Sciences, University of Bath, Bath BA2 7AY, United Kingdom}
\email{mm2683@bath.ac.uk}

\author{Juncheng Wei}
\address[Juncheng Wei]{Department of Mathematics, Chinese University of Hong Kong, Shatin, NT, Hong Kong}
\email{wei@math.cuhk.edu.hk}


\begin{document}
\begin{abstract}
We construct smooth, non-locally-conformally-flat metrics on the closed ball for which the boundary Yamabe equation admits $L^\infty$-unbounded sequences of positive solutions.
The background metric and the prescribed scalar and boundary mean curvatures are fixed along each sequence.
For zero scalar curvature and positive boundary mean curvature, such examples exist with umbilic boundary in every dimension $N\ge22$ and with nonumbilic boundary in every $N\ge15$.
For positive scalar curvature and minimal boundary, the corresponding ranges are $N\ge21$ and $N\ge15$.
For every $N\ge9$, examples with either umbilic or nonumbilic boundary also exist when the scalar curvature is $N(N-1)$ and the prescribed boundary mean curvature is below a negative threshold depending on $N$.
The construction uses polynomial metric perturbations and corrected bubbles on the half-space, followed by conformal compactification to the ball.
We evaluate or estimate the correction solving the linearized Neumann or Robin boundary problem and its contribution to the quadratic term of the reduced energy.
We prove that this quadratic term has negative, nondegenerate local extrema with respect to tangential translations and scale in every stated dimension.
\end{abstract}

\maketitle

\tableofcontents

\section{Introduction}
Let $(M^N,g)$ be a smooth compact Riemannian manifold with boundary, $N\ge3$.
A natural extension of the classical Yamabe problem is to find a metric in the conformal class of $g$ with constant scalar curvature and constant boundary mean curvature.
After Cherrier's work on nonlinear Neumann problems~\cite{Cherrier}, Escobar studied two basic forms of this question: constant scalar curvature with minimal boundary \cite{EscobarJDG}, and zero scalar curvature with
constant boundary mean curvature~\cite{EscobarAnnals}.
The prescription of both curvature constants was subsequently developed by Escobar~\cite{EscobarIndiana} and Han--Li~\cite{HanLiDuke,HanLiCAG}; further existence results were obtained by Chen--Ruan--Sun~\cite{ChenRuanSun}.

Throughout the paper, $\sigma$ and $\eta$ are real constants.
For a positive function $u$, set $\hat g:=u^{4/(N-2)}g$.
The conditions $R_{\hat g}=N(N-1)\sigma$ and $H_{\hat g}=\eta$ are equivalent to finding a positive solution of
\begin{equation}\label{eq:intro-yamabe}
\begin{cases}
-\Delta_g u+c_NR_gu
=\dfrac{N(N-2)}4\sigma u^{(N+2)/(N-2)}&\text{in }M,\\
\partial_\nu u+\dfrac{N-2}{2}H_gu
=\dfrac{N-2}{2}\eta u^{N/(N-2)}&\text{on }\partial M.
\end{cases}
\end{equation}
Here $\nu$ is the outward unit normal, $H_g$ is the boundary mean curvature, and $c_N:=(N-2)/(4(N-1))$.
We call $(\sigma,\eta)=(1,\kappa)$ Type I and $(0,2)$ Type II.
The minimal Type-I problem has $\kappa=0$.
For $\sigma>0$, constant rescaling changes $(\sigma,2)$ to $(1,2/\sqrt\sigma)$, so the scalar-flat problem corresponds to the normalized limit $\kappa\to+\infty$ as $\sigma\downarrow0$.
A boundary is umbilic if its trace-free second fundamental form vanishes identically; this condition is preserved by conformal change.

Besides existence, one asks whether the full set of positive solutions is compact when the background metric and prescribed curvatures are fixed.
For closed manifolds, Brendle~\cite{Brendle} constructed smooth noncompact examples in dimensions $N\ge52$, and Brendle--Marques~\cite{BrendleMarques} extended the construction to every $N\ge25$.
Under the positive-mass hypothesis, Khuri--Marques--Schoen~\cite{KhuriMarquesSchoen} proved the complementary compactness theorem in dimensions at most $24$.
Their proof excludes blow-up by comparing the mass of the asymptotically flat metric obtained from a Green function with the local Pohozaev identity.

For scalar-flat metrics with boundary, compactness depends also on the boundary geometry.
Almaraz~\cite{AlmarazCompactness} proved compactness for $N\ge7$ when the trace-free second fundamental form is nowhere zero, under the positivity and conformal assumptions of that theorem.
Kim--Musso--Wei~\cite{KimMussoWei} proved compactness in dimensions four and five, and in dimension six under the same nonvanishing condition, assuming positive conformal type and excluding the conformal ball.
These results leave open the possibility of concentration at points where the trace-free second fundamental form vanishes.

Almaraz~\cite{AlmarazBlowup} constructed scalar-flat examples on the ball with umbilic boundary in every dimension $N\ge25$.
Disconzi--Khuri~\cite[Theorem~1.4]{DisconziKhuri} obtained noncompact examples in the same range with positive scalar curvature and minimal boundary.
For positive scalar and boundary mean curvatures, Chen--Wu~\cite{ChenWu} obtained umbilic examples for $N\ge62$ and every positive boundary parameter; Ho--Shin~\cite[Theorem~1.2]{HoShin} obtained examples for $N\ge35$
and sufficiently small positive parameter.

Our main result is the following.

\begin{maintheorem}
\label{thm:main}
Let $B^N:=\{x\in\R^N:|x|\le1\}$.
The following noncompactness statements hold.
\begin{enumerate}[label=\textup{(\roman*)},leftmargin=*]
\item For $(\sigma,\eta)=(0,2)$, an $L^\infty(B^N)$-unbounded sequence of positive solutions exists on a smooth metric with nonumbilic boundary for every integer $N\ge15$.
For every integer $N\ge22$, such a sequence also exists on a smooth metric with umbilic boundary.
\item For $(\sigma,\eta)=(1,0)$, the same conclusions hold with nonumbilic boundary for every integer $N\ge15$ and with umbilic boundary for every integer $N\ge21$.
\item For every integer $N\ge9$, there is a constant $L_N>0$ such that, for each fixed $\kappa\le-L_N$, an $L^\infty(B^N)$-unbounded sequence exists for $(\sigma,\eta)=(1,\kappa)$ on a smooth metric with umbilic boundary and on a smooth
metric with nonumbilic boundary.
\end{enumerate}
In each assertion, the metric $g$ is not locally conformally flat and is independent of the sequence index.
Every member $u_j\in C^\infty(B^N)$ is positive and solves \eqref{eq:intro-yamabe} with the same constants $(\sigma,\eta)$.
In fact, the boundary traces are unbounded:
\begin{equation}\label{eq:intro-blowup-sequence}
\max_{\partial B^N}u_j\to\infty
\qquad\text{as }j\to\infty.
\end{equation}
The metric may depend on $N$, the prescribed constants, and whether the boundary is required to be umbilic or nonumbilic.
\end{maintheorem}

A separate construction yields explicit sufficient bounds on $\kappa$.
\begin{maincorollary}
\label{cor:explicit-negative-bounds}
For every integer $N\ge9$ and every fixed
\[
\kappa\le-\frac6{\sqrt{N-8}},
\]
there exist a smooth metric $g$ on $B^N$ with nonumbilic boundary and a positive solution sequence for $(\sigma,\eta)=(1,\kappa)$ with all the properties in Theorem~\ref{thm:main}.
For every integer $N\ge11$ and every fixed
\[
\kappa\le-\frac{\sqrt{23+4\sqrt{35}}}{\sqrt{N-10}},
\]
the same conclusion holds with umbilic boundary.
\end{maincorollary}
We will prove Corollary~\ref{cor:explicit-negative-bounds} in Appendix~\ref{app:normal-weyl-bounds}.

\begin{maincorollary}\label{thm:positive-end-main}
The equation with $(\sigma,\eta)=(1,\kappa)$ admits metrics and solution sequences with the properties in Theorem~\ref{thm:main} in the following additional ranges.
\begin{enumerate}[label=\textup{(\roman*)},leftmargin=*]
\item For each fixed integer $N\ge15$, there is $\kappa_0>0$ such that the conclusion holds with nonumbilic boundary for every fixed $\kappa$ satisfying $|\kappa|<\kappa_0$.
For each fixed integer $N\ge21$, the same statement holds with umbilic boundary.
\item For each fixed integer $N\ge15$, there is $K>0$ such that the conclusion holds with nonumbilic boundary for every fixed $\kappa>K$.
For each fixed integer $N\ge22$, the same statement holds with umbilic boundary.
\end{enumerate}
The constants $\kappa_0$ and $K$ may depend on $N$ and on whether umbilic or nonumbilic boundary is required.
The metric is fixed along each sequence but may depend on the prescribed value of $\kappa$.
\end{maincorollary}

\begin{remark}
In the companion paper~\cite{GongKimMussoWeiCompactness}, we will prove the complementary compactness results under positive conformal type, exclusion of the conformal round hemisphere, and the positive-mass and rigidity hypotheses for the associated Green function metrics.
The upper dimensions are $14$ in general and $21$ for umbilic boundary in the scalar-flat problem, and $14$ and $20$, respectively, in the minimal-boundary problem.
For every fixed $\kappa\in\R$, compactness holds through dimension eight.
Together with the complementary curvature intervals in Corollary~\ref{thm:positive-end-main}, these results establish the optimality of the dimension ranges proved here.
\end{remark}

\begin{remark}
For umbilic boundary, Theorem~\ref{thm:main} improves the previously known noncompactness range $N\ge25$ to $N\ge22$ for scalar-flat metrics and to $N\ge21$ for positive scalar curvature with minimal boundary.
After the normalization $u\mapsto2^{-(N-2)/2}u$, the latter examples in dimensions $21\le N\le24$ satisfy the hypotheses of \cite[Theorem~1.1]{DisconziKhuri} and have an $L^\infty$-unbounded set of positive solutions.
We have been unable to verify the argument for the Neumann condition in \cite[Proposition~5.1]{DisconziKhuri}, specifically the passage following \cite[(5.6)]{DisconziKhuri} from a zero normal derivative of the sum to zero normal derivatives of its individual harmonic components.
\end{remark}

\begin{remark}
To our knowledge, this is the first construction of noncompact solution sequences for the boundary Yamabe equation on a fixed non-locally-conformally-flat background using scalar-derived tensors generated by Hessians of harmonic polynomials.
The earlier boundary constructions in \cite{AlmarazBlowup,DisconziKhuri,ChenWu,HoShin} use polynomial tensors generated by an algebraic Weyl tensor.
Here harmonic quadratic and cubic polynomials generate the tensors in \eqref{eq:full-scalar-module} and \eqref{eq:scalar-cubic-seed}, respectively.
For these tensors, the correction in \eqref{eq:linear-corrector-inverse} need not vanish identically for a centered bubble and contributes to the energy and its translation--scale Hessian.
\end{remark}

\begin{remark}
Our second main ingredient is the evaluation of the scalar correction without an explicit formula for the exact half-space corrector.
We construct a rational multiple of the bubble that solves the interior linearized equation exactly but leaves a boundary residual.
The Green identity in Lemma~\ref{lem:corrector-pairing-identity} expresses the remaining contribution to the corrected energy as a boundary-response pairing.
For the scalar-flat and minimal-boundary problems, we evaluate these pairings exactly through finite beta sums and algebraic telescoping of the spectral series.
A precursor is the decomposition in \cite[Section~4.1]{KimMussoWei} into an explicit interior particular solution and a harmonic remainder satisfying a forced Robin condition, whose quadratic energy is estimated.
Here the exact pairings also determine all translation and scale derivatives required for the noncompactness construction.
\end{remark}

The construction places disjoint rescaled polynomial metric perturbations near distinct boundary points tending to a single boundary point.
The scalar coefficients multiplying these tensors are chosen relative to the localization scales so that the metric perturbation extends smoothly and vanishes to infinite order at the limiting point; see Lemma~\ref{lem:localized-energy}.
In the nonumbilic examples, the trace-free second fundamental form is nonzero at the boundary centers of these perturbations.
On the minimal backgrounds used here, \eqref{eq:intro-yamabe} implies $\partial_\nu u_j<0$ for negative prescribed mean curvature, so the maxima lie in the interior even though the boundary supremum diverges as in
\eqref{eq:intro-blowup-sequence}.

Table~\ref{tab:intro-profiles} summarizes the six constructions.
For the local construction, use half-space coordinates $x=(\bar x,t)\in\R^{N-1}\times[0,\infty)$, with inward normal coordinate $t:=x_N$.
Before cutoff, the metric is $g:=\exp(\mu h)$, where $\mu\in\R$ is a scalar coefficient and $h$ is a tangential polynomial tensor.
We call $h$ the metric profile and use total degree in $(\bar x,t)$.
\begin{table}[htbp]
\centering\small
\setlength{\tabcolsep}{5pt}\renewcommand{\arraystretch}{1.15}
\begin{tabular}{@{}llccl@{\hspace{1.4em}}l@{}}\toprule
$(\sigma,\eta)$ & Boundary & Dimension & Degree & Metric profile & Extremum\\\midrule
$(0,2)$ & nonumbilic & $N\ge15$ & $6$
& $h^{\mathrm{II,nu}}_{N,\tau}$, \eqref{eq:fixed-nonumbilic-families} & minimum\\
$(1,0)$ & nonumbilic & $N\ge15$ & $6$
& $h^{\mathrm{I,nu}}_{N,\tau}$, \eqref{eq:fixed-nonumbilic-families} & minimum\\
$(0,2)$ & umbilic & $N\ge22$ & $8$
& $h^{\mathrm{II,u}}_{N,\tau}$, \eqref{eq:fixed-umbilic-families} & minimum\\
$(1,0)$ & umbilic & $N\ge21$ & $9$
& $h^{\mathrm{I,u}}_{N,\tau}$, \eqref{eq:fixed-umbilic-families} & minimum\\
$(1,-L)$ & umbilic & $N\ge9$ & $3$
& $h^-_{N,L,\tau,0}$, \eqref{a:eq:hLmain} & maximum\\
$(1,-L)$ & nonumbilic & $N\ge9$ & $3$
& $h^-_{N,L,\tau,\theta}$, \eqref{a:eq:hLmain} & maximum\\\bottomrule
\end{tabular}
\caption{Polynomial metric profiles for the six cases of Theorem~\ref{thm:main}.
The pair $(\sigma,\eta)$ prescribes scalar curvature $N(N-1)\sigma$ and boundary mean curvature $\eta$.
In the last two rows, $L>0$ is fixed and sufficiently large, depending on $N$.
The last column refers to the corrected quadratic energy $F_h$ in \eqref{eq:common-Schur-formula}; every listed extremum has negative value and is nondegenerate in all translation and scale variables.}
\label{tab:intro-profiles}
\end{table}

\medskip
\noindent\textbf{Metric perturbations and the scalar correction.}
The construction requires a polynomial metric perturbation $h$ and the first-order correction to the conformal factor in \eqref{eq:projected-linear-corrector}--\eqref{eq:linear-corrector-inverse}, which we call the scalar correction.
For the prescribed curvatures, consider the corresponding bubble $U_z$ in \eqref{eq:common-TypeI-bubble} or \eqref{eq:common-TypeII-bubble}, with tangential center $\xi$ and scale $\eps$.
Perturbing the metric introduces an error in the equation for $U_z$.
The correction cancels the first-order error caused by $h$ and contributes to the quadratic energy in \eqref{eq:corrected-energy-Taylor}, which we use to select the bubble parameters.

For the scalar-flat and minimal-boundary problems, fix a symmetric trace-free matrix $A$ and set $Q_A(\bar x):=\bar x^{\mathsf T}A\bar x$.
The tensors $S_N^{(k)}(b,e)$ in \eqref{eq:full-scalar-module} are formed from the harmonic polynomial $Q_A$, its first two derivatives, and polynomial factors in $\bar x$, with the tangential trace removed.
We use the finite sums in \eqref{eq:full-scalar-profile},
\[
h(\bar x,t)=\sum_{k,p}c_{kp}t^pS_N^{(k)}(b_{kp},e_{kp}).
\]
The tensors $P_N^{(k)}=S_N^{(k)}(b_k,e_k)$ in \eqref{a:eq:Achain} are obtained from this family by choosing $b_k,e_k$ so that $\partial_a(P_N^{(k)})_{ab}=0$.
These divergence-free generators are used in the nonumbilic scalar-flat and minimal-boundary profiles.
The two umbilic profiles use other choices of $b_{kp},e_{kp}$ in the same tensor formula and have nonzero divergence.
The two pairs of profiles are defined in \eqref{eq:fixed-nonumbilic-families} and \eqref{eq:fixed-umbilic-families}, respectively.

Each $S_N^{(k)}(b,e)$ depends only on $\bar x$, so only the terms with $p=1$ contribute to $\partial_t h(\bar x,0)$.
Omitting those terms makes the boundary of $\exp(\mu h)$ totally geodesic, by Lemma~\ref{lem:polynomial-boundary-geometry}.
The nonumbilic profiles instead contain $\tau tA$ and satisfy $\partial_t h(0,0)=\tau A\ne0$ for the selected $\tau\ne0$.
For negative mean curvature, we also use the Hessian of the harmonic cubic in \eqref{eq:scalar-cubic-seed}; the resulting degree-three metric perturbation is \eqref{a:eq:hLmain}.

For a fixed $h$, Lemma~\ref{lem:source-kernel-orthogonality} computes the first-order error caused by perturbing the metric:
\[
S_z[h]=h_{ij}\partial_{ij}U_z
+(\partial_jh_{ij})\partial_iU_z
+c_N(\partial_i\partial_jh_{ij})U_z.
\]
The last two terms contain the divergence and double divergence of $h$ and must be retained for the umbilic profiles.
The associated functional is $\Lambda_z[h](\psi)=\int_{\R^N_+}S_z[h]\psi\dd x$ in \eqref{eq:polynomial-source-functional}.
The Jacobi form \eqref{eq:common-Jacobi-form} defines $\mathcal L_z$ on the complement specified by \eqref{eq:common-orthogonality-functionals}; the boundary condition is Neumann when the prescribed mean curvature is
zero and Robin otherwise.
The correction in \eqref{eq:linear-corrector-inverse} and the quadratic energy in \eqref{eq:common-Schur-formula} are
\begin{align*}
v_z[h]&=-\mathcal L_z^{-1}\Lambda_z[h],\\
F_h(z)&=\mathcal Q[h](z)-\Lambda_z[h]\bigl(\mathcal L_z^{-1}\Lambda_z[h]\bigr).
\end{align*}
Here $\mathcal Q[h](z)$ is the quadratic coefficient in \eqref{eq:H-quadratic-cutoff-exact}, with $h_R$ replaced by $h$, before correcting the bubble.
Proposition~\ref{prop:scalar-elimination} proves the energy expansion and justifies the cutoff limits for polynomial tensors.

The value of $F_h$ and its parameter derivatives depend on inverse pairings, which we compute as follows.
For a fixed bubble, put $S:=\Lambda_z[h]$ and $Z:=\mathcal L_z^{-1}S=-v_z[h]$, and let $X$ be an explicit approximation in the same kernel complement.
The diagonal case of \eqref{eq:corrector-pairing-identity} reads
\[
\langle S,Z\rangle
=2\langle S,X\rangle-\mathcal B_z(X,X)
+\mathcal B_z(Z-X,Z-X).
\]
The first two terms are explicit; the last term is determined by the residual of $X$.
For the scalar-flat and minimal-boundary profiles, $Z-X$ solves the homogeneous interior equation with the negative boundary residual of $X$; its pairings are evaluated in Lemma~\ref{lem:endpoint-boundary-response}.
For negative mean curvature, Proposition~\ref{prop:cubic-whole-space-energy} supplies explicit whole-space solutions, and Section~\ref{sec:cubic-comparison} estimates their boundary residuals after restriction to the half-space.

Once these pairings have been computed, a quadratic equation selects one coefficient of $h$ so that $(\xi,\eps)=(0,1)$ is a critical point of $F_h$; see \eqref{eq:larger-coefficient-root} and \eqref{eq:cubic-unit-coefficient}.
This coefficient is then held fixed during translation, dilation, and concentration.
We prove that the critical value is negative and that the full translation--scale Hessian is definite.
The nonlinear reduction and localization in Theorem~\ref{thm:fixed-metric-transfer} turn these strict extrema into positive solution sequences on one fixed smooth background, following
\cite{Brendle,BrendleMarques,AlmarazBlowup}.

\medskip
\noindent\textbf{Organization.}
Section~\ref{sec:variational} introduces the energy, the model bubbles, and their linearized equations.
Section~\ref{sec:polynomial-reduction} derives the correction to a bubble when the metric is perturbed and the resulting change in energy.
Sections~\ref{sec:algebra}--\ref{sec:certification} construct the polynomial metric perturbations and prove the energy minima for the scalar-flat and minimal-boundary problems.
Section~\ref{sec:TI-well} treats negative boundary mean curvature.
Section~\ref{sec:common-analytic} combines the local perturbations into a single smooth metric and constructs the unbounded solution sequences; Section~\ref{sec:applications} completes the proofs of the main results.
The appendices contain the defining coefficients, the supporting calculations and verification data, and the proof of the explicit negative-curvature bounds.
Only Proposition~\ref{prop:four-scalar-minima} uses computer-assisted calculations; these are documented in Appendices~\ref{app:sign-witnesses} and~\ref{app:verification}.

\medskip
\noindent\textbf{Conventions.}
Throughout the paper,
\[
m:=N-1,\qquad x=(\bar x,t)\in\R^m\times[0,\infty),\qquad
t:=x_N,\qquad s:=|\bar x|^2.
\]
Thus $|x|^2=s+t^2$, and $\nabla_{\bar x}$ and $\Delta_{\bar x}$ denote tangential derivatives.
Spatial indices $i,j,k$ range from $1$ to $N$; tangential indices $a,b,\ldots$ range from $1$ to $m$.
Parameter indices $\alpha,\beta$ range from $1$ to $N$ for $z\in\R^{N-1}\times(0,\infty)$.
Indices on homogeneous tensor rows are specified where they occur.
The letter $z$ is reserved for bubble parameters; shifted normal coordinates are denoted by $\widehat t$.
In the spectral calculations for the scalar-flat problem $(\sigma,\eta)=(0,2)$ and the minimal-boundary problem $(\sigma,\eta)=(1,0)$, $q$ is the degree of a tangential spherical harmonic and $n$ is the summation index.
The total spherical degree in those calculations is $j:=n+q$.
We write $\ell:=\log\eps$ for logarithmic bubble scale.
Coefficient vectors and matrices are written in bold.
Repeated spatial indices are summed.
A tangential tensor has $h_{iN}=0$, and $h^\top$ denotes its tangential block.
We denote the Euclidean metric by $g_{\mathrm{euc}}$; its components are $\delta_{ij}$, and $\delta^{ij}$ denotes the inverse matrix.
Unsubscripted gradients, Laplacians, and contractions are Euclidean.
The parameter $\mu$ multiplies the metric perturbation $h$ in \eqref{eq:metric-variation-path}, whereas $\eps$ is the bubble scale in \eqref{eq:common-TypeI-bubble}--\eqref{eq:common-TypeII-bubble}.
Tensor products and exponentials use the matrix representation; $A^{\mathsf T}$ denotes transpose.
We use $(\divg _g h)_j:=\nabla^ih_{ij}$, and $(X\lrcorner h)_j:=X^ih_{ij}$.
The measures $\dd v_g$, $\dd\sigma_g$, and $|\Sph^m|$ have their Riemannian meanings.
A boundary center $(\bar b,0)$ is identified with $\bar b$ in bubble subscripts, and $B_r^+(b):=B_r^N(b)\cap\R^N_+$.

Polynomial degree is total degree.
We write $d_*$ for a degree bound, $d_i:=\deg h_i$ for the degree of a homogeneous row, and $d_0$ for the lowest row degree.
The two shape coefficients in $S_N^{(k)}(b,e)$ are denoted by $b,e$, independently of polynomial degree.
Fix $\chi\in C^\infty([0,\infty);[0,1])$ with $\chi=1$ on $[0,1]$ and $\chi=0$ on $[2,\infty)$, and set $\chi_R(x):=\chi(|x|/R)$.
Constants may depend on the fixed dimension, curvature, tensor, parameter set, and cutoff, but not on the localization parameters.

We use the induced metric $g^\top:=g|_{T\partial M}$ and the second fundamental form $L_g(X,Y):=g(\nabla_X\nu,Y)$, with $H_g:=(N-1)^{-1}\operatorname{tr}_{g^\top}L_g$.
On the Euclidean half-space, $\pa_t$ points inward and $\nu=-\pa_t$.
In block coordinates with $g_{aN}=0$ and $g_{NN}=1$,
\[
L_{ab}=g(\nabla_{\pa_a}\nu,\pa_b)
=-\frac12\pa_tg_{ab},\qquad
H_g=\frac1{N-1}g^{ab}L_{ab}.
\]
In particular,
\begin{equation*}
\pi_g:=L_g-H_gg^\top.
\end{equation*}
For a symmetric matrix field $h$, $\exp(h)$ denotes its matrix exponential and the metric with that coefficient matrix.
If $h$ is tangential, this metric has the form
\begin{equation*}
\exp(h)=\dd t^2+
\sum_{a,b=1}^{N-1}[\exp(h)]_{ab}\dd x_a\dd x_b.
\end{equation*}
If $\tr h=0$, then
\[
\begin{gathered}
\det(\exp(h))=1,\qquad
\dd v_{\exp(h)}=\dd x,\qquad \dd\sigma_{\exp(h)}=\dd\bar x,\\
H_{\exp(h)}=-\frac{\pa_t\log\det(\exp(h))}{2(N-1)}=0.
\end{gathered}
\]

\medskip
\noindent\textbf{Statement on the Use of AI.}
During the preparation of this manuscript, ChatGPT and Codex (OpenAI) were used to assist with mathematical derivations, symbolic calculations, verification code, and the organization and wording of the text.
The arguments and computer-assisted checks are documented in the manuscript and its verification materials.
The authors are responsible for the mathematical content and references.

\medskip
\noindent\textbf{Acknowledgement.}
L. Gong and J. Wei gratefully acknowledge support from the Research Grants Council of Hong Kong (RGC) through the project: \emph{On Fujita equation in the critical or supercritical regime}.
S. Kim was supported by Basic Science Research Program through the National Research Foundation of Korea (NRF) funded by the Ministry of Science and ICT (RS2025-00558417), and by Samsung Science and Technology Foundation
under Project Number SSTF-BA2601-01.

\section{Preliminaries}\label{sec:variational}
We recall the variational setting, bubbles and linearized operators of the boundary Yamabe reduction.
The scalar-flat inverse follows from \cite[Proposition~2.7]{AlmarazBlowup}, and finite-mean-curvature nondegeneracy from \cite[Lemma~2.1]{AlmarazWang}.

\subsection{The variational formulation}

For fixed $(\sigma,\eta)$, define
\begin{equation}\label{eq:common-functional-expanded}
\begin{aligned}
E_g(u):={}&\int_M\bigl(|\nabla u|_g^2+c_NR_gu^2\bigr)\dd v_g
+\frac{N-2}{2}\int_{\pa M}H_gu^2\dd\sigma_g\\
&-\frac{(N-2)^2\sigma}{4}\int_M(u_+)^{2N/(N-2)}\dd v_g
-\frac{(N-2)^2\eta}{2(N-1)}
\int_{\pa M}(u_+)^{2(N-1)/(N-2)}\dd\sigma_g,
\end{aligned}
\end{equation}
where $u_+:=\max\{u,0\}$.
We use the conformal operators
\[
L_g^{\mathrm{conf}}:=-\frac{4(N-1)}{N-2}\Delta_g+R_g,\qquad
B_g^{\mathrm{conf}}:=\frac{2(N-1)}{N-2}\partial_\nu+(N-1)H_g.
\]
The positive critical points satisfy
\begin{equation*}
\begin{cases}
L_g^{\mathrm{conf}}u=N(N-1)\sigma u^{(N+2)/(N-2)}&\text{in }M,\\
B_g^{\mathrm{conf}}u=(N-1)\eta u^{N/(N-2)}&\text{on }\pa M.
\end{cases}
\end{equation*}
In particular, for $(\sigma,\eta)=(0,2)$,
\begin{equation}
\begin{aligned}
E_g(u)={}&\int_M\bigl(|\nabla u|_g^2+c_NR_gu^2\bigr)\dd v_g
+\frac{N-2}{2}\int_{\pa M}H_gu^2\dd\sigma_g\\
&-\frac{(N-2)^2}{N-1}\int_{\pa M}(u_+)^{2(N-1)/(N-2)}\dd\sigma_g.
\end{aligned}
\end{equation}
Multiplication $u\mapsto cu$, $c>0$, changes the prescribed parameters by
\[
(\sigma,\eta)\mapsto
\bigl(c^{-4/(N-2)}\sigma,c^{-2/(N-2)}\eta\bigr).
\]

\begin{proposition}\label{prop:geometric-action}
On a compact manifold with boundary define
\begin{equation}
\begin{aligned}
\mathcal A_{\sigma,\eta}(\hat{g}):={}&\int_M R_{\hat{g}}\dd v_{\hat{g}}
+2(N-1)\int_{\pa M}H_{\hat{g}}\dd\sigma_{\hat{g}}\\
&-(N-1)(N-2)\sigma\operatorname{Vol}_{\hat{g}}(M)
-2(N-2)\eta\operatorname{Area}_{\hat{g}}(\pa M).
\end{aligned}
\end{equation}
For $u>0$ and $\hat{g}=u^{4/(N-2)}g$,
\begin{equation}\label{eq:energy-geometric-action}
E_g(u)=c_N\mathcal A_{\sigma,\eta}(\hat{g}).
\end{equation}
With covariant metric variation $h$, its first variation is
\begin{equation}\label{eq:geometric-action-first-variation}
\begin{aligned}
D\mathcal A_{\sigma,\eta}(\hat{g})[h]
={}&-\int_M\left\langle\operatorname{Ric}_{\hat{g}}-\tfrac12R_{\hat{g}}\hat{g}
+\tfrac12(N-1)(N-2)\sigma \hat{g},h\right\rangle_{\hat{g}}\dd v_{\hat{g}}\\
&-\int_{\pa M}\left\langle L_{\hat{g}}-(N-1)H_{\hat{g}}\hat{g}^\top
+(N-2)\eta \hat{g}^\top,h^\top\right\rangle_{\hat{g}}\dd\sigma_{\hat{g}}.
\end{aligned}
\end{equation}
In particular, an Einstein model with $\operatorname{Ric}_{\hat{g}}=(N-1)\sigma \hat{g}$ and $L_{\hat{g}}=\eta \hat{g}^\top$ is stationary under all smooth metric variations.
\end{proposition}
\begin{proof}
The conformal transformation laws in our outward-normal convention are
\[
R_{\hat{g}}u^{(N+2)/(N-2)}=-c_N^{-1}\Delta_gu+R_gu,\qquad
H_{\hat{g}}u^{N/(N-2)}=H_gu+\frac2{N-2}\pa_\nu u.
\]
Using the conformal volume and area measures, we obtain
\begin{align*}
\int R_{\hat{g}}\dd v_{\hat{g}}
&=c_N^{-1}\int|\nabla u|_g^2\dd v_g+\int R_gu^2\dd v_g
-c_N^{-1}\int_{\pa M}u\pa_\nu u\dd\sigma_g,\\
2(N-1)\int H_{\hat{g}}\dd\sigma_{\hat{g}}
&=2(N-1)\int H_gu^2\dd\sigma_g
+\frac{4(N-1)}{N-2}\int u\pa_\nu u\dd\sigma_g.
\end{align*}
The fluxes cancel.
After multiplication by $c_N$, the boundary coefficient is $(N-2)/2$, and the volume and area multipliers are $(N-2)^2\sigma/4$ and $(N-2)^2\eta/(2(N-1))$, respectively.
This proves \eqref{eq:energy-geometric-action}.

For completeness, put $\mathcal H:=(N-1)H_{\hat{g}}$, $b:=h(\nu,\cdot)|_{T\pa M}$, and $V_i:=\nabla^jh_{ij}-\nabla_i\tr_{\hat{g}} h$.
By differentiating the unit normal and the tangential trace of $L_{\hat{g}}$, we find
\[
2D\mathcal H_{\hat{g}}[h]=-V(\nu)-\divg_{\pa M}b
 -\langle L_{\hat{g}},h^\top\rangle_{\hat{g}}.
\]
The scalar-curvature variation is $DR_{\hat{g}}[h]=-\langle\operatorname{Ric}_{\hat{g}},h\rangle_{\hat{g}}+\divg_{\hat{g}} V$.
Its flux cancels the first term in $2D\mathcal H_{\hat{g}}[h]$; the integral of the tangential divergence is zero.
Differentiating the two measures leaves $-\langle L_{\hat{g}}-\mathcal H \hat{g}^\top,h^\top\rangle$ on the boundary.
The volume and area variations are one half of the respective traces of $h$, which proves \eqref{eq:geometric-action-first-variation}.
Both displayed tensors vanish at the stated model.
\end{proof}

The Type-I bubbles compactify to round caps of sectional curvature one, and the Type-II bubbles to a flat ball of radius $1/2$.
These are the stationary models of Proposition~\ref{prop:geometric-action}.
For a compactly supported metric perturbation and a positive scalar factor whose ratio to the bubble extends smoothly to the compactification, the artificial flux is $O(R^{2-N})$ as $R\to\infty$.

\subsection{Model bubbles and projected inverses}

From this subsection onward the analytic domain is $\R^N_+$, with boundary $\pa\R^N_+=\{t=0\}$ and Euclidean outward normal $\nu=-\pa_t$.
We use the same functional $E_g$ from \eqref{eq:common-functional-expanded} with $M=\R^N_+$.
In particular, a model bubble is a positive solution of
\[
\begin{cases}
-\Delta u=\dfrac{N(N-2)}4\sigma u^{(N+2)/(N-2)}
&\text{in }\R^N_+,\\[1mm]
-\pa_tu=\dfrac{N-2}{2}\eta u^{N/(N-2)}
&\text{on }\pa\R^N_+.
\end{cases}
\]
Its two bubble families are
\begin{align}
U_{\kappa,\xi,\eps}(x)
&:=\left(\frac{2\eps}{\eps^2+|\bar x-\xi|^2+(t+\kappa\eps)^2}
\right)^{(N-2)/2},\label{eq:common-TypeI-bubble}\\
U^{\mathrm{sf}}_{\xi,\eps}(x)
&:=\left(\frac{\eps}{(\eps+t)^2+|\bar x-\xi|^2}
\right)^{(N-2)/2}.\label{eq:common-TypeII-bubble}
\end{align}
Write $z:=(\xi,\eps)$, denote the bubble of the prescribed equation by $U_z$, and set $U_\kappa:=U_{\kappa,0,1}$.
The centered unit parameter is $z_0:=(0,1)$.

The energy space is
\begin{align*}
\Sigma:=\{w\in H^1_{\mathrm{loc}}(\overline{\R^N_+}):{}&
\nabla w\in L^2(\R^N_+),\quad
w\in L^{2N/(N-2)}(\R^N_+),\\
&w|_{\pa\R^N_+}\in L^{2(N-1)/(N-2)}(\pa\R^N_+)\},
\qquad \|w\|_\Sigma:=\|\nabla w\|_2.
\end{align*}
By the Sobolev and trace inequalities,
\begin{equation}\label{eq:interior-boundary-dual-bounds}
\begin{aligned}
\left|\int_{\R^N_+}f\psi\dd x\right|
&\le C\|f\|_{L^{2N/(N+2)}(\R^N_+)}\|\psi\|_\Sigma,\\
\left|\int_{\pa\R^N_+}b\psi\dd\bar x\right|
&\le C\|b\|_{L^{2(N-1)/N}(\pa\R^N_+)}\|\psi\|_\Sigma,
\qquad \psi\in\Sigma.
\end{aligned}
\end{equation}

With $z=(z_1,\ldots,z_N)=(\xi_1,\ldots,\xi_{N-1},\eps)$, define the parameter derivatives
\begin{equation*}
\Phi_{z,\alpha}:=\pa_{z_\alpha}U_z\quad(1\le\alpha\le N),\qquad
\mathcal K_z:=\operatorname{span}\{\Phi_{z,1},\ldots,\Phi_{z,N}\}.
\end{equation*}
Fix
\begin{equation}\label{eq:common-orthogonality-functionals}
\mathfrak l_{z,\alpha}(w):=
\begin{cases}
\displaystyle\int_{\R^N_+}U_z^{4/(N-2)}\Phi_{z,\alpha}w\dd x,
&\text{positive-scalar-curvature},\\[2mm]
\displaystyle\int_{\pa\R^N_+}U_z^{2/(N-2)}\Phi_{z,\alpha}w\dd\bar x,
&\text{scalar-flat}.
\end{cases}
\end{equation}
These are the interior $L^2$ constraints on the compactified cap and the boundary constraints of \cite[Section~2]{AlmarazBlowup}: the conformal volume and area weights, after division of both functions by $U_z$, are
$U_z^{4/(N-2)}$ and $U_z^{2/(N-2)}$, respectively.

Define
\begin{equation*}
\Sigma_z:=\bigcap_{\alpha=1}^N\ker\mathfrak l_{z,\alpha},\qquad
\mathbf G_z:=\bigl(\mathfrak l_{z,\alpha}(\Phi_{z,\beta})\bigr)_{\alpha,\beta=1}^N.
\end{equation*}
The space $\Sigma_z$ has the norm inherited from $\Sigma$.
A functional on $\Sigma$ is restricted to $\Sigma_z$ when it is an argument of $\cL_z^{-1}$.

The linearized bilinear form, or Jacobi form, is
$$
\cB_z(v,\psi):=\frac12D^2E_{g_{\mathrm{euc}}}(U_z)[v,\psi].
$$
In other words,
\begin{equation}\label{eq:common-Jacobi-form}
\begin{aligned}
\cB_z(v,\psi)
={}&\int_{\R^N_+}\nabla v\cdot\nabla\psi\dd x
-\frac{N(N+2)\sigma}{4}\int_{\R^N_+}U_z^{4/(N-2)}v\psi\dd x\\
&-\frac{N\eta}{2}\int_{\pa\R^N_+}U_z^{2/(N-2)}v\psi\dd\bar x.
\end{aligned}
\end{equation}
Differentiating the flat equation, we obtain $\cB_z(\Phi_{z,\alpha},\psi)=0$ for every $\psi\in\Sigma$.
The projected linearized operator on $\Sigma_z$ is
\begin{equation*}
\cL_z:\Sigma_z\to\Sigma_z^*,\qquad
(\cL_zv)(\psi):=\cB_z(v,\psi)\quad(v,\psi\in\Sigma_z).
\end{equation*}
\begin{proposition}
\label{prop:projected-inverses}
For either model bubble, the weighted Gram matrix $\mathbf G_z$ is positive definite.
The map
\begin{equation}\label{eq:explicit-Sigma-projection}
\Pi_zw:=w-\sum_{\alpha,\beta=1}^N\Phi_{z,\alpha}(\mathbf G_z^{-1})_{\alpha\beta}\mathfrak l_{z,\beta}(w)
\end{equation}
is a bounded projection onto $\Sigma_z$ with kernel $\mathcal K_z$, and
\begin{equation}\label{eq:energy-space-decomposition}
\Sigma=\Sigma_z\oplus\mathcal K_z.
\end{equation}

These projections depend smoothly on the bubble parameters, and $\Pi_z|_{\Sigma_{z_0}}$ is a local isomorphism onto $\Sigma_z$.
If the full kernel of $\cB_z$ is $\mathcal K_z$, then $\cL_z:\Sigma_z\to\Sigma_z^*$ is invertible.
Its inverse is uniformly bounded on compact parameter sets on which the weights vary smoothly and the full-kernel characterization holds.
This also applies to compact finite-curvature intervals for the Type-I family.
\end{proposition}
\begin{proof}
The functions $\Phi_{z,\alpha}$ are independent and $\mathbf G_z$ is their weighted Gram matrix with a positive weight.
Thus \eqref{eq:explicit-Sigma-projection} has the asserted image and kernel and proves \eqref{eq:energy-space-decomposition}.
The weights and projections depend smoothly on $z$, and $\mathbf G_z$ is uniformly positive on compact parameter sets.
Near $z_0$, both $\Pi_{z_0}\Pi_z|_{\Sigma_{z_0}}$ and $\Pi_z\Pi_{z_0}|_{\Sigma_z}$ are invertible perturbations of the identity.
Thus $\Pi_z|_{\Sigma_{z_0}}:\Sigma_{z_0}\to\Sigma_z$ is an isomorphism.

The potential terms are compact relative to the Dirichlet form: use compact interior and trace embeddings on bounded half-balls, and H\"older's inequality on the complements with $U_z^{4/(N-2)}\in L^{N/2}$ and
$U_z^{2/(N-2)}\in L^{N-1}(\pa\R^N_+)$.
The full Jacobi operator and its restriction to the finite-codimensional complement are Fredholm of index zero.
If the full kernel is $\mathcal K_z$, a restricted kernel element $v\in\Sigma_z$ annihilates all of $\Sigma$ by symmetry and \eqref{eq:energy-space-decomposition}; hence $v\in\mathcal K_z\cap\Sigma_z=\{0\}$.
This proves invertibility of $\cL_z$ without a positivity assumption.

Transport the operators by $\Pi_z|_{\Sigma_{z_0}}$ to a fixed space.
The smooth weights ensure operator-norm continuity there, so bounded inversion persists locally.
Covering a compact parameter set by finitely many such neighborhoods proves the uniform bound.
The same argument includes a finite curvature parameter whenever the stated kernel hypothesis holds.
\end{proof}

For the scalar-flat bubble \eqref{eq:common-TypeII-bubble}, the Jacobi form \eqref{eq:common-Jacobi-form} reduces to
\begin{equation*}
\cB_z(v,\psi)=\int_{\R^N_+}\nabla v\cdot\nabla\psi\dd x
-N\int_{\pa\R^N_+}U_z^{2/(N-2)}v\psi\dd\bar x.
\end{equation*}
With the scalar-flat constraints in \eqref{eq:common-orthogonality-functionals}, translation--dilation invariance and \cite[Proposition~2.7]{AlmarazBlowup} identify the full kernel as $\mathcal K_z$.
Proposition~\ref{prop:projected-inverses} therefore applies.
The conformal ball identity \eqref{a:eq:ballforms}, proved in Lemma~\ref{a:lem:angular}, also exhibits the negative direction of this form; invertibility on $\Sigma_z$ does not require positivity.

For the positive-scalar-curvature model, fix $\kappa\in\R$ and set $(\sigma,\eta)=(1,\kappa)$.
The equation on the Euclidean half-space is
\begin{equation}\label{eq:TI-boundary-yamabe}
\begin{cases}
-\Delta u=\dfrac{N(N-2)}4u^{\frac{N+2}{N-2}}
&\text{in }\R^N_+,\\[1mm]
-\pa_tu=\dfrac{N-2}{2}\kappa u^{\frac N{N-2}}
&\text{on }\pa\R^N_+.
\end{cases}
\end{equation}

For $-1<a<1$, let
\begin{equation*}
\mathcal C_a:=\{X\in\Sph^N:X_{N+1}\ge a\}.
\end{equation*}
The parameter $a$ and prescribed mean curvature $\kappa$ in \eqref{eq:TI-boundary-yamabe} are related by
\begin{equation*}
a=\frac{\kappa}{\sqrt{1+\kappa^2}},
\qquad
\kappa=\frac{a}{\sqrt{1-a^2}}.
\end{equation*}
For the round metric $g_a$ on $\mathcal C_a$,
\begin{equation*}
R_{g_a}=N(N-1),
\qquad
L_{g_a}=\kappa g_a^\top,
\qquad
H_{g_a}=\kappa.
\end{equation*}
In particular, $a=0$ gives the hemisphere, and $a$ has the sign of the prescribed mean curvature.

On $\R^N_+$ set
\begin{equation}
\Omega_\kappa(x)
:=\frac{2}{1+s+(t+\kappa)^2},
\qquad
g_\kappa:=\Omega_\kappa^2g_{\mathrm{euc}},
\qquad
U_\kappa=\Omega_\kappa^{\frac{N-2}{2}}.
\end{equation}

\begin{lemma}
The metric $g_\kappa$ is the pullback of the round metric on $\mathcal C_a$.
With respect to the outward normal, its curvatures satisfy
\[
R_{g_\kappa}=N(N-1),\qquad H_{g_\kappa}=\kappa.
\]
Equivalently, $U_\kappa$ solves \eqref{eq:TI-boundary-yamabe}.
\end{lemma}

\begin{proof}
The conformal factor is the standard stereographic factor translated by $-\kappa e_N$, so the interior metric is round.

To identify the portion of the sphere, set $q:=(\bar x,t+\kappa)$ and use stereographic coordinates
\[
X_i:=\frac{2q_i}{1+|q|^2}\ (1\le i\le N),\qquad
X_{N+1}:=\frac{1-|q|^2}{1+|q|^2}.
\]
The half-space condition $q_N\ge\kappa$ becomes $X_N-\kappa X_{N+1}\ge\kappa$.
The vector normal to this hyperplane has length $\sqrt{1+\kappa^2}$; rotating its unit normal to the last coordinate direction transforms the cap condition into $X_{N+1}\ge\kappa/\sqrt{1+\kappa^2}=a$.
On $t=0$ the Euclidean outward normal is $-\pa_t$ and
\[
\nu(\log\Omega_\kappa)
=\frac{2\kappa}{1+s+\kappa^2}
=\kappa\Omega_\kappa.
\]
By the conformal transformation law for averaged mean curvature, $H_{g_\kappa}=\kappa$.
\end{proof}

\begin{lemma}[Almaraz--Wang]
For every fixed finite $\kappa$, the full Jacobi form satisfies
\[
\ker\cB_z=\mathcal K_z
=\operatorname{span}\{\partial_{z_\alpha}U_z:1\le\alpha\le N\}.
\]
Consequently $\cL_z:\Sigma_z\to\Sigma_z^*$ is invertible, uniformly on compact bubble-parameter sets and compact finite-$\kappa$ intervals.
\end{lemma}
\begin{proof}
We use the nondegeneracy lemma of Almaraz--Wang \cite[Lemma~2.1]{AlmarazWang}, with their parameter $T=-\kappa$.
Its bubble is $2^{-(N-2)/2}U_\kappa$, so its linearized differential equation and Robin condition agree exactly with \eqref{eq:common-Jacobi-form}.
The lemma characterizes every decaying solution as a linear combination of the tangential translations and the dilation.
A finite-energy kernel function has the required decay: conformal transfer to the round cap produces an $H^1$ weak solution, the missing stereographic boundary point has zero $H^1$ capacity, and local elliptic regularity
makes the transformed function smooth at that point.
The Hardy and capacity estimates are detailed in the proof of Proposition~\ref{prop:curvature-completion}.
Transferring back, we obtain $O(|x|^{2-N})$ decay.
This proves the full-kernel hypothesis independently of the projected inverse.
Invertibility and the stated uniform bounds now follow from Proposition~\ref{prop:projected-inverses}.
\end{proof}

\section{The metric expansion and scalar correction}
\label{sec:polynomial-reduction}
We express the reduction of \cite{Brendle,BrendleMarques,AlmarazBlowup} as a second-order expansion in the metric and conformal factor.
The scalar correction is the stationary point of this quadratic expression on the Jacobi complement.

For a smooth tangential symmetric trace-free tensor $h$, fix $R>0$, put $h_R:=\chi_Rh$ and denote its components by $h_{R,ij}$.
For $\mu\in\R$ near zero, consider the metric path
\begin{equation}\label{eq:metric-variation-path}
\exp(\mu h_R)
=\dd t^2+
\sum_{a,b=1}^{N-1}[\exp(\mu h_R)]_{ab}\dd x_a\dd x_b.
\end{equation}
At $(g_{\mathrm{euc}},U_z)$, define the pure metric and mixed Taylor coefficients by
\begin{align*}
\mathcal Q[h_R](z)&:=\left.\frac12\pa_\mu^2
E_{\exp(\mu h_R)}(U_z)\right|_{\mu=0},\\
\Lambda_z[h_R](\psi)&:=\left.\frac12\pa_\mu
DE_{\exp(\mu h_R)}(U_z)[\psi]\right|_{\mu=0},
\qquad \psi\in\Sigma.
\end{align*}
The pure conformal coefficient is the Jacobi form $\cB_z(v,\psi)=\tfrac12D^2E_{g_{\mathrm{euc}}}(U_z)[v,\psi]$ from \eqref{eq:common-Jacobi-form}.

Let $v_z[h_R]\in\Sigma$ be the unique corrector for $U_z$ determined by
\begin{equation}\label{eq:projected-linear-corrector}
\begin{cases}
\cB_z(v_z[h_R],\psi)=-\Lambda_z[h_R](\psi)
&\text{for every }\psi\in\Sigma_z,\\
\mathfrak l_{z,\alpha}(v_z[h_R])=0&1\le\alpha\le N.
\end{cases}
\end{equation}
Equivalently,
\begin{equation}\label{eq:linear-corrector-inverse}
v_z[h_R]:=-\cL_z^{-1}\bigl(\Lambda_z[h_R]|_{\Sigma_z}\bigr).
\end{equation}
For fixed $R$ and $z$, set $F_{h_R}(z):=\mathcal Q[h_R](z)-\Lambda_z[h_R](\cL_z^{-1}\Lambda_z[h_R])$.
Proposition~\ref{prop:scalar-elimination} proves the expansion along the corrected path as $\mu\to0$:
\begin{equation}\label{eq:corrected-energy-Taylor}
E_{\exp(\mu h_R)}(U_z+\mu v_z[h_R])
=E_{g_{\mathrm{euc}}}(U_z)+\mu^2F_{h_R}(z)+o(\mu^2).
\end{equation}

\subsection{The energy expansion}
\begin{proposition}
\label{prop:scalar-elimination}
For fixed $R$ and $z$, with $w\in\Sigma_z$, the joint expansion as $(\mu,w)\to(0,0)$ in $\R\times\Sigma$ is
\begin{equation}\label{eq:metric-scalar-quadratic-expansion}
\begin{aligned}
E_{\exp(\mu h_R)}(U_z+w)
={}&E_{g_{\mathrm{euc}}}(U_z)+\mu^2\mathcal Q[h_R](z)
+2\mu\Lambda_z[h_R](w)+\cB_z(w,w)\\
&+o\bigl((|\mu|+\|w\|_\Sigma)^2\bigr).
\end{aligned}
\end{equation}

If $h$ is a tangential symmetric trace-free polynomial with $\deg h\le d_*$ and $N>2d_*+2$, the following cutoff limits exist:
\begin{align*}
\mathcal Q[h](z)&:=\lim_{R\to\infty}\mathcal Q[h_R](z),\\
\Lambda_z[h]&:=\lim_{R\to\infty}\Lambda_z[h_R]\quad\text{in }\Sigma^*,\\
F_h(z)&:=\lim_{R\to\infty}F_{h_R}(z).
\end{align*}
Set $v_z[h]:=-\cL_z^{-1}(\Lambda_z[h]|_{\Sigma_z})$.
The resulting coefficient is
\begin{equation}\label{eq:common-Schur-formula}
F_h(z)=\mathcal Q[h](z)
-\Lambda_z[h]\bigl(\cL_z^{-1}\Lambda_z[h]\bigr).
\end{equation}
The same formula holds for $h_R$.
These coefficients are the stationary values of the corresponding quadratic expressions on $\Sigma_z$.
Moreover, $F_{h_R}\to F_h$ locally in $C^2$ of $z$ as $R\to\infty$.
The function $F_h$ depends continuously in $C^2$ on the coefficients of $h$ and on auxiliary parameters for which the bubbles and projected inverses have this regularity.
\end{proposition}

We first compute the pure metric and mixed terms, including the estimates needed to pass to polynomial tensors.

\begin{proposition}\label{prop:metric-bubble-expansion}
For either model bubble $U_z$ and every cutoff tensor $h_R$ above,
\begin{equation}\label{eq:exact-linear-cancellation}
\left.\pa_\mu E_{\exp(\mu h_R)}(U_z)\right|_{\mu=0}=0.
\end{equation}
The pure metric quadratic coefficient is
\begin{equation}\label{eq:H-quadratic-cutoff-exact}
\begin{aligned}
\mathcal Q[h_R](z)={}&\frac12\int_{\R^N_+}
h_{R,ik}h_{R,jk}\pa_iU_z\pa_jU_z\dd x
-\frac{c_N}{4}\int_{\R^N_+}|\pa h_R|^2U_z^2\dd x\\
&+\frac{c_N}{2}\int_{\R^N_+}
(\pa_i h_{R,ik})(\pa_j h_{R,jk})U_z^2\dd x
+c_N\int_{\R^N_+}h_{R,ik}\pa_jh_{R,jk}
\pa_i(U_z^2)\dd x.
\end{aligned}
\end{equation}
For polynomial $h$, let $\mathcal Q[h](z)$ denote the same four integrals with $h_R$ replaced by $h$.
If $\deg h\le d_*$ and $N>2d_*+2$, they are finite and
\begin{equation}\label{eq:Fraw-general}
\mathcal Q[h_R](z)\to \mathcal Q[h](z)\qquad\text{as }R\to\infty.
\end{equation}
For fixed $R$,
\[
E_{\exp(\mu h_R)}(U_z)
=E_{g_{\mathrm{euc}}}(U_z)+\mu^2\mathcal Q[h_R](z)+O(|\mu|^3).
\]
The remainder and convergence are uniform on compact bubble-parameter sets.
When $\divg h=0$, the last two integrals vanish in $\mathcal Q[h]$, and their cutoff-generated terms in $\mathcal Q[h_R]$ tend to zero as $R\to\infty$.
\end{proposition}
\begin{proof}
For $|\mu|\|h_R\|_\infty\le1$, the determinant-one condition implies $\dd v_{\exp(\mu h_R)}=\dd x$, $\dd\sigma_{\exp(\mu h_R)}=\dd\bar x$, and $H_{\exp(\mu h_R)}=0$.
Thus, for fixed $u$,
\begin{align*}
\int_{\R^N_+}(u_+)^{2N/(N-2)}\dd v_{\exp(\mu h_R)}
&=\int_{\R^N_+}(u_+)^{2N/(N-2)}\dd x,\\
\int_{\pa\R^N_+}(u_+)^{2(N-1)/(N-2)}\dd\sigma_{\exp(\mu h_R)}
&=\int_{\pa\R^N_+}(u_+)^{2(N-1)/(N-2)}\dd\bar x.
\end{align*}
Only the gradient and scalar-curvature terms depend on the metric.

Expanding the Christoffel symbols for this determinant-one metric, we find
\begin{equation*}
R_{\exp(\mu h_R)}
=\mu\pa_i\pa_jh_{R,ij}+\mu^2\left[
-\pa_i\bigl(h_{R,ik}\pa_jh_{R,jk}\bigr)
+\frac12(\pa_i h_{R,ik})(\pa_j h_{R,jk})
-\frac14|\pa h_R|^2\right]+\mathscr E_3,
\end{equation*}
where the remainder satisfies the pointwise bound
\[
|\mathscr E_3|
\le C|\mu|^3
\bigl(|h_R|^2|\pa^2h_R|+|h_R||\pa h_R|^2\bigr).
\]
Together with the inverse-metric expansion
\[
[\exp(-\mu h_R)]_{ij}=\delta_{ij}-\mu h_{R,ij}
+\tfrac12\mu^2h_{R,ik}h_{R,jk}+O(|\mu|^3|h_R|^3),
\]
this is the metric expansion used in the reductions of \cite[Section~3]{AlmarazBlowup} and \cite{Brendle,BrendleMarques}.

For either bubble $U_z$ in \eqref{eq:common-TypeI-bubble}--\eqref{eq:common-TypeII-bubble}, with $z$ fixed, the first variation along \eqref{eq:metric-variation-path} cancels exactly.
Integrating twice by parts in the tangential variables, we obtain
\[
\left.\pa_\mu E_{\exp(\mu h_R)}(U_z)\right|_{\mu=0}
=-\int_{\R^N_+}h_{R,ab}\pa_aU_z\pa_bU_z\dd x
+c_N\int_{\R^N_+}h_{R,ab}\pa_{ab}(U_z^2)\dd x.
\]
Since $\tr h_R=0$,
\[
h_{R,ab}\pa_{ab}U_z=\frac{N}{N-2}U_z^{-1}h_{R,ab}\pa_aU_z\pa_bU_z,
\qquad -1+2c_N+\frac{2Nc_N}{N-2}=0.
\]
This proves \eqref{eq:exact-linear-cancellation}.

The second-order terms are those in \eqref{eq:H-quadratic-cutoff-exact}.
The integration by parts in the divergence term has no flux at infinity because $h_R$ is compactly supported, and no flux on $t=0$ because $h_{R,Nk}=0$.
For $h_R$, the polynomial tail and cutoff terms are bounded by
\[
C R^{2d_*+2-N}\to0\qquad\text{as }R\to\infty.
\]
This proves the limit \eqref{eq:Fraw-general}.
\end{proof}

\begin{lemma}
\label{lem:source-kernel-orthogonality}
Let $h$ be a tangential symmetric trace-free polynomial with $\deg h\le d_*$ and $N>2d_*+2$.
Set
\[
S_z[h]:=h_{ij}\pa_{ij}U_z+(\pa_jh_{ij})\pa_iU_z
 +c_N(\pa_i\pa_jh_{ij})U_z.
\]
The mixed coefficient extends to a functional $\Lambda_z[h]\in\Sigma^*$ given, for every $\psi\in\Sigma$, by
\begin{equation}\label{eq:polynomial-source-functional}
\Lambda_z[h](\psi)
=-\int_{\R^N_+}h_{ij}\pa_iU_z\pa_j\psi\dd x
+c_N\int_{\R^N_+}(\pa_i\pa_jh_{ij})U_z\psi\dd x
=\int_{\R^N_+}S_z[h]\psi\dd x.
\end{equation}
It annihilates the Jacobi kernel:
\begin{equation}\label{eq:polynomial-source-kernel-orthogonality}
\Lambda_z[h](\Phi_{z,\alpha})=0\qquad(1\le\alpha\le N).
\end{equation}
Consequently the projected corrector satisfies the equation on the full space:
\begin{equation}\label{eq:full-linear-corrector}
\cB_z(v_z[h],\psi)=-\Lambda_z[h](\psi)
\qquad\text{for every }\psi\in\Sigma.
\end{equation}
The cutoff sources converge in $\Sigma^*$ with the rate in \eqref{eq:cutoff-source-dual-convergence}, uniformly together with their first two parameter derivatives on compact sets.
The source formula, kernel orthogonality, and full-test equation also hold for smooth compactly supported $h_R$.
\end{lemma}
\begin{proof}
Differentiating at fixed $(h_R,z,\psi)$ and integrating by parts in the tangential variables, we obtain
\begin{equation}\label{eq:linear-metric-residual-components}
\begin{aligned}
\Lambda_z[h_R](\psi)
&=-\int_{\R^N_+} h_{R,ij}\pa_iU_z\pa_j\psi\dd x
+c_N\int_{\R^N_+}(\pa_i\pa_jh_{R,ij})U_z\psi\dd x\\
&=\int_{\R^N_+}\bigl[h_{R,ij}\pa_{ij}U_z
+(\pa_jh_{R,ij})\pa_iU_z
+c_N(\pa_i\pa_jh_{R,ij})U_z\bigr]\psi\dd x.
\end{aligned}
\end{equation}
There is no boundary term because $h_{R,iN}=0$, and no term at infinity because $h_R$ is compactly supported.

All three density terms have the bound $C(1+|x|)^{d_*-N}$, uniformly on compact parameter sets.
Thus $\Lambda_z[h]\in\Sigma^*$ by \eqref{eq:interior-boundary-dual-bounds}.
Fix $R$ and differentiate \eqref{eq:exact-linear-cancellation} on the full space, keeping $h_R$ and $R$ fixed:
\begin{equation}\label{eq:metric-variation-kernel-chain-rule}
0
=\pa_{z_\alpha}\left[
\left.\pa_\mu E_{\exp(\mu h_R)}(U_z)
\right|_{\mu=0}\right]
=\left.\pa_\mu
\bigl(DE_{\exp(\mu h_R)}(U_z)[\pa_{z_\alpha}U_z]\bigr)
\right|_{\mu=0}
=2\Lambda_z[h_R](\Phi_{z,\alpha}).
\end{equation}
Compact support and smooth parameter dependence justify interchange of the derivatives.

For the cutoff limit, we use symmetry to write the full product identities
\begin{align*}
\pa_jh_{R,ij}
&=\chi_R\pa_jh_{ij}+(\pa_j\chi_R)h_{ij},\\
\pa_i\pa_jh_{R,ij}
&=\chi_R\pa_i\pa_jh_{ij}
+2(\pa_i\chi_R)(\pa_jh_{ij})
+(\pa_i\pa_j\chi_R)h_{ij}.
\end{align*}
On $R<|x|<2R$, derivatives of $\chi_R$ of orders one and two are bounded by $CR^{-1}$ and $CR^{-2}$.
By subtracting \eqref{eq:polynomial-source-functional} from \eqref{eq:linear-metric-residual-components}, we obtain, for every $\psi\in\Sigma$,
\begin{equation}\label{eq:cutoff-source-difference}
\begin{aligned}
&(\Lambda_z[h_R]-\Lambda_z[h])(\psi)\\
&\quad=\int_{\R^N_+}\Bigl[
(\chi_R-1)S_z[h]
+(\pa_j\chi_R)h_{ij}\pa_iU_z
+c_N\{2(\pa_i\chi_R)(\pa_jh_{ij})
+(\pa_i\pa_j\chi_R)h_{ij}\}U_z
\Bigr]\psi\dd x.
\end{aligned}
\end{equation}
Each term in brackets is bounded in absolute value by $C(1+|x|)^{d_*-N}\mathbf1_{\{|x|>R\}}$.
For $p:=2N/(N+2)$, radial integration shows that
\[
\bigl\|(1+|x|)^{d_*-N}\mathbf1_{\{|x|>R\}}\bigr\|_{L^p}
\le CR^{d_*-N+N/p}=CR^{-(N-2d_*-2)/2}\to0\qquad\text{as }R\to\infty.
\]
Applying \eqref{eq:interior-boundary-dual-bounds} to \eqref{eq:cutoff-source-difference} proves
\begin{equation}\label{eq:cutoff-source-dual-convergence}
\|\Lambda_z[h_R]-\Lambda_z[h]\|_{\Sigma^*}
\le CR^{-(N-2d_*-2)/2}\to0\qquad\text{as }R\to\infty,
\end{equation}
uniformly for $z$ in a compact set.
The same bound holds after one or two derivatives with respect to $z$.
Indeed, for $|\alpha|\le2$ and $0\le j\le2$, direct differentiation of either bubble shows that
\[
|\pa_z^\alpha\pa_x^jU_z(x)|\le C(1+|x|)^{2-N-j}.
\]
Together with $|\pa^jh|\le C(1+|x|)^{d_*-j}$ this bounds every differentiated summand of \eqref{eq:cutoff-source-difference} by the same integrable tail.
Since $\|\Phi_{z,\alpha}\|_\Sigma$ is uniformly bounded on compact parameter sets, we deduce $\Lambda_z[h](\Phi_{z,\alpha})=0$ from \eqref{eq:metric-variation-kernel-chain-rule} and
\eqref{eq:cutoff-source-dual-convergence}, as asserted in \eqref{eq:polynomial-source-kernel-orthogonality}.

We next prove \eqref{eq:full-linear-corrector} from the projected linear equation \eqref{eq:projected-linear-corrector}.
For an arbitrary $\psi\in\Sigma$, we use the projection \eqref{eq:explicit-Sigma-projection} to write
\begin{equation*}
\psi=\Pi_z\psi+
\sum_{\alpha,\beta=1}^N \Phi_{z,\alpha}(\mathbf G_z^{-1})_{\alpha\beta}\mathfrak l_{z,\beta}(\psi).
\end{equation*}
The corrector equation applies to $\Pi_z\psi$.
Both $\cB_z(v_z[h],\Phi_{z,\alpha})$ and $\Lambda_z[h](\Phi_{z,\alpha})$ vanish, by the kernel identity and the orthogonality just proved.
Hence \eqref{eq:full-linear-corrector} holds.
\end{proof}

\begin{proof}[Proof of Proposition~\ref{prop:scalar-elimination}]
For fixed $R$, the map $(\mu,w)\mapsto E_{\exp(\mu h_R)}(U_z+w)$ is $C^2$ on a neighborhood of $(0,0)$ in $\R\times\Sigma$: the metric coefficients are smooth in $\mu$, and both nonlinear energy exponents exceed two.
Its linear terms vanish because $U_z$ is critical and \eqref{eq:exact-linear-cancellation} holds.
The definitions of $\mathcal Q[h_R]$, $\Lambda_z[h_R]$ and $\cB_z$ identify its second derivatives, proving \eqref{eq:metric-scalar-quadratic-expansion}.

Substitute $w=\mu v$.
For fixed $h_R$, the quadratic coefficient is $\mathcal Q[h_R](z)+2\Lambda_z[h_R](v)+\cB_z(v,v)$.
Its derivative in a direction $\psi\in\Sigma_z$ is $2\Lambda_z[h_R](\psi)+2\cB_z(v,\psi)$, so its unique stationary point is \eqref{eq:projected-linear-corrector} by invertibility of $\cL_z$.
This proves the Taylor identity \eqref{eq:corrected-energy-Taylor}.

For polynomial $h$, the limits of $\mathcal Q[h_R]$ and $\Lambda_z[h_R]$ as $R\to\infty$ follow from Proposition~\ref{prop:metric-bubble-expansion} and Lemma~\ref{lem:source-kernel-orthogonality}.
By bounded inversion, $v_z[h_R]\to v_z[h]$ in $\Sigma$ as $R\to\infty$.
The limiting quadratic coefficient, for $v\in\Sigma_z$, is
\begin{equation}
\mathcal Q[h](z)+2\Lambda_z[h](v)+\cB_z(v,v).
\end{equation}
Its stationary point is $v_z[h]$.
Testing the corrector equation yields
\begin{align*}
\cB_z(v_z[h],v_z[h])&=-\Lambda_z[h](v_z[h]),\\
2\Lambda_z[h](v_z[h])+\cB_z(v_z[h],v_z[h])
&=-\Lambda_z[h](\cL_z^{-1}\Lambda_z[h]).
\end{align*}
Thus the cutoff limit defining $F_h$ exists and equals \eqref{eq:common-Schur-formula}; the same calculation applies directly to $h_R$.
This is a stationary value and requires invertibility, not positivity, of $\cL_z$.

It remains to prove $C^2$ convergence and parameter dependence.
The four integrands defining $\mathcal Q[h]$ and their first two parameter derivatives have integrable majorant $C(1+|x|)^{2d_*+2-2N}$.
Their cutoff tails, including derivatives of the cutoff, are $O(R^{2d_*+2-N})$.
By Lemma~\ref{lem:source-kernel-orthogonality}, the source converges in $\Sigma^*$ through two parameter derivatives as $R\to\infty$.
To differentiate $\cL_z^{-1}$ on its varying domain, fix $z_0$ and use
\[
T_z:=\Pi_z|_{\Sigma_{z_0}}:\Sigma_{z_0}\to\Sigma_z.
\]
For $z$ near $z_0$, $T_z$ is an isomorphism by the projection argument after \eqref{eq:energy-space-decomposition}.
Define $A_z:\Sigma_{z_0}\to\Sigma_{z_0}^*$ by
\[
(A_zv)(\psi):=\cB_z(T_zv,T_z\psi),
\qquad v,\psi\in\Sigma_{z_0}.
\]
Thus $A_z$ represents the prescribed family $\cL_z$ on $\Sigma_{z_0}$.
The smooth dependence of $\Pi_z$ and the weights in \eqref{eq:common-Jacobi-form} makes $A_z$ a $C^2$ family in operator norm.
Since $A_{z_0}$ is invertible, a Neumann series bounds $A_z^{-1}$ uniformly near $z_0$, and
\begin{align*}
A_z^{-1}-A_{z_0}^{-1}
&=A_z^{-1}(A_{z_0}-A_z)A_{z_0}^{-1},\\
\pa_{z_\alpha}A_z^{-1}
&=-A_z^{-1}(\pa_{z_\alpha}A_z)A_z^{-1}.
\end{align*}
Differentiating the second identity proves that $A_z^{-1}$ depends $C^2$-smoothly on $z$.
In these coordinates,
\[
v_z[h]=-T_zA_z^{-1}
\bigl(\psi\mapsto\Lambda_z[h](T_z\psi)\bigr).
\]
The same formula with $h_R$ in place of $h$ and the source estimates imply $v_z[h_R]\to v_z[h]$ locally in $C^2$ of $z$ as $R\to\infty$.
Pairing with the source in \eqref{eq:common-Schur-formula}, we conclude that $F_{h_R}\to F_h$ locally in $C^2$ as $R\to\infty$.
These estimates also prove continuous $C^2$ dependence on the coefficients of $h$ and on the auxiliary parameters in the statement.
\end{proof}

\subsection{Geometric interpretation}
The corrector equation against all test functions has the following geometric interpretation.
For a smooth compactly supported tangential trace-free tensor $h_R$, write $g_z:=U_z^{4/(N-2)}g_{\mathrm{euc}}$ and $v:=v_z[h_R]$.
The half-space metric $g_z$ extends, under the bubble compactification, to the compact cap or ball just described.

\begin{proposition}\label{prop:curvature-completion}
The tensor
\begin{equation}\label{eq:completed-geometric-variation}
h_{R,z}:=U_z^{4/(N-2)}h_R+\frac4{N-2}\frac v{U_z}g_z
\end{equation}
extends smoothly to the compact model and satisfies
\begin{equation}\label{eq:completed-curvature-constraints}
DR_{g_z}(h_{R,z})=0,\qquad DH_{g_z}(h_{R,z})=0.
\end{equation}
It is the first derivative of the positive conformal metric path
\[
\hat{g}_{\mu,R}:=(U_z+\mu v)^{4/(N-2)}\exp(\mu h_R)
\]
for sufficiently small $|\mu|$, with $R$ fixed.
\end{proposition}
\begin{proof}
Kernel orthogonality and the full-test argument in Lemma~\ref{lem:source-kernel-orthogonality} apply directly to compact support.
There are consequently no undetermined kernel multipliers in the interior or boundary equation for $v$.
Set
\begin{align*}
\mathscr P_g(u)&:=-\Delta_gu+c_NR_gu
-\frac{N(N-2)\sigma}{4}u^{(N+2)/(N-2)},\\
\mathscr T_g(u)&:=\pa_{\nu_g}u+\frac{N-2}{2}H_gu
-\frac{N-2}{2}\eta u^{N/(N-2)}.
\end{align*}
At the Euclidean bubble their scalar linearizations are
\[
J_zv:=-\Delta v-\frac{N(N+2)\sigma}{4}U_z^{4/(N-2)}v,
\qquad
B_zv:=\pa_\nu v-\frac{N\eta}{2}U_z^{2/(N-2)}v.
\]
Their metric derivatives are respectively the full source $S_z[h_R]$ in \eqref{eq:linear-metric-residual-components} and zero: the normal and the averaged mean curvature do not change along the raw metric path.
Hence the full-test equation is exactly
\[
J_zv+S_z[h_R]=0\quad\text{in }\R^N_+,\qquad
B_zv=0\quad\text{on }\pa\R^N_+.
\]
For $\hat{g}=u^{4/(N-2)}g$, the conformal laws read
\begin{align*}
\mathscr P_g(u)&=c_Nu^{(N+2)/(N-2)}
\bigl(R_{\hat{g}}-N(N-1)\sigma\bigr),\\
\mathscr T_g(u)&=\frac{N-2}{2}u^{N/(N-2)}(H_{\hat{g}}-\eta).
\end{align*}
Differentiating at $(g,u)=(g_{\mathrm{euc}},U_z)$, we obtain
\begin{align*}
DR_{g_z}(h_{R,z})
&=c_N^{-1}U_z^{-(N+2)/(N-2)}(J_zv+S_z[h_R])=0,\\
DH_{g_z}(h_{R,z})
&=\frac2{N-2}U_z^{-N/(N-2)}B_zv=0.
\end{align*}

To justify smoothness, put $\varphi:=v/U_z$.
By conformal transport,
\[
\int\varphi^2\dd v_{g_z}=\int U_z^{4/(N-2)}v^2\dd x,\qquad
\int|\nabla\varphi|_{g_z}^2\dd v_{g_z}
=\int|\nabla v-v\nabla\log U_z|^2\dd x.
\]
Since $U_z^{4/(N-2)}\le C(1+|x|)^{-4}$ and $|\nabla\log U_z|\le C(1+|x|)^{-1}$, Hardy's inequality bounds both integrals by $C\|v\|_\Sigma^2$.
Thus $\varphi\in H^1$ of the compact model.
The transformed forcing is smooth and vanishes near the omitted boundary point, since $h_R$ does.
The full weak equation transports to a smooth elliptic equation with its Robin boundary condition.
A point has zero $H^1$ capacity for $N>2$: cutoffs vanishing in a ball of radius $r$ can have squared gradient integral $O(r^{N-2})$.
Using such cutoffs in the weak equation and then density extends it across the omitted point.
Interior and Robin boundary regularity imply $v/U_z\in C^\infty$ on the entire model.
The metric term $U_z^{4/(N-2)}h_R$ is also smooth, and vanishes near that point.
Finally $v/U_z$ is bounded, so $1+\mu v/U_z>0$ for sufficiently small $|\mu|$.
\end{proof}

\begin{proposition}
For $h_R$ and its completed variation $h_{R,z}$ in Proposition~\ref{prop:curvature-completion}, let $F_{h_R}$ be the quadratic coefficient in \eqref{eq:corrected-energy-Taylor}.
Then
\begin{equation}\label{eq:geometric-reduced-hessian}
F_{h_R}(z)=\frac{c_N}{2}D^2\mathcal A_{\sigma,\eta}(g_z)
[h_{R,z},h_{R,z}].
\end{equation}
The same identity holds after polarization.
For a polynomial $h$ of degree at most $d_*$ with $N>2d_*+2$, its meaning is the specified cutoff limit
\begin{equation}\label{eq:geometric-hessian-cutoff-limit}
F_h(z)=\lim_{R\to\infty}\frac{c_N}{2}
D^2\mathcal A_{\sigma,\eta}(g_z)
[h_{R,z},h_{R,z}],
\end{equation}
locally in $C^2$ of the bubble parameters, with $h_{R,z}$ as defined in \eqref{eq:completed-geometric-variation}.
\end{proposition}
\begin{proof}
Substitute $w=\mu v_z[h_R]$ in \eqref{eq:metric-scalar-quadratic-expansion}.
Its quadratic coefficient is $F_{h_R}$.
By \eqref{eq:energy-geometric-action} the same expansion equals $c_N\mathcal A_{\sigma,\eta}(\hat{g}_{\mu,R})$.
By the chain rule, its second derivative is
\[
c_ND^2\mathcal A(g_z)[h_{R,z},h_{R,z}]
+c_ND\mathcal A(g_z)[\hat{g}''_{0,R}].
\]
The last term vanishes by model stationarity, and the Taylor coefficient is one half of this second derivative.
This proves \eqref{eq:geometric-reduced-hessian}; the corrector is linear in $h_R$, so polarization follows as well.

By Proposition~\ref{prop:scalar-elimination}, $F_{h_R}\to F_h$ locally in $C^2$ as $R\to\infty$.
Passing to the limit in \eqref{eq:geometric-reduced-hessian} proves \eqref{eq:geometric-hessian-cutoff-limit}.
\end{proof}

\begin{corollary}
\label{cor:boundary-gauge-directions}
For every smooth symmetric tensor $h$ on the compact model and every smooth vector field $X$ tangent to the boundary, one has
\[
D^2\mathcal A_{\sigma,\eta}(g_z)[\mathcal L_Xg_z,h]=0.
\]

The reduced value is also independent of the scalar complement when the two correctors solve the same full linear equation.
\end{corollary}
\begin{proof}
By boundary-preserving diffeomorphism invariance, $D\mathcal A(\hat{g})[\mathcal L_X\hat{g}]=0$.
Differentiate in the direction $h$; the additional term $D\mathcal A(g_z)[\mathcal L_Xh]$ vanishes by stationarity.
Thus boundary-preserving Lie derivatives are null directions of the Hessian.
Changing the scalar complement changes a solution of the same full equation by a Jacobi field.
That field pairs to zero with $\Lambda_z[h]$, so the reduced value is unchanged.
\end{proof}

\begin{remark}
Corollary~\ref{cor:boundary-gauge-directions} applies to smooth boundary-preserving diffeomorphisms; an arbitrary change of polynomial representative must also respect the boundary conditions and the cutoff limit \eqref{eq:geometric-hessian-cutoff-limit}.
Complement independence concerns the reduced value, not the coercivity constant.
The curvature constraints \eqref{eq:completed-curvature-constraints} impose neither zero trace nor zero divergence on $h_{R,z}$.
\end{remark}

\section{Polynomial tensors and parameter derivatives}\label{sec:algebra}
We first describe the polynomial tensors and the four scalar-flat and minimal-boundary profiles.
We then establish their source, curvature, and angular identities.
The corrected matrix $\mathbf M$ in \eqref{eq:coefficient-matrix-entries} determines the centered energy and its scale derivatives through \eqref{a:eq:scalingmatrix}.
The coefficients $\alpha(\eps),\gamma(\eps)$ in \eqref{eq:explicit-alpha-gamma} determine the translation Hessian.
The translation gradient and the mixed translation--scale derivatives vanish at the center.
This section derives the formulas for these quantities; Section~\ref{sec:certification} evaluates them and proves the signs needed for the four minima.

\subsection{Scalar tensors and metric profiles}
Fix $A\in\operatorname{Sym}^2_0(\R^m)$ and use $Q_A(\bar x)=\bar x^{\mathsf T}A\bar x$ from the introduction.
Then $\Delta_{\bar x}Q_A=0$ and $A=\tfrac12\nabla_{\bar x}^2Q_A$.
For $k\ge2$ define the two-parameter family
\begin{equation}\label{eq:full-scalar-module}
\begin{split}
S_N^{(k)}(b,e):=\operatorname{tf}_m\bigl[
s^kA&+b s^{k-1}(\bar x\otimes A\bar x+A\bar x\otimes\bar x)\\
&+e s^{k-2}Q_A\bar x\otimes\bar x\bigr],
\end{split}
\end{equation}
where $\operatorname{tf}_mT:=T-\frac{\tr T}{m}I_m$ denotes the tangential trace-free part.
In particular, its trace correction is $-(2b+e)s^{k-1}Q_AI_m/m$.
For $k=1$ require $e=0$ and omit the last term; for $k=0$ set $S_N^{(0)}(0,0):=A$.
The generator $A$ is suppressed in $S_N^{(k)}$ and $P_N^{(k)}$; brackets are used only when a different generator is needed.

The coefficient of $s^kA$ is normalized to one, while $b,e$ are free.
The linear span of these tensors is generated by $Q_A$ and its first two derivatives; no divergence or radial-transversality condition is imposed.
Our scalar-flat and minimal-boundary profiles are defined by
\begin{equation}\label{eq:full-scalar-profile}
h(\bar x,t):=\sum_{(k,p)\in\mathcal I}
c_{kp}t^pS_N^{(k)}(b_{kp},e_{kp}),\qquad 2k+p\le d_*.
\end{equation}
Here $\mathcal I$ is a finite set of pairs of nonnegative integers, $c_{kp}$ is the overall coefficient of the row, and $b_{kp},e_{kp}$ are its two relative coefficients.
All nonzero rows used below admit this normalization.
We take $b_{0p}=e_{0p}=0$ and $e_{1p}=0$; unlisted rows have $c_{kp}=0$ and contribute nothing.
The normal powers are independent of the tangential family: the $p=0$ and $p=1$ rows determine the boundary value and first normal derivative, respectively.
After setting the normal components to zero, we have
\begin{equation}\label{eq:admissible-polynomial}
h_{Ni}=0,\qquad \tr h=0,\qquad \deg h\le d_*.
\end{equation}

\begin{lemma}\label{lem:scalar-divergence}
For $k\ge2$ set
\[
u_k:=2k+(m+2k)b-\frac{2(2b+e)}m,\qquad
v_k:=2(k-1)\left(b-\frac{2b+e}m\right)+(m+2k-1)e.
\]
Then
\begin{equation}
\begin{aligned}
\divg S_N^{(k)}(b,e)&=u_ks^{k-1}A\bar x+v_ks^{k-2}Q_A\bar x,\\
\divg\divg S_N^{(k)}(b,e)
&=[2(k-1)u_k+(m+2k-2)v_k]s^{k-2}Q_A.
\end{aligned}
\end{equation}
At $k=1$ use the same $u_1$ with $e=0$; then $\divg S_N^{(1)}=u_1A\bar x$ and $\divg\divg S_N^{(1)}=0$.
At $k=0$ both divergences vanish.
\end{lemma}
\begin{proof}
Expand the trace-free projection in \eqref{eq:full-scalar-module} and differentiate, using $\nabla s=2\bar x$, $\nabla Q_A=2A\bar x$, and $\tr A=0$.
The coefficients of $s^{k-1}A\bar x$ and $s^{k-2}Q_A\bar x$ are $u_k$ and $v_k$.
A second differentiation uses
\[
\divg(s^{k-1}A\bar x)=2(k-1)s^{k-2}Q_A,\qquad
\divg(s^{k-2}Q_A\bar x)=(m+2k-2)s^{k-2}Q_A.
\]
The low-degree cases follow directly from their definitions.
\end{proof}

The following choices of $b,e$ make $S_N^{(k)}(b,e)$ divergence free.
For $k\ge1$ put
\begin{equation}\label{a:eq:Achain}
\begin{gathered}
P_N^{(k)}:=S_N^{(k)}(b_k,e_k),\qquad
\Delta_k:=(m-1)(m+2k)^2-2m,\\
b_k:=-\frac{2k((m-1)(m+2k)+2)}{\Delta_k},\qquad
e_k:=\frac{4k(k-1)(m-2)}{\Delta_k}.
\end{gathered}
\end{equation}
Set $b_0:=0$, $e_0:=0$, and $P_N^{(0)}:=A$; the formula also implies $e_1=0$.
For $b_{kp}=b_k$ and $e_{kp}=e_k$, \eqref{eq:full-scalar-profile} reduces to $h=\sum c_{kp}t^pP_N^{(k)}$.

\begin{lemma}\label{a:lem:Achain}
Each $P_N^{(k)}$ is tangential, trace-free, divergence-free, and homogeneous of degree $2k$.
Set
\[
\mathsf s_k:=1+2b_k+e_k-\frac{2b_k+e_k}m,\qquad
\vartheta_k:=2k(2k+m-2)+4b_k.
\]
Then
\begin{equation}\label{a:eq:Acontractions}
\bar x^{\mathsf T}P_N^{(k)}\bar x=\mathsf s_k s^kQ_A,\qquad
\Delta_{\bar x}P_N^{(k)}=\vartheta_kP_N^{(k-1)}\quad(k\ge1).
\end{equation}
\end{lemma}
\begin{proof}
Trace freeness follows from the projection in \eqref{eq:full-scalar-module}.
By substituting \eqref{a:eq:Achain} into Lemma~\ref{lem:scalar-divergence}, we find
\[
2k+(m+2k)b_k-\frac{2(2b_k+e_k)}m=0,\qquad
2(k-1)\left(b_k-\frac{2b_k+e_k}m\right)+(m+2k-1)e_k=0.
\]
These two equations determine $b_k,e_k$ uniquely when $m>2$; at $k=1$ the second coefficient is zero.
The radial contraction follows directly from the trace-free formula.
The Laplacian has the same equivariant form in degree $2k-2$ and remains trace-free and divergence-free.
Its coefficient of $s^{k-1}A$ is $\vartheta_k$, so the same two equations identify it with $\vartheta_kP_N^{(k-1)}$.
The case $k=1$ follows directly.
\end{proof}

\begin{lemma}
Let $m>2$, $A\ne0$, and let $h:=\sum c_{kp}t^pP_N^{(k)}$ be nonzero.
Then
\[
\Delta_{\bar x}Q_A=0,\qquad A=\tfrac12\nabla_{\bar x}^2Q_A,
\qquad x^{\mathsf T}hx
=Q_A\sum_{k,p}c_{kp}\mathsf s_k s^kt^p\not\equiv0.
\]
In particular, $hx\not\equiv0$, and its centered scalar source in \eqref{a:eq:Asource} is nonzero.
\end{lemma}
\begin{proof}
The first two identities follow from $\tr A=0$.
By \eqref{a:eq:Achain},
\[
\mathsf s_0=1,\qquad
\mathsf s_k=\frac{m(m-2)(m+1)}{(m-1)(m+2k)^2-2m}>0\quad(k\ge1).
\]
The radial contraction follows from Lemma~\ref{a:lem:Achain}.
The monomials $s^kt^p$ are linearly independent, and $Q_A\not\equiv0$.
By \eqref{a:eq:Asource}, the source is also nonzero.
\end{proof}

\begin{lemma}
For $k\ge2$,
\begin{equation}
S_N^{(k)}(b,e)\bar x
=(1+b)s^kA\bar x+
\frac{(m-2)b+(m-1)e}{m}s^{k-1}Q_A\bar x.
\end{equation}
In particular, $S_N^{(k)}(b_{k0},e_{k0})\bar x=0$ for
\begin{equation}\label{eq:longitudinal-boundary-rule}
b_{k0}=-1,\qquad e_{k0}=\frac{m-2}{m-1},\qquad k\ge2.
\end{equation}
Its overall coefficient $c_{k0}$ remains free.
\end{lemma}
\begin{proof}
Contract \eqref{eq:full-scalar-module} with $\bar x$ and use the trace correction $-(2b+e)s^{k-1}Q_AI_m/m$.
Substitution of \eqref{eq:longitudinal-boundary-rule} makes both coefficients vanish.
\end{proof}

These tensors are written in Gaussian normal coordinates: $g_{aN}=0$ and $g_{NN}=1$ for $g=\exp(h)$, but the tangential coordinates on the boundary are not assumed geodesic.
\begin{lemma}\label{lem:polynomial-boundary-geometry}
Let $h$ satisfy \eqref{eq:admissible-polynomial}.
Its exponential metric satisfies
\[
g=\dd t^2+[\exp(h)]_{ab}\dd x_a\dd x_b,\qquad H_g=0.
\]
Moreover,
\begin{equation}\label{eq:umbilic-polynomial-condition}
\pa_t h(\bar x,0)=0\quad\Leftrightarrow\quad
\left.\pa_t\exp(h(\bar x,t))\right|_{t=0}=0\quad\Leftrightarrow\quad
\pa\R^N_+\text{ is totally geodesic}.
\end{equation}
The same equivalence holds for $\exp(\mu h)$ whenever $\mu\ne0$.
\end{lemma}
\begin{proof}
In these product coordinates, the averaged mean curvature is
\[
H_g=-\frac1{2m}\partial_t\log\det(\exp(h))
=-\frac1{2m}\partial_t\tr h=0.
\]
For symmetric matrices $B,C$,
\[
\left.\frac{\dd}{\dd u}\exp(B+uC)\right|_{u=0}
=\int_0^1\exp((1-r)B)C\exp(rB)\dd r.
\]
In a basis diagonalizing $B$, this map multiplies each entry of $C$ by a positive number and is therefore injective.
Since the second fundamental form is $-\tfrac12\partial_tg$, this proves the equivalence, also after multiplication of $h$ by a nonzero constant.
\end{proof}

We now specify $A$ and the coefficients in \eqref{eq:full-scalar-profile} for each equation and umbilicity condition.
All generating matrices have trace zero and norm one.
The two nonumbilic profiles use
\begin{equation}\label{a:eq:normalizedA}
A_{\mathrm r}:=\frac1{\sqrt2}\diag(1,-1,0,\ldots,0).
\end{equation}
The scalar-flat umbilic profile uses the orthogonally equivalent matrix
\begin{equation}
A_{\mathrm c}:=\frac1{\sqrt2}\diag(0,1,-1,0,\ldots,0).
\end{equation}
For the minimal umbilic profile, put $m_+:=\lfloor m/2\rfloor$, $m_-:=m-m_+$, and use
\begin{equation}
A_{\mathrm b}:=\frac1{\sqrt m}\diag\left(
\underbrace{\sqrt{\frac{m_-}{m_+}},\ldots,\sqrt{\frac{m_-}{m_+}}}_{m_+\text{ entries}},
\underbrace{-\sqrt{\frac{m_+}{m_-}},\ldots,-\sqrt{\frac{m_+}{m_-}}}_{m_-\text{ entries}}\right).
\end{equation}
At $N=21$, this is $\diag(I_{10},-I_{10})/\sqrt{20}$.
In each family $A$ denotes its assigned generating matrix, with the fixed dimension suppressed.

For both nonumbilic families take $A=A_{\mathrm r}$ and $(b_{kp},e_{kp})=(b_k,e_k)$ from \eqref{a:eq:Achain}.
The overall coefficients $c^{\mathrm{II},15}_{kp}$ and $c^{\mathrm I,15}_{kp}$ are the two columns of Table~\ref{a:tab:N15coeff}.
With $\star\in\{\mathrm{II},\mathrm I\}$, define
\begin{equation}\label{eq:fixed-nonumbilic-families}
h^{\star,\mathrm{nu}}_{N,\tau}
:=\tau tA+\sum_{1<2k+p\le6}c^{\star,15}_{kp}t^p
S_N^{(k)}(b_k,e_k),\qquad N\ge15.
\end{equation}
Thus only the $(0,1)$ coefficient is replaced by $\tau$.
These are the divergence-free profiles, since $S_N^{(k)}(b_k,e_k)=P_N^{(k)}$.

For the umbilic families use the same definition with the triples $(c^{\star}_{kp},b^{\star}_{kp},e^{\star}_{kp})$ specified by the products in Table~\ref{tab:umbilic-coefficients}:
\begin{equation}\label{eq:fixed-umbilic-families}
h^{\star,\mathrm u}_{N,\tau}
:=c^{\star}_{02}\tau t^2A
+\sum_{(k,p)\ne(0,2)}c^{\star}_{kp}t^p
S_N^{(k)}(b^{\star}_{kp},e^{\star}_{kp}).
\end{equation}
The sum runs over the other listed rows.
There are no $p=1$ rows.
Only the $(0,2)$ row is multiplied by $\tau$.
More explicitly, the data and dimension ranges are
\[
\begin{array}{c|c|c|c}
\text{equation}&A&c_{02}&\text{range}\\ \hline
\mathrm{II},\mathrm u&A_{\mathrm c}&-2&N\ge22\\
\mathrm I,\mathrm u&A_{\mathrm b}&-1.2158794046&N\ge21.
\end{array}
\]
The scalar-flat family has total degree eight and $c^{\mathrm{II}}_{k0}=0$ for every $k$.
The minimal family has total degree nine and its only nonzero boundary coefficients are
\[
c^{\mathrm I}_{20}=-0.002306568,\qquad
c^{\mathrm I}_{30}=0.0003657766,\qquad
c^{\mathrm I}_{40}=-0.000009126.
\]
These three rows use \eqref{eq:longitudinal-boundary-rule}, so $h(\bar x,0)\bar x=0$; the scalar-flat profile satisfies this identity because its boundary value vanishes.
All independent coefficients in the two tables are fixed exact rationals throughout the indicated dimension ranges.
The structural coefficients in \eqref{a:eq:Achain} and \eqref{eq:longitudinal-boundary-rule} retain their explicit dependence on $m$.
The stationarity equation selects $\tau$, which is then held fixed under bubble variations and localization.
Both umbilic families have nonzero divergence and use the full source \eqref{eq:polynomial-source-functional} and raw coefficient $\mathcal Q[h]$.
\begin{proposition}
The two nonumbilic profiles satisfy
\[
\partial_t h(0,0)=\tau A,
\]
so their exponential metrics have nonumbilic boundary when $\tau\ne0$.
Both umbilic profiles satisfy
\[
\partial_t h(\bar x,0)=0;
\]
their boundaries are totally geodesic, also after multiplication by a radial cutoff.
The scalar-flat umbilic profile has zero boundary value and, for $\tau\ne0$, nonzero linearized normal Weyl curvature at the origin.
For the minimal profile, the squared eigenvalues of $A$ lie in $[0,1/2]$; they coincide precisely when $N$ is odd.
\end{proposition}
\begin{proof}
In \eqref{eq:fixed-nonumbilic-families}, the terms other than $\tau tA$ have degree greater than one.
Thus the first normal derivative at the origin is $\tau A$, a nonzero trace-free matrix when $\tau\ne0$.
Section~\ref{sec:certification} selects $\tau>0$.

For the scalar-flat umbilic profile, every boundary row is zero.
Thus every nonzero row has $p\ge2$, the tensor has a factor $t^2$, and
\[
h(\bar x,0)=\partial_t h(\bar x,0)=0.
\]
The induced boundary metric is Euclidean and the exponential metric has exactly totally geodesic boundary, also after a radial cutoff.
The leading metric term is $-2\tau t^2A$.

The tensor is generally not divergence free.
Its divergence has the term
\[
\left(0.0836m+0.9836-\frac{0.3344}{m}\right)t^2A\bar x
\]
with no additional factor of $s$; the coefficient is positive for $m\ge21$.
The full source and both divergence contributions to the raw energy are therefore essential.
The leading metric term also determines
\[
W'_{aNbN}(0)=2\tau\frac{N-3}{N-2}A_{ab}.
\]
The selected coefficient is positive, so this curvature is nonzero.
The degree condition is $N>18$; the sign proof starts at $N=22$.

For the minimal umbilic profile, the boundary rows need not vanish.
For example, the $(1,2)$ row alone contributes
\[
\operatorname{div}\bigl(c_{12}t^2S_{21}^{(1)}(b_{12},0)\bigr)
=\frac{50899050721}{5\cdot10^9}t^2A_{\rm b}\bar x\ne0.
\]
This follows from Lemma~\ref{lem:scalar-divergence} with $m=20$ and the listed exact coefficients.
Thus the divergence terms in $S_z[h]$ must be retained for this tensor.
The boundary value is generally nonzero, but there is no $p=1$ row, so $\partial_t h(\bar x,0)=0$.
Odd normal powers $p\ge3$ are retained; no reflection symmetry is imposed.

The two squared eigenvalues of the generating matrix are $m_-/(m_+m)$ and $m_+/(m_-m)$.
They lie in $[0,1/2]$ throughout $N\ge21$.
For every odd ambient dimension, $m_+=m_-$ and $A_{\mathrm b}^2=I/(N-1)$, so the invariant translation Hessian is scalar.
In even ambient dimensions the two squared eigenvalues are distinct and both enter its evaluation.
Multiplication by a smooth radial cutoff preserves the zero first normal derivative at $t=0$.
Total geodesicity follows from Lemma~\ref{lem:polynomial-boundary-geometry}.
The Weyl identity used above is also a specialization of Proposition~\ref{prop:polynomial-geometric-jets} below.
\end{proof}

\subsection{Sources, curvature, and angular identities}
\label{subsec:polynomial-identities}
We compute directly with the rows in \eqref{eq:full-scalar-profile}.
We write $\widetilde S_z[h]$ for the source formed with the normalized bubble $\widetilde U_z$, and $\widetilde S[h]:=\widetilde S_{z_0}[h]$ at the centered unit bubble.
The source $S_z[h]$ continues to use the geometric bubble $U_z$; their relation is stated in \eqref{eq:source-corrector-normalization-bridge}.
For a fixed row $t^pS_N^{(k)}(b,e)$, write $u:=u_k$, $v:=v_k$ for the divergence coefficients of Lemma~\ref{lem:scalar-divergence}, and put
\[
\mathsf r_k:=b+e-\frac{2b+e}m,\qquad
\mathsf s_k:=1+b+\mathsf r_k,\qquad
\mathsf d_k:=2(k-1)u+(m+2k-2)v.
\]
The dependence of $\mathsf r_k,\mathsf s_k,\mathsf d_k$ on the row coefficients $(b,e)$ is suppressed.
For $(b,e)=(b_k,e_k)$, $\mathsf s_k$ agrees with the contraction coefficient in \eqref{a:eq:Acontractions}.
For $k=0$ take $u=v=\mathsf d_k=0$; for $k=1$ take $v=\mathsf d_k=0$.
In every formula below, terms whose coefficient is zero are omitted before evaluating a power of $s$.

\begin{lemma}
\label{lem:scalar-centered-source}
For $h=t^pS_N^{(k)}(b,e)$,
\begin{align*}
h\bar x&=t^p\bigl[(1+b)s^kA\bar x+\mathsf r_k s^{k-1}Q_A\bar x\bigr],\\
\bar x^{\mathsf T}h\bar x&=\mathsf s_k s^kt^pQ_A,\\
\divg h&=t^p\bigl[us^{k-1}A\bar x+vs^{k-2}Q_A\bar x\bigr],\\
\divg\divg h&=\mathsf d_k s^{k-2}t^pQ_A.
\end{align*}
For the normalized centered bubble $\widetilde U:=D^{-(N-2)/2}$, with $D=s+(1+t)^2$ in Type II and $D=1+s+t^2$ in minimal Type I, the full source is
\begin{equation}\label{eq:full-scalar-centered-source}
\widetilde S[h]
=\widetilde U Q_A t^p\left[
\frac{N(N-2)\mathsf s_k s^k}{D^2}
-\frac{(N-2)(u+v)s^{k-1}}D+c_N\mathsf d_k s^{k-2}\right].
\end{equation}
For a profile of the form \eqref{eq:full-scalar-profile}, multiply each row formula by $c_{kp}$ and sum.
Proposition~\ref{sec:certificate-normalization} relates this normalization to the geometric bubble $U_z$.
\end{lemma}
\begin{proof}
The contraction identities follow from \eqref{eq:full-scalar-module}, and the divergence identities follow from Lemma~\ref{lem:scalar-divergence}.
At the centered bubble,
\[
\pa_a\widetilde U=-\frac{(N-2)x_a}{D}\widetilde U,\qquad
\pa_{ab}\widetilde U=
\left[\frac{N(N-2)x_ax_b}{D^2}
-\frac{(N-2)\delta_{ab}}D\right]\widetilde U.
\]
Substitute these expressions in
\[
\widetilde S[h]=h_{ab}\pa_{ab}\widetilde U
+(\pa_ah_{ab})\pa_b\widetilde U
+c_N(\pa_a\pa_bh_{ab})\widetilde U.
\]
The trace term vanishes.
The remaining three terms are those in \eqref{eq:full-scalar-centered-source}; the formula for their sum follows by linearity.
In the divergence-free specialization, $u=v=\mathsf d_k=0$, so only the first term remains.
\end{proof}

\begin{proposition}
\label{prop:polynomial-geometric-jets}
Use the outward normal.
For a tangential symmetric trace-free tensor $h$ and $g=\exp(\mu h)$,
\[
\begin{gathered}
L_{ab}=-\frac12\pa_tg_{ab},\\
h=-2tP(\bar x),\ \tr P=0\ \Rightarrow\ L_g|_{t=0}=\mu P(\bar x),\\
\pa_th(\bar x,0)=0\ \Rightarrow\ L_g|_{t=0}=0.
\end{gathered}
\]

Let $A\in\operatorname{Sym}^2_0(\R^{N-1})$ and $h_2:=\alpha P_N^{(1)}+\zeta t^2A$, with $\alpha,\zeta\in\R$.
Then
\begin{equation}\label{eq:degree2Weyl}
W^{\mathrm{lin}}_{atbt}
=\frac{\alpha(m+2b_1)-(N-3)\zeta}{N-2}A_{ab}.
\end{equation}
For a general normalized scalar row, take $h_2:=\alpha S_N^{(1)}(b,0)+\zeta t^2A$ and put $u_1:=2+(m+2-4/m)b$.
Then
\begin{equation}\label{eq:full-degree2-curvature}
\begin{aligned}
(\divg h_2)_a&=\alpha u_1(A\bar x)_a,& R^{\mathrm{lin}}&=0,\\
R^{\mathrm{lin}}_{atbt}&=-\zeta A_{ab},&
\operatorname{Ric}^{\mathrm{lin}}_{tt}&=0,\\
\operatorname{Ric}^{\mathrm{lin}}_{ab}
&=[\alpha(u_1-m-2b)-\zeta]A_{ab},\\
W^{\mathrm{lin}}_{atbt}
&=\frac{\alpha(m+2b-u_1)-(N-3)\zeta}{N-2}A_{ab}.
\end{aligned}
\end{equation}
If $|A|=1$ and $h=-2ctA+O(|x|^3)$, the exact metric $\exp(h)$ satisfies at the origin
\begin{equation}\label{eq:linear-normal-Weyl}
W_{aNbN}=c^2\left(\frac{\delta_{ab}}{N-1}-(A^2)_{ab}\right).
\end{equation}
\end{proposition}
\begin{proof}
The boundary identities follow from
\[
L_{ab}=-\frac12\pa_tg_{ab}
\]
and \eqref{eq:umbilic-polynomial-condition}.
For $h_2=\alpha S_N^{(1)}(b,0)+\zeta t^2A$, direct differentiation shows that
\[
\divg h_2=\alpha u_1A\bar x,\qquad
\Delta h_2=[\alpha(2m+4b)+2\zeta]A.
\]
Since $\tr A=0$, the double divergence vanishes even when $u_1\ne0$.
The first five identities in \eqref{eq:full-degree2-curvature} follow from the flat curvature linearization; the last follows by subtracting the Schouten tensor.
For the divergence-free specialization $b=b_1$ and $u_1=0$, which recovers \eqref{eq:degree2Weyl} without discarding a divergence term in the larger family.

For the linear normal term, the exponential metric has at the origin
\[
R_{aNbN}=-c^2(A^2)_{ab},\qquad
\operatorname{Ric}_{ab}=0,\qquad
\operatorname{Ric}_{NN}=-c^2,\qquad R=-c^2.
\]
To check these formulas, diagonalize $A$ and use the two-jet of $\dd t^2+\sum_a \exp(-2ctA_{aa})(\dd x^a)^2$.
Its normal sectional curvatures are $-c^2A_{aa}^2$ and its tangential sectional curvatures are $-c^2A_{aa}A_{bb}$; the trace-free condition cancels the tangential Ricci terms.
Subtracting the Schouten part proves \eqref{eq:linear-normal-Weyl}.
Terms of degree at least three do not affect the metric two-jet.
\end{proof}

In the nonumbilic constructions, $A=P(0)$ determines the leading trace-free second fundamental form.
The nonzero Weyl coefficients for all four scalar families are checked in Section~\ref{sec:applications}.
For $A=2^{-1/2}\diag(1,-1,0,\ldots,0)$, \eqref{eq:linear-normal-Weyl} is nonzero whenever $c\ne0$.
A change of curvature-sign convention changes both curvature formulas by a sign.

For the angular calculations, write
\[
\bar x=\sqrt{s}\,\omega,\qquad
\omega=(\omega_1,\ldots,\omega_m)\in\Sph^{m-1},\qquad
\sum_{i=1}^m\omega_i^2=1.
\]
Thus $\omega_i$ is the $i$-th coordinate on the tangential unit sphere, and $\dd\omega$ denotes its Euclidean surface measure.
We use
\[
(b)_0:=1,\qquad (b)_j:=\prod_{r=0}^{j-1}(b+r),\qquad
\av f:=\frac1{|\Sph^{m-1}|}\int_{\Sph^{m-1}}f\dd\omega.
\]
\begin{lemma}\label{lem:angular-gram}
A spherical monomial with an odd exponent has average zero, while
\begin{equation}\label{eq:common-spherical-moment}
\av\prod_{i=1}^m\omega_i^{2r_i}
=\frac{\prod_{i=1}^m(1/2)_{r_i}}{(m/2)_{r_1+\cdots+r_m}},
\qquad r_i\in\{0,1,2,\ldots\}.
\end{equation}
For $|A|=1$, let $S:=S_N^{(k)}(b,e)$ and $\widetilde S:=S_N^{(l)}(\widetilde b,\widetilde e)$ on the unit sphere.
Put $q_A:=2/[m(m+2)]$, and let $\mathsf r_k,\widetilde{\mathsf r}_l$ be the corresponding radial coefficients defined above.
Then
\begin{align*}
\av\langle S,\widetilde S\rangle
={}&1+\frac{2(b+\widetilde b)+2b\widetilde b}{m}\\
&+q_A\left[e+\widetilde e+2b\widetilde b
+2(b\widetilde e+\widetilde b e)+e\widetilde e
-\frac{(2b+e)(2\widetilde b+\widetilde e)}m\right],\\
\av\langle S\bar x,\widetilde S\bar x\rangle
={}&\frac{(1+b)(1+\widetilde b)}m
+q_A\bigl[(1+b)\widetilde{\mathsf r}_l+(1+\widetilde b)\mathsf r_k
+\mathsf r_k\widetilde{\mathsf r}_l\bigr].
\end{align*}
For the divergence-free specialization, define
\[
C_{kl}:=\av\langle P_N^{(k)},P_N^{(l)}\rangle,\quad
R_{kl}:=\av\langle P_N^{(k)}\bar x,P_N^{(l)}\bar x\rangle,\quad
G_{kl}:=\av\langle\nabla_{\bar x}P_N^{(k)},\nabla_{\bar x}P_N^{(l)}\rangle.
\]
The first two are given by the preceding formulas with $(b,e)=(b_k,e_k)$ and $(\widetilde b,\widetilde e)=(b_l,e_l)$.

Furthermore,
\begin{equation}\label{a:eq:Gkl}
2G_{kl}=d(d+m-2)C_{kl}-\vartheta_kC_{k-1,l}-\vartheta_lC_{k,l-1},
\qquad d:=2k+2l,
\end{equation}
with terms containing a negative index omitted.
For the divergence-free generators, $C_{kl}$, $R_{kl}$ and $G_{kl}$ are rational functions of $m$.
\end{lemma}
\begin{proof}
Reflection symmetry makes the odd moments vanish.
Integrating a Gaussian monomial in Cartesian and polar coordinates proves \eqref{eq:common-spherical-moment}.
In particular, $\av|A\omega|^2=1/m$ and $\av Q_A(\omega)^2=q_A$.
Expanding the trace-free formulas for $S$ and $\widetilde S$ and using these two moments proves the displayed inner products.
Substituting $(b,e)=(b_k,e_k)$ and $(\widetilde b,\widetilde e)=(b_l,e_l)$ then determines $C_{kl}$ and $R_{kl}$.

For $G_{kl}$, apply the Laplacian to $\langle P_N^{(k)},P_N^{(l)}\rangle$ and use \eqref{a:eq:Acontractions}.
This scalar polynomial is homogeneous of degree $d=2k+2l$.
The spherical average of its Laplacian is $d(d+m-2)C_{kl}$, since the integral of the spherical Laplacian is zero.
The product rule now yields \eqref{a:eq:Gkl}.
\end{proof}

\subsection{Translation and scale derivatives}
\label{subsec:polynomial-parameter-derivatives}
For the four scalar-flat and minimal-boundary profiles, use the normalized bubbles
\[
\widetilde U_{0,\xi,\eps}:=2^{-(N-2)/2}U_{0,\xi,\eps},
\qquad \widetilde U^{\mathrm{sf}}_{\xi,\eps}:=U^{\mathrm{sf}}_{\xi,\eps}.
\]
At the centered unit bubble, $D=1+s+t^2$ in the minimal-boundary case and $D=s+(1+t)^2$ in the scalar-flat case.
The Jacobi forms remain \eqref{eq:common-Jacobi-form}, with interior potential $N(N+2)D^{-2}$ in the minimal case and boundary potential $N(1+s)^{-1}$ in the scalar-flat case.
For $(\sigma,\eta)\in\{(0,2),(1,0)\}$, define the rescaled quadratic functional
\begin{equation}\label{eq:certificate-functional-normalization}
\widehat F_h:=
\begin{cases}
c_N^{-1}2^{2-N}F_h,&(\sigma,\eta)=(1,0),\\
c_N^{-1}F_h,&(\sigma,\eta)=(0,2).
\end{cases}
\end{equation}
For $z=(\xi,\eps)$, write $\widetilde U:=\widetilde U_z$ and use the projected Jacobi operator $\cL_z:\Sigma_z\to\Sigma_z^*$ defined by \eqref{eq:common-Jacobi-form}.
It is linearized at the geometric bubble $U_z$.
Define the normalized source by
\[
\widetilde S_z[h]:=h_{ab}\widetilde U_{ab}+(\pa_bh_{ab})\widetilde U_a
+c_N(\pa_a\pa_bh_{ab})\widetilde U.
\]
Define the normalized corrector by
\[
Z:=Z_z[h]:=\cL_z^{-1}\widetilde S_z[h]\quad\Leftrightarrow\quad
\begin{cases}
Z\in\Sigma_z,\\
\displaystyle\cB_z(Z,\psi)=\int_{\R^N_+}\widetilde S_z[h]\psi\dd x
&\text{for every }\psi\in\Sigma_z,
\end{cases}
\]
and $\langle \widetilde S_z[h],\cL_z^{-1}\widetilde S_z[h]\rangle:=\int_{\R^N_+}\widetilde S_z[h]Z\dd x$.
Thus the inverse is taken on the complement of the translation and scale kernel, with the source formed from $\widetilde U_z$.

The normalized source and the geometric scalar correction satisfy
\begin{equation}\label{eq:source-corrector-normalization-bridge}
\widetilde S_z[h]=\frac{\widetilde U_z}{U_z}S_z[h],\qquad
v_z[h]=-\frac{U_z}{\widetilde U_z}Z_z[h].
\end{equation}
The ratios are constant in space and in $z$.
The sign follows from \eqref{eq:linear-corrector-inverse}, while the energy normalizations are \eqref{eq:certificate-functional-normalization}.

\begin{proposition}\label{sec:certificate-normalization}
For the scalar-flat or minimal-boundary equation, the normalized quadratic functional satisfies
\begin{equation}\label{a:eq:F}
\begin{aligned}
\widehat F_h={}&\frac{c_N^{-1}}{2}\int(h^2)_{ab}\widetilde U_a\widetilde U_b
-\frac14\int|\pa h|^2\widetilde U^2
+\frac12\int|\operatorname{div}h|^2\widetilde U^2\\
&+\int h_{ak}(\pa_bh_{bk})\pa_a(\widetilde U^2)
-c_N^{-1}\langle \widetilde S_z[h],\cL_z^{-1}\widetilde S_z[h]\rangle.
\end{aligned}
\end{equation}
The rescaling \eqref{eq:certificate-functional-normalization} preserves critical points and Hessian inertia.
The normalized source retains the kernel orthogonality of Lemma~\ref{lem:source-kernel-orthogonality}.
\end{proposition}
\begin{proof}
Substitute the constant multiples defining $\widetilde U$ into the raw coefficient $\mathcal Q[h]$ defined by \eqref{eq:H-quadratic-cutoff-exact} and the scalar correction of Proposition~\ref{prop:scalar-elimination}.
Multiplying the bubble by a constant multiplies each raw term by the square of that constant.
The source is linear in the bubble, so its inverse pairing has the same scaling.
Dividing by $c_N$, we obtain \eqref{a:eq:F}.
The Jacobi form is unchanged, since its potentials are those of the original geometric bubble.
The factors in \eqref{eq:certificate-functional-normalization} are positive and independent of the bubble parameters, proving the remaining claims.
\end{proof}

For logarithmic scale $\ell:=\log\eps$,
\begin{equation}
\partial_\ell=\eps\partial_\eps,\qquad
\partial_{\ell\ell}\widehat F_h
=\eps^2\partial_{\eps\eps}\widehat F_h+\eps\partial_\eps\widehat F_h.
\end{equation}
At a stationary scale the two second derivatives have the same sign; at scale one they agree.

\begin{lemma}
\label{lem:polynomial-integrability-scaling}
Let $h$ be a symmetric, tangential, trace-free polynomial tensor with $\deg h\le d_*$ and
\begin{equation}\label{a:eq:degreecondition}N>2d_*+2.
\end{equation}
All terms in \eqref{a:eq:F} are finite and depend $C^2$ on the bubble parameters, uniformly on compact parameter sets.

For a decomposition into homogeneous tangential, trace-free tensors $h_i$ (with repeated degrees allowed), put $\ell:=\log\eps$ and
\[
h=\sum_i c_i h_i,\qquad \deg h_i=d_i,\qquad
\mathbf c(\ell):=\exp(\ell\mathbf D)\mathbf c,\qquad
\mathbf c:=(c_i),\qquad \mathbf D:=\diag(d_i).
\]
At the centered unit bubble, write $\widetilde U:=\widetilde U_{(0,1)}$ and $\cL:=\cL_{(0,1)}$.
Set $S_i:=\widetilde S[h_i]$ and define the symmetric pairing
\begin{equation}\label{eq:coefficient-matrix-pairing}
\begin{aligned}
|\Sph^{N-2}|\mathcal M(h_i,h_j):={}&
\frac1{2c_N}\int(h_i\nabla\widetilde U)\cdot(h_j\nabla\widetilde U)
-\frac14\int\langle\pa h_i,\pa h_j\rangle\widetilde U^2\\
&+\frac12\int(\divg h_i)\cdot(\divg h_j)\widetilde U^2\\
&+\frac12\int\bigl[h_i(\divg h_j)+h_j(\divg h_i)\bigr]
  \cdot\nabla(\widetilde U^2)
-c_N^{-1}\langle S_i,\cL^{-1}S_j\rangle.
\end{aligned}
\end{equation}
All integrals are over $\R^N_+$.
The same formula applies when either row is replaced by a tangential derivative of that row.
Let $\mathbf M^{\rm raw}_{ij}$ denote the sum of the four integral terms divided by $|\Sph^{N-2}|$.
The corrected matrix has entries
\begin{equation}\label{eq:coefficient-matrix-entries}
\mathbf M_{ij}:=\mathcal M(h_i,h_j)
=\mathbf M^{\rm raw}_{ij}
-\frac{c_N^{-1}}{|\Sph^{N-2}|}\langle S_i,\cL^{-1}S_j\rangle.
\end{equation}
Then
\begin{equation}\label{a:eq:scalingmatrix}
\begin{aligned}
\frac{\widehat F_h(0,\eps)}{|\Sph^{N-2}|}&=\mathbf c(\ell)^{\mathsf T}\mathbf M\mathbf c(\ell),\\
\frac{\widehat F_{h,\ell}(0,\eps)}{|\Sph^{N-2}|}
&=\mathbf c(\ell)^{\mathsf T}(\mathbf D\mathbf M+\mathbf M\mathbf D)\mathbf c(\ell),\\
\frac{\widehat F_{h,\ell\ell}(0,\eps)}{|\Sph^{N-2}|}
&=\mathbf c(\ell)^{\mathsf T}(\mathbf D^2\mathbf M+2\mathbf D\mathbf M\mathbf D+\mathbf M\mathbf D^2)\mathbf c(\ell),
\qquad \ell=\log\eps.
\end{aligned}
\end{equation}
The translation Hessian is $\eps^{-2}$ times a quadratic form in $\mathbf c(\ell)$.
\end{lemma}
\begin{proof}
Proposition~\ref{prop:scalar-elimination} establishes integrability, cutoff convergence, and $C^2$ dependence of the corrected coefficient under \eqref{a:eq:degreecondition}; its proof also establishes these properties
for the separate metric and source terms.
Proposition~\ref{sec:certificate-normalization} preserves them under its fixed positive rescaling.
Rescale the spatial variables by $\eps$.
Homogeneity changes the coefficient of $h_i$ to $c_i\eps^{d_i}$.
The bubble, Jacobi form and orthogonality constraints have the same covariance, so the first identity in \eqref{a:eq:scalingmatrix} follows from uniqueness in \eqref{eq:projected-linear-corrector}.
The other two follow by differentiation in $\ell$.
The translated tensor depends on $\xi/\eps$; two translation derivatives contribute the factor $\eps^{-2}$.

The evaluation of the pairing and the assembly of the coefficient matrices are given in Proposition~\ref{prop:endpoint-matrix-assembly}.
\end{proof}

\begin{proposition}
\label{prop:scalar-translated-sources}
Let $h$ belong to \eqref{eq:full-scalar-profile}, with $A\ne0$ and scalar coefficients independent of the entries of $A$.
Assume the pairings and their first two parameter derivatives are finite.
For either prescribed pair $(0,2)$ or $(1,0)$, the centered source in \eqref{eq:full-scalar-centered-source} has tangential harmonic degree two.
Its first translated sources have only degrees one and three; only the degree-two component of each second translated source can pair with the centered source.
At every centered bubble, there are scalar functions $\alpha,\gamma$ such that
\begin{align*}
D_\xi^2\widehat F_h(0,\eps)&=\alpha(\eps)|A|^2\Id+\gamma(\eps)A^2,\\
D_\xi\widehat F_h(0,\eps)&=0,\qquad
D_\xi D_\eps\widehat F_h(0,\eps)=0.
\end{align*}
For a unit generating matrix, choose orthonormal tangential coordinates in which
\[
A^2=\diag(\lambda_1,\ldots,\lambda_m),\qquad
\lambda_a\ge0,\qquad \sum_{a=1}^m\lambda_a=1.
\]
Thus $\lambda_a$ is the square of an eigenvalue of $A$, and the translation eigenvalue in direction $\mathbf e_a$, the $a$th coordinate unit vector, is $\alpha(\eps)+\lambda_a\gamma(\eps)$.
For rows $h_i$ formed with this $A$, define
\begin{equation}\label{eq:translation-coefficient-matrices}
\begin{split}
(\mathbf T_{\lambda_a})_{ij}:={}&2\mathcal M(\pa_a h_i,\pa_a h_j)
+\mathcal M(\pa_{aa}h_i,h_j)+\mathcal M(h_i,\pa_{aa}h_j),\\
&a=1,\ldots,m.
\end{split}
\end{equation}
For fixed row coefficients and a fixed scalar-flat or minimal-boundary equation, these matrices depend on $A$ only through the indicated eigenvalue and are affine in that eigenvalue.
If $\lambda_a\ne\lambda_b$, they determine
\[
\mathbf T_\lambda:=
\frac{\lambda_b-\lambda}{\lambda_b-\lambda_a}\mathbf T_{\lambda_a}
+\frac{\lambda-\lambda_a}{\lambda_b-\lambda_a}\mathbf T_{\lambda_b}.
\]
With $\mathbf c(\ell)$ as in \eqref{a:eq:scalingmatrix}, the explicit coefficient formulas are
\begin{equation}\label{eq:explicit-alpha-gamma}
\begin{aligned}
\alpha(\eps)&=|\Sph^{N-2}|\eps^{-2}
\mathbf c(\ell)^{\mathsf T}
\frac{\lambda_b\mathbf T_{\lambda_a}-\lambda_a\mathbf T_{\lambda_b}}
{\lambda_b-\lambda_a}\mathbf c(\ell),\\
\gamma(\eps)&=|\Sph^{N-2}|\eps^{-2}
\mathbf c(\ell)^{\mathsf T}
\frac{\mathbf T_{\lambda_b}-\mathbf T_{\lambda_a}}
{\lambda_b-\lambda_a}\mathbf c(\ell).
\end{aligned}
\end{equation}
The indices $i,j$ label homogeneous rows of the chosen family, so $\mathbf T_\lambda$ is a matrix acting on their coefficients.
If $A^2=I/m$, its translation Hessian determines only $\alpha(\eps)+\gamma(\eps)/m$.
The two coefficients can still be evaluated separately by forming the same rows with a unit generating matrix having two distinct squared eigenvalues, as in Proposition~\ref{prop:endpoint-matrix-assembly}.
The full inverse contribution is given by \eqref{eq:full-source-Hessian-pairing}, with the explicit source components \eqref{eq:full-scalar-first-source} and
\eqref{eq:full-scalar-second-source-projection}.
\end{proposition}
\begin{proof}
Translate the tensor against the fixed centered unit bubble.
In this proof, $\cL=\cL_{(0,1)}$ is the projected Jacobi operator at that bubble; translation changes the tensor and its source, while this operator and its domain $\Sigma_{(0,1)}$ remain fixed.
For every homogeneous tensor $h_i$ define
\[
S_i:=\widetilde S[h_i],\qquad S_i^{(a)}:=\widetilde S[\pa_a h_i],\qquad
S_i^{(ab)}:=\widetilde S[\pa_a\pa_b h_i],
\]
where $\widetilde S[h]$ is the normalized source defined before Proposition~\ref{sec:certificate-normalization}.
The Jacobi inverse is then fixed.
For $h_\xi(\bar x,t):=h(\bar x+\xi,t)$, linearity implies
\begin{equation}\label{eq:full-source-Hessian-pairing}
\pa_{\xi_a\xi_b}\langle \widetilde S[h_\xi],\cL^{-1}\widetilde S[h_\xi]\rangle\big|_{\xi=0}
=2\langle S^{(a)},\cL^{-1}S^{(b)}\rangle
+2\langle S^{(ab)},\cL^{-1}S\rangle.
\end{equation}
Polarizing \eqref{a:eq:F}, the $(i,j)$ coefficient in the second derivative in direction $a$ is the derivative of all four raw terms minus
\begin{equation}
c_N^{-1}\{2\langle S_i^{(a)},\cL^{-1}S_j^{(a)}\rangle
+\langle S_i^{(aa)},\cL^{-1}S_j\rangle
+\langle S_i,\cL^{-1}S_j^{(aa)}\rangle\}.
\end{equation}

Angular orthogonality separates different harmonic degrees, but does not discard multiplicities within a degree or cross terms between normal powers.
The projection of the second source must retain its full component that pairs with the centered source.

It suffices to differentiate a row $h:=t^pS_N^{(k)}(b,e)$.
Use $\mathsf r_k,\mathsf s_k,u,v,\mathsf d_k$ from Subsection~\ref{subsec:polynomial-identities}.
The centered source \eqref{eq:full-scalar-centered-source} has tangential angular degree two.
Let
\[
C_a:=x_aQ_A-\frac{2s}{m+2}(A\bar x)_a,\qquad
\Delta_{\bar x}C_a=0,
\]
and set
\begin{align*}
r_1&:=t^p\left[
\frac{2N(N-2)\mathsf r_k s^k}{D^2}
-\frac{(N-2)(u+2v)s^{k-1}}D+2c_N\mathsf d_k s^{k-2}\right],\\
r_3&:=t^p\biggl[
\frac{2N(N-2)(k\mathsf s_k-\mathsf r_k)s^{k-1}}{D^2}
-\frac{(N-2)\{2(k-1)(u+v)-v\}s^{k-2}}D\\
&\hspace{18mm}+2c_N(k-2)\mathsf d_k s^{k-3}\biggr].
\end{align*}
Then
\begin{equation}\label{eq:full-scalar-first-source}
\widetilde S[\pa_a h]=\widetilde U
\left\{\left(r_1+\frac{2s}{m+2}r_3\right)(A\bar x)_a+r_3C_a\right\}.
\end{equation}
This displays the degree-one and degree-three components separately.
Kernel orthogonality removes the Jacobi kernel pairing; it does not remove the whole degree-one component.
A nonzero degree-one component requires a projected inverse estimate, separately from the degree-at-least-two coercivity inequalities below.

For the second translated source, use vectors whose entries correspond to $A_{ab}$, $(A\bar x)_ax_b+(A\bar x)_bx_a$, $\delta_{ab}Q_A$, and $x_ax_bQ_A$.
The tensor, divergence, and double-divergence terms have coefficient vectors $\mathsf H^{(0)},\mathsf H^{(1)},\mathsf H^{(2)}$, respectively:
\begin{align*}
\mathsf H^{(0)}_A&:=2(\mathsf r_k-b)s^k,
&\mathsf H^{(0)}_B&:=\{4(k-1)\mathsf r_k+2b\}s^{k-1},\\
\mathsf H^{(0)}_I&:=\left\{2k\mathsf s_k-4\mathsf r_k-\frac{2(2b+e)}m\right\}s^{k-1},
&\mathsf H^{(0)}_x&:=\{4k(k-1)\mathsf s_k-8(k-1)\mathsf r_k+2e\}s^{k-2},\\
\mathsf H^{(1)}_A&:=2vs^{k-1},
&\mathsf H^{(1)}_B&:=\{2(k-1)u+(4k-6)v\}s^{k-2},\\
\mathsf H^{(1)}_I&:=\{2(k-1)u+2(k-2)v\}s^{k-2},
&\mathsf H^{(1)}_x&:=4(k-2)\{(k-1)(u+v)-v\}s^{k-3},\\
\mathsf H^{(2)}_A&:=2\mathsf d_k s^{k-2},& \mathsf H^{(2)}_B&:=4(k-2)\mathsf d_k s^{k-3},\\
\mathsf H^{(2)}_I&:=2(k-2)\mathsf d_k s^{k-3},& \mathsf H^{(2)}_x&:=4(k-2)(k-3)\mathsf d_k s^{k-4}.
\end{align*}
Set $f:=t^p\{N(N-2)D^{-2}\mathsf H^{(0)}-(N-2)D^{-1}\mathsf H^{(1)}+c_N\mathsf H^{(2)}\}$.
With entries $f_A,f_B,f_I,f_x$, one has
\[
\widetilde S[\pa_a\pa_bh]=\widetilde U\bigl[
f_AA_{ab}+f_B\{(A\bar x)_ax_b+(A\bar x)_bx_a\}
+f_I\delta_{ab}Q_A+f_xx_ax_bQ_A\bigr].
\]
Only its degree-two component parallel to $Q_A$ pairs with the centered source.
For $A\ne0$ that component is
\begin{equation}\label{eq:full-scalar-second-source-projection}
\widetilde UQ_A\left[
f_I\delta_{ab}+2f_B\frac{(A^2)_{ab}}{|A|^2}
+\frac{s f_x}{m+4}
\left(\delta_{ab}+4\frac{(A^2)_{ab}}{|A|^2}\right)\right].
\end{equation}
These identities follow by differentiating the radial and divergence identities for $t^pS_N^{(k)}(b,e)$ in Lemma~\ref{lem:scalar-centered-source}.
The last projection follows from \eqref{eq:common-spherical-moment}.
Terms with zero coefficient are omitted before evaluating powers of $s$, also for $k=0,1$.
For the full profile, multiply each row formula by $c_{kp}$ and sum; this retains all cross terms in the first/first and second/zeroth pairings of \eqref{eq:full-source-Hessian-pairing}.

For a profile generated by one generating matrix $A$, orthogonal equivariance and quadratic dependence on $A$ imply
\begin{equation}\label{eq:full-scalar-translation-equivariance}
D_\xi^2\widehat F_h(0,\eps)
=\alpha(\eps)|A|^2\Id+\gamma(\eps)A^2.
\end{equation}
Indeed these are the symmetric matrix covariants of degree two of a trace-free symmetric matrix: the remaining possible term $(\tr A)A$ vanishes.
By tangential parity, $D_\xi\widehat F_h(0,\eps)=0$ and $D_\xi D_\eps\widehat F_h(0,\eps)=0$.
Differentiating the polarized energy in direction $\mathbf e_a$ proves \eqref{eq:translation-coefficient-matrices}.
The invariant Hessian formula holds for every coefficient vector $\mathbf c$; polarization in $\mathbf c$ therefore proves the affine dependence of $\mathbf T_{\lambda_a}$ on $\lambda_a$.
By homogeneous scaling, for every squared eigenvalue,
\[
\alpha(\eps)+\lambda_a\gamma(\eps)
=|\Sph^{N-2}|\eps^{-2}\mathbf c(\ell)^{\mathsf T}
\mathbf T_{\lambda_a}\mathbf c(\ell).
\]
Solving the two equations corresponding to distinct $\lambda_a,\lambda_b$ proves \eqref{eq:explicit-alpha-gamma}.
If all squared eigenvalues coincide, their sum is one, so each is $1/m$ and only the stated combination is determined by that Hessian.
\end{proof}

\begin{corollary}
For $h=\sum c_{kp}t^pP_N^{(k)}$, the centered source has only angular degree two and the first translated sources have only angular degree three.
Consequently, the scalar correction to the translation Hessian uses only the degree-three first/first pairing and the degree-two second/zeroth pairing.
For the rank-two unit generating matrix, the two translation eigenvalues are
\[
\alpha(\eps)\quad\text{with multiplicity }N-3,\qquad
\alpha(\eps)+\tfrac12\gamma(\eps)\quad\text{with multiplicity }2,
\]
as in \eqref{a:eq:twoeigenvalues}.
\end{corollary}
\begin{proof}
For the divergence-free specialization of \eqref{eq:full-scalar-profile}, with the normalization \eqref{a:eq:normalizedA}, the centered source is
\begin{equation}\label{a:eq:Asource}
\widetilde S[t^pP_N^{(k)}]=N(N-2)\mathsf s_k s^kt^pQ_A D^{-(N+2)/2},
\end{equation}
where $D=s+(t+1)^2$ in the scalar-flat problem and $D=1+s+t^2$ in the minimal-boundary problem.

For the divergence-free polynomial family, the first translated source has only angular degree three.
More explicitly, if
\[
C_a(\bar x):=x_aQ_A-\frac2{m+2}s(A\bar x)_a,\qquad
u_k^{\mathrm{trans}}:=2\left\{k\mathsf s_k-b_k-e_k+\frac{2b_k+e_k}m\right\},
\]
then
\begin{equation}\label{a:eq:sourcefirst}
x_bx_c\pa_a(P_N^{(k)})_{bc}=u_k^{\mathrm{trans}}s^{k-1}C_a(\bar x),\qquad \Delta_{\bar x}C_a=0.
\end{equation}
For $k=0$ the expression is zero.
Direct differentiation followed by the three identities in Lemma~\ref{a:lem:Achain} proves \eqref{a:eq:sourcefirst}; those identities cancel the possible angular-degree-one term.

Only the degree-two projection of $S^{(aa)}$ pairs with the centered source.
For a coordinate eigenvector of the normalized $A$, set
\begin{align*}
\mathsf r_k&:=b_k+e_k-\frac{2b_k+e_k}m,\\
T_{0k}&:=2k\mathsf s_k-4\mathsf r_k-\frac{2(2b_k+e_k)}m,\\
T_{xk}&:=4k(k-1)\mathsf s_k-8(k-1)\mathsf r_k+2e_k,\\
T_{vk}&:=8k\mathsf s_k-8k(1+b_k)-8\mathsf r_k+4b_k.
\end{align*}
On $|\bar x|=1$, the coefficient of $Q_A$ in the degree-two harmonic projection of $x_bx_c\pa_{aa}(P_N^{(k)})_{bc}$ is
\begin{equation*}
T_{0k}+\frac{T_{xk}}{m+4}
+(A^2)_{aa}\left(\frac{4T_{xk}}{m+4}+T_{vk}\right).
\end{equation*}
The angular norms used with \eqref{a:eq:Asource} and \eqref{a:eq:sourcefirst} are
\begin{equation}
\av Q_A^2=q_A,\qquad
\av C_a^2=\frac{q_A}{m+4}\left(1+\frac{2m}{m+2}(A^2)_{aa}\right).
\end{equation}

For a rank-two unit generating matrix, \eqref{eq:full-scalar-translation-equivariance} has the two eigenvalues
\begin{equation}\label{a:eq:twoeigenvalues}
\alpha(\eps),\qquad \alpha(\eps)+\frac12\gamma(\eps),
\end{equation}
with multiplicities $N-3$ and $2$, respectively.
\end{proof}
\section{Minima for the scalar-flat and minimal-boundary problems}\label{sec:certification}
For each of the four scalar-flat and minimal-boundary profiles, we compute the corrector pairings, assemble the energy and translation matrices, and select the coefficient $\tau_N$ for which $(0,1)$ is a negative nondegenerate minimum.
The computer-assisted reconstruction and sign checks occur in Proposition~\ref{prop:four-scalar-minima}.

Throughout this section $\ell=\log\eps$ is the dilation variable and $\tau$ is a coefficient of the fixed tensor.
We retain the matrices $\mathbf M,\mathbf T_\lambda$ and the pairing $\mathcal M(h_i,h_j)$ from Section~\ref{subsec:polynomial-parameter-derivatives}, with their normalization by $|\Sph^{N-2}|$.
For the sign tests, normalize by the angularly averaged squared $L^2$ norm $M_N:=|\Sph^{N-2}|^{-1}\int_{\R^N_+}\widetilde U^2$:
\begin{equation}
M_N^{\mathrm{II}}=\frac{B((N-1)/2,(N-3)/2)}{2(N-4)},\qquad
M_N^{\mathrm I}=\frac{\Gamma((N-1)/2)\sqrt\pi\Gamma((N-4)/2)}{4\Gamma(N-2)}.
\end{equation}
Both are positive on the ranges considered.
By Proposition~\ref{sec:certificate-normalization}, the exact normalization identities are
\begin{equation}\label{eq:complete-normalization-bridge}
F_h=c_N\widehat F_h\quad(\mathrm{II}),\qquad
F_h=c_N2^{N-2}\widehat F_h\quad(\mathrm I).
\end{equation}
These positive factors and the mass normalization preserve critical points and Hessian inertia.

\subsection{Corrector approximation and exact pairings}
\label{sec:exact-endpoint-pairings}
At the centered unit bubble, retain $\widetilde U=D^{-(N-2)/2}$, $\cB=\cB_{(0,1)}$, and $\cL=\cL_{(0,1)}$ from Section~\ref{sec:algebra}.
For each row $h_i$, the forcing term, exact projected corrector, and explicit approximation are denoted by
\[
S_i:=\widetilde S[h_i],\qquad Z_i:=\cL^{-1}S_i,\qquad X_i\in\Sigma_{(0,1)}.
\]
We use the normalized source and inverse in Proposition~\ref{sec:certificate-normalization}, with the Jacobi form \eqref{eq:common-Jacobi-form}.
The inverse is taken on the full Jacobi kernel complement, including any negative mode.
For the translated sources, use
\[
S_i^{(a)}:=\widetilde S[\pa_a h_i],\quad Z_i^{(a)}:=\cL^{-1}S_i^{(a)},\quad
X_i^{(a)},\qquad
S_i^{(ab)}:=\widetilde S[\pa_a\pa_bh_i],\quad
Z_i^{(ab)}:=\cL^{-1}S_i^{(ab)},\quad X_i^{(ab)}.
\]
Here $X_i^{(a)}$ and $X_i^{(ab)}$ are constructed for the indicated sources; they are not defined by differentiating $X_i$.
The Jacobi operator remains fixed, as in Proposition~\ref{prop:scalar-translated-sources}.

Write $J$ for the interior differential operator and $\mathcal B_\partial$ for the boundary differential operator:
\[
\begin{array}{c|c|c|c}
&D&J\psi&\mathcal B_\partial\psi\ \ (t=0)\\\hline
\mathrm{II}&s+(t+1)^2&-\Delta\psi&-\pa_t\psi-\dfrac{N}{1+s}\psi\\[4pt]
\mathrm I&1+s+t^2&-\Delta\psi-N(N+2)D^{-2}\psi&-\pa_t\psi
\end{array}
\]
The boundary residual is $\mathcal B_\partial X_i$.
Green's formula reads
\[
\cB(v,\psi)=\int_{\R^N_+}(Jv)\psi
+\int_{\partial\R^N_+}(\mathcal B_\partial v)\psi.
\]
For compatible boundary forcing $f$, define $\mathcal R_\perp f:=v|_{\partial\R^N_+}$, where $v\in\Sigma_{(0,1)}$ solves $Jv=0$ and $\mathcal B_\partial v=f$.
All pairings below are assumed finite.

\begin{lemma}
\label{lem:corrector-pairing-identity}
For sources compatible with the Jacobi kernel, let $Z_i:=\cL^{-1}S_i$ and let $X_i$ belong to the same projected space.
Then
\begin{equation}\label{eq:corrector-pairing-identity}
\langle S_i,Z_j\rangle=\langle S_i,X_j\rangle+\langle S_j,X_i\rangle-\cB(X_i,X_j)+\cB(Z_i-X_i,Z_j-X_j).
\end{equation}
In particular,
\[
\langle S_i,Z_i\rangle
=2\langle S_i,X_i\rangle-\cB(X_i,X_i)
+\cB(Z_i-X_i,Z_i-X_i).
\]
If each approximation also satisfies $JX_i=S_i$ in the interior, then
\[
J(Z_i-X_i)=0,\qquad
\mathcal B_\partial(Z_i-X_i)=-\mathcal B_\partial X_i,
\]
and the same identity becomes
\begin{equation}\label{eq:response-reconstruction}
\langle S_i,Z_j\rangle
=\frac12\int_{\R^N_+}(S_iX_j+S_jX_i)
-\frac12\int_{\partial\R^N_+}(X_i\mathcal B_\partial X_j+X_j\mathcal B_\partial X_i)
+\langle\mathcal B_\partial X_i,\mathcal R_\perp(\mathcal B_\partial X_j)\rangle_\partial.
\end{equation}
In this formula the last term is exactly $\cB(Z_i-X_i,Z_j-X_j)$.
The first identity holds for any symmetric invertible weak operator on a projected space; the second uses the displayed Green identity.
Neither requires positivity of $\cB$.
\end{lemma}
\begin{proof}
Expand $\cB(Z_i-X_i,Z_j-X_j)$ and use $\cB(Z_i,\psi)=\langle S_i,\psi\rangle$ and symmetry of $\cB$.
This proves \eqref{eq:corrector-pairing-identity} and its diagonal case.
If $JX_i=S_i$, the equations for $Z_i-X_i$ imply
\[
(Z_i-X_i)|_\partial=-\mathcal R_\perp(\mathcal B_\partial X_i).
\]
Applying Green's formula to the two errors, we obtain
\[
\cB(Z_i-X_i,Z_j-X_j)
=\langle\mathcal B_\partial X_i,
\mathcal R_\perp(\mathcal B_\partial X_j)\rangle_\partial.
\]
Applying it to $X_i,X_j$ and averaging the two orders yields
\[
\cB(X_i,X_j)=\frac12\int(S_iX_j+S_jX_i)
+\frac12\int_\partial
(X_i\mathcal B_\partial X_j+X_j\mathcal B_\partial X_i).
\]
Substitution into \eqref{eq:corrector-pairing-identity} proves \eqref{eq:response-reconstruction}.
The integrations by parts follow first with cutoffs and then by the integrability estimates of Lemma~\ref{lem:polynomial-integrability-scaling}.
\end{proof}

By \eqref{eq:response-reconstruction}, the scalar-flat and minimal-boundary calculations reduce to constructing $X$, evaluating the interior and boundary integrals, and computing the boundary response.
The construction of $X$ requires a finite-dimensional solvability statement.
Following the formulation of Khuri--Marques--Schoen~\cite[Proposition~4.1]{KhuriMarquesSchoen}, we specify an invariant polynomial space and determine the kernel and range of the operator on that space.

Let $H_q(\bar x)$ be a nonzero homogeneous harmonic polynomial of degree $q$.
Section~\ref{subsec:polynomial-parameter-derivatives} writes each source component in the form
\[
S=\frac{H_q\widetilde U}{D^2}V(D,y),\qquad
X:=\frac{H_q\widetilde U}{D}R(D,y),\qquad
y:=\begin{cases}\widehat t:=t+1,&\mathrm{II},\\t,&\mathrm I.\end{cases}
\]
Here $V$ is known from the source and $R$ is the unknown polynomial factor of $X$.
Regard $D,y$ as independent polynomial variables and set
\[
\mathcal P_d:=\operatorname{span}\{D^ay^b:a,b\in\mathbb Z_{\ge0},\ 2a+b\le d\}.
\]
Thus $D$ has weight two and $y$ has weight one.
Define the induced polynomial operator $\mathcal T_q$ by
\[
J\left(\frac{H_q\widetilde U}{D}R\right)
=:\frac{H_q\widetilde U}{D^2}\mathcal T_qR.
\]

\begin{lemma}
\label{lem:polynomial-corrector-solvability}
Let $0\le d<N$ be an integer.
Each operator in \eqref{eq:II-polynomial-operator} and \eqref{eq:I-polynomial-operator} preserves $\mathcal P_d$.
For $q=2,3$, every $V\in\mathcal P_d$ has a unique solution
\[
R\in\mathcal P_d,\qquad \mathcal T_qR=V.
\]
For $q=1$, the kernel consists exactly of the constants.
A solution exists, uniquely modulo constants, if and only if
\[
\begin{cases}
V(0,0)=0,&\mathrm{II},\\[3pt]
\displaystyle\int_{\mathbb R^N}H_1(\bar x)^2
(1+|x|^2)^{-N-1}V(1+|x|^2,t)\dd x=0,&\mathrm I.
\end{cases}
\]
Here the first condition is evaluation in the independent variables $(D,y)$; the second uses the polynomial extension to $t\in\mathbb R$.
The degree-one source components in \eqref{eq:full-scalar-first-source} satisfy the corresponding compatibility condition.
Under $N>2d_*+2$, this construction therefore produces an interior solution $X$ for every component needed for the energy and translation Hessian for these equations.
Its Jacobi projection belongs to $\Sigma_{(0,1)}$ and still satisfies $JX=S$.
\end{lemma}
\begin{proof}
In Type II, differentiating with $D=s+\widehat t^2$, we obtain
\begin{equation}\label{eq:II-polynomial-operator}
\mathcal T_qR=-4D^2R_{DD}+(2N-4q)DR_D-4\widehat tDR_{D\widehat t}
+2N\widehat tR_{\widehat t}-DR_{\widehat t\widehat t}+2N(q-1)R.
\end{equation}
In minimal Type I, the same computation with $D=1+s+t^2$ reads
\begin{equation}\label{eq:I-polynomial-operator}
\mathcal T_qR=-4D(D-1)R_{DD}+((2N-4q)D-4N)R_D
-4tDR_{Dt}+2NtR_t-DR_{tt}+2N(q-1)R.
\end{equation}
In particular, its action on a monomial is
\begin{align*}
\mathcal T_q(D^ay^b)={}&2(N-2a)(a+b+q-1)D^ay^b
-b(b-1)D^{a+1}y^{b-2}\\
&+\begin{cases}0,&\mathrm{II},\\
4a(a-1-N)D^{a-1}y^b,&\mathrm I.
\end{cases}
\end{align*}
Terms with zero coefficient are omitted.
All terms have weight at most $2a+b$, so $\mathcal P_d$ is invariant.
Order its monomials by decreasing $2a+b$, and then by decreasing $b$.
The matrix is square and triangular, with diagonal entries
\[
2(N-2a)(a+b+q-1),\qquad 2a+b\le d<N.
\]
They are positive for $q=2,3$.
Thus the matrix is invertible, and successive coefficient matching constructs the unique $R$.
For $q=1$ only the constant monomial has zero diagonal entry.
Since $\mathcal T_1(1)=0$, the kernel is exactly one-dimensional and the range has codimension one.

In Type II, no nonconstant monomial contributes a constant term to $\mathcal T_1R$.
Its range is therefore exactly the polynomials with zero constant coefficient, by the codimension just proved.
In minimal Type I, write the integral in the statement as $\Lambda(V)$ and put
\[
\Phi:=H_1(\bar x)(1+|x|^2)^{-N/2},\qquad J\Phi=0
\quad\text{on }\mathbb R^N.
\]
For every $R\in\mathcal P_d$, Green's formula on balls implies
\[
\Lambda(\mathcal T_1R)
=\int_{\mathbb R^N}\Phi J(\Phi R)\dd x
=\int_{\mathbb R^N}(J\Phi)\Phi R\dd x=0.
\]
The integrals converge and the boundary term at radius $r$ is $O(r^{d-N})$, which tends to zero as $r\to\infty$.
Since $\Lambda(1)>0$, the range is precisely $\ker\Lambda$.
This proves both solvability criteria.
Adding a constant to $R$ adds a multiple of the translation Jacobi field $H_1\widetilde U/D$ to $X$.

We next verify compatibility for the degree-one sources.
For a row $h:=t^pS_N^{(k)}(b,e)$ with $k\ge1$, \eqref{eq:full-scalar-first-source} shows that its degree-one source polynomial has the form
\[
V=t^p\bigl(c_0s^k+c_1Ds^{k-1}+c_2D^2s^{k-2}\bigr),
\]
with constants $c_0,c_1,c_2$ determined by that formula and any term with a negative power of $s$ absent.
The row $k=0$ has zero tangential derivative.
In Type II, substitute $s=D-y^2$ and $t=y-1$.
Every displayed term has zero constant coefficient in $(D,y)$, so $V(0,0)=0$.
In minimal Type I, substitute $s=D-1-t^2$.
If $p$ is odd, the integral defining $\Lambda(V)$ vanishes by parity.
If $p$ is even, it is twice the half-space pairing of this source component with $H_1\widetilde U/D$, and vanishes by Lemma~\ref{lem:source-kernel-orthogonality}.
Angular orthogonality removes the other harmonic components.
This argument applies to each normal power separately, and then to their sum.

The centered sources have $q=2$; first translated sources have $q=1,3$; only $q=2$ from the second translated sources pairs with a centered source.
For a row of degree $d_i$, their polynomial factors have weighted degree at most $d_i$ in the centered and degree-one cases, and at most $d_i-2$ in the other two cases, as follows from
\eqref{eq:full-scalar-centered-source}, \eqref{eq:full-scalar-first-source}, and \eqref{eq:full-scalar-second-source-projection}.
The same bounds hold for $R$ by invariance and uniqueness modulo constants.
Hence all the constructed approximations satisfy
\[
|\pa^jX(x)|\le C(1+|x|)^{d_*+2-N-j},\qquad j=0,1.
\]
Finite energy follows from $N>2d_*+2$.
Subtracting the Jacobi component places $X$ in $\Sigma_{(0,1)}$.
This subtraction changes neither the interior equation nor the boundary residual, since every Jacobi field satisfies both homogeneous equations.

Finally, differentiating and then setting $t=0$ determines the residual in Lemma~\ref{lem:corrector-pairing-identity}:
\[
\mathcal B_\partial X=-\frac{H_q\widetilde U}{D}
\begin{cases}
2R_D(D,1)+R_{\widehat t}(D,1),&\mathrm{II},\\
R_t(D,0),&\mathrm I.
\end{cases}
\]
In Type II the derivative of the bubble factor cancels the Robin term; in minimal Type I, $D_t=0$ at the boundary.
For example, for the Type-II row $h=tA$, we have
\[
S=N(N-2)tQ_A D^{-(N+2)/2},\qquad
R(D,\widehat t)=\frac{N-2}{4}(\widehat t-2),
\]
and therefore
\[
X=\frac{N-2}{4}(t-1)Q_A D^{-N/2},\qquad
\mathcal B_\partial X=-\frac{N-2}{4}Q_A(1+s)^{-N/2}.
\]
Here $JX=S$ and the harmonic degree is two, so no Jacobi projection is needed.
The nonzero residual is precisely the remaining boundary error.
\end{proof}

To evaluate the boundary term in Lemma~\ref{lem:corrector-pairing-identity}, normalize the harmonic factor by $\av H_q^2=1$.
The approximations just constructed have boundary residuals of the form
\[
\mathcal B_\partial X=-\frac{H_q\widetilde U}{D}P(D)
\quad\text{on }t=0,
\]
where $P$ is a polynomial and $D=1+s$ on the boundary in both equations.
Thus bilinearity reduces their response pairings to the monomials in $P$.
For nonnegative integers $a,b$, define
\begin{equation}\label{eq:boundary-response-definition}
\mathcal R_q(a,b):=\frac{
\langle H_q\widetilde U D^{a-1},
\mathcal R_\perp(H_q\widetilde U D^{b-1})\rangle_\partial}
{|\Sph^{N-2}|M_N}.
\end{equation}
This is the boundary-response contribution of two monomial residuals, after removing their common angular norm and the mass factor $M_N$.
Rotational symmetry makes it independent of the choice of normalized $H_q$.
For $q=1$, the individual entries are interpreted spectrally with the Jacobi mode omitted; the full residuals used in the reconstruction are compatible with that mode.
Orthogonality separates distinct harmonic degrees.

With $X$ known, the products $S_iX_j$ and $X_i\mathcal B_\partial X_j$ in the identity are finite sums of monomials times a power of $D$.
After the angular integration from \eqref{eq:common-spherical-moment}, the required interior integrals are
\begin{equation}\label{a:eq:moments}
M_{R,p}(\nu):=\frac12\int_0^\infty\!\int_0^\infty
s^{(m+R)/2-1}t^pD^{-\nu}\dd s\dd t.
\end{equation}
Here $R$ in the subscript is a radial exponent, and $p$ is a normal exponent.
These are unnormalized moments; any division by $M_N$ is written explicitly.
They also evaluate all four raw terms in $\mathcal M(h_i,h_j)$ from \eqref{eq:coefficient-matrix-pairing}.

\begin{lemma}\label{lem:endpoint-moments}
Put $a:=(m+R)/2$ and $b:=(p+1)/2$.
If $a>0$, $b>0$, and $\nu>a+b$, then
\begin{equation}\label{a:eq:momentII}
M_{R,p}(\nu)=
\begin{cases}
\frac12 B(a,\nu-a)B(p+1,2\nu-2a-p-1),&\mathrm{II},\\[3pt]
\dfrac{\Gamma(a)\Gamma(b)\Gamma(\nu-a-b)}{4\Gamma(\nu)},&\mathrm I.
\end{cases}
\end{equation}
For the boundary moment associated with \eqref{a:eq:moments}, it suffices that $a>0$ and $\nu>a$; then
\begin{equation}\label{a:eq:boundarymoment}
\frac12\int_0^\infty s^{a-1}(1+s)^{-\nu}\dd s
=\frac12B(a,\nu-a).
\end{equation}
\end{lemma}
\begin{proof}
In the scalar-flat case, substitute $s=(1+t)^2u$ in the inner integral of \eqref{a:eq:moments}.
Its value is $\frac12B(a,\nu-a)(1+t)^{2a-2\nu}$; the remaining integral in $t$ is the second beta factor.
In the minimal case, substitute $s=(1+t^2)u$ and then $v=t^2$.
The two beta factors simplify to the displayed gamma quotient.
Formula~\eqref{a:eq:boundarymoment} is the defining beta integral.
The stated inequalities are exactly the convergence conditions.
\end{proof}

For the sign calculations below, we normalize the minimal Type-I moments as follows.
The distinction between even and odd normal powers introduces the gamma ratio
\begin{equation}\label{eq:minimal-gamma-parameter}
w_N:=\frac{\Gamma((N-5)/2)}{\sqrt\pi\Gamma((N-4)/2)}.
\end{equation}
Let $k,r$ be nonnegative integers and write $\nu:=N-2+u$, where $u$ is an integer exponent shift for which the indicated moments converge.
Divide \eqref{a:eq:momentII} by $M_N^{\mathrm I}$ and use the gamma recurrence.
For $R=2k$ and $p=2r$ or $2r+1$, we obtain
\begin{equation}\label{eq:minimal-normalized-moments}
\begin{aligned}
\frac{M_{2k,2r}(N-2+u)}{M_N^{\mathrm I}}
&=\frac{(m/2)_k(1/2)_r((N-4)/2)_{u-k-r}}{(N-2)_u},\\
\frac{M_{2k,2r+1}(N-2+u)}{M_N^{\mathrm I}}
&=\frac{(m/2)_k r!((N-5)/2)_{u-k-r}}{(N-2)_u}w_N.
\end{aligned}
\end{equation}
Thus the normalized even moments are rational in $N$, whereas the odd moments have one factor $w_N$.
By applying the same recurrence to \eqref{a:eq:boundarymoment}, we find
\[
\frac{B(m/2+k,N-2+u-m/2-k)}{2M_N^{\mathrm I}}
=\frac{(N-5)w_N(m/2)_k(m/2-1)_{u-k}}{(N-2)_u}.
\]
Negative Pochhammer indices are interpreted by the gamma quotient on the convergent range.
All moments used below satisfy the convergence inequalities in Lemma~\ref{lem:endpoint-moments}.

It remains to evaluate the last term of the corrector identity, namely the responses $\mathcal R_q$ defined in \eqref{eq:boundary-response-definition}.
Compactification transforms the boundaries of the two compactified models into the same round sphere.
Its harmonics of total degree $j=n+q$ diagonalize the boundary response.
In Type II the coefficient is an inverse Robin eigenvalue; in minimal Type I it is the ratio of boundary value to outward normal derivative of the regular solution of the homogeneous interior equation.
This explains the only change between the two expansions in the next lemma.

\begin{lemma}
\label{lem:endpoint-boundary-response}
Let $q\in\{1,2,3\}$ and $a,b\ge0$ be integers with $N>\max\{4,2(a+b+q)\}$.
In Type II,
\begin{equation}\label{eq:II-response-series}
\mathcal R_q(a,b)=K_q(a,b)
\sum_{\substack{n\ge0\\n+q\ne1}}
\frac{(N-2+2q)_n(a+q)_n(b+q)_n}
{n!(N-1+q-a)_n(N-1+q-b)_n}
\frac{2n+N-2+2q}{N-2+2q}\frac1{n+q-1},
\end{equation}
where the angular overlaps supply the prefactor
\[
K_q(a,b):=\frac{B(m/2-a,m/2+q)B(m/2-b,m/2+q)}
{4M_N B(m/2+q,m/2+q)}.
\]
For minimal Type I, use $M_N=M_N^{\mathrm I}$ and replace $(n+q-1)^{-1}$ by $\beta_{n+q}$, with the response factor on the unit hemisphere
\begin{equation}\label{eq:minimal-response-factor}
\beta_j:=\frac{\Gamma((j-1)/2)\Gamma((j+N)/2)}
{2\Gamma(j/2)\Gamma((j+N+1)/2)},\qquad j\ge2.
\end{equation}
Both projected series converge.
The Type-II series reduces to finite beta sums and is rational in $N$.
For minimal Type I, the proof establishes a polynomial summation criterion for evaluating the even and odd parts exactly; its required instances are verified in Proposition~\ref{prop:four-scalar-minima}.
\end{lemma}
\begin{proof}

We first compute the spherical overlaps.
Under conformal compactification by the normalized bubble, the ball has radius $1/2$.
The residual $-H_q\widetilde U D^{a-1}$ becomes $-H_qD^a$, since a boundary forcing is divided by $\widetilde U^{N/(N-2)}$.
The common minus sign cancels in its quadratic corrector pairing.
Set $v:=s/(1+s)$, $A_*:=N-2+2q$, and set $Y_q:=H_q|_{\Sph^{N-2}}$, normalized by $\fint_{\Sph^{N-2}}Y_q^2=1$.
Then $H_q(\bar x)=s^{q/2}Y_q(\omega)$ for $\bar x\ne0$, with $\omega=\bar x/|\bar x|$.
A boundary harmonic basis is
\[
\psi_{n,q}:=[4v(1-v)]^{q/2}C_n^{A_*/2}(2v-1)Y_q(\omega).
\]
The radial boundary measure is $\tfrac12v^{m/2-1}(1-v)^{m/2-1}\,dv$ after dividing by $|\Sph^{N-2}|$.
Let $I_{n,a}$ be the overlap of $H_qD^a$ with $\psi_{n,q}$ and $\norm{\psi_{n,q}}^2$ the squared norm of $\psi_{n,q}$ in this measure.
By the Rodrigues formula and beta integration,
\begin{align}\label{eq:boundary-harmonic-overlaps}
I_{n,a}
&=2^{q-1}B(m/2-a,m/2+q)
\frac{(A_*)_n(a+q)_n}{n!(N-1+q-a)_n},\\
\norm{\psi_{n,q}}^2
&=2^{2q-1}B(m/2+q,m/2+q)
\frac{(A_*)_n}{n!}\frac{A_*}{2n+A_*}.\notag
\end{align}
The inverse Robin eigenvalue on this radius-$1/2$ ball is $1/[2(n+q-1)]$.
Thus
\[
\mathcal R_q(a,b)=M_N^{-1}
\sum_{\substack{n\ge0\\n+q\ne1}}
\frac{I_{n,a}I_{n,b}}{2(n+q-1)\norm{\psi_{n,q}}^2}.
\]
This is \eqref{eq:II-response-series}: the overlap quotient supplies one factor $1/2$, and the inverse radius factor supplies the other.
The term $n+q=1$ is omitted by projection.
The series converges when $N>2(a+b+q)$; every entry used at its stated dimension satisfies this strict condition.

We next evaluate the Type-II series by finite sums.
Set $C_a:=N-1+q-a$, $k:=a+q-1$, $B_*:=b+q$, and $C_b:=N-1+q-b$.
Then
\[
\frac{(A_*)_n}{(C_a)_n}=\frac{(C_a+n)_k}{(C_a)_k},\qquad
\frac{(a+q)_n}{n!}=\frac{(n+1)_k}{k!}.
\]
For $q\ge2$, cancel $n+q-1$, expand the remaining polynomial in falling factorials, and use
\begin{equation}
\sum_{n\ge0}n^{\underline r}\frac{(B_*)_n}{(C_b)_n}
=\frac{r!(B_*)_r(C_b-1)}{(C_b-B_*-r-1)_{r+1}}.
\end{equation}
To prove this identity, represent the ratio by a beta integral and sum $\sum n^{\underline r}x^n=r!x^r/(1-x)^{r+1}$ inside it.
For $q=1$, remove $n=0$ first.
Polynomial division leaves, in addition,
\[
\sum_{n\ge1}\frac{(B_*)_n}{(C_b)_n n}
=\psi(C_b)-\psi(C_b-B_*)=\sum_{r=1}^{B_*}\frac1{C_b-r}.
\]
Here $\psi:=\Gamma^{\prime}/\Gamma$ is the logarithmic derivative of the gamma function.
The same beta integral proves the first equality, and $B_*$ is an integer.
Thus all required scalar-flat corrector pairings are rational in $N$.
For example
\begin{equation}
\mathcal R_2(0,0)=\frac{N(N^2-2N-5)}{4(N-2)(N-1)}.
\end{equation}

For minimal Type I, the separated solution regular at the pole has radial factor
\[
(\sin\theta)^j\,{}_2F_1\left(\frac{j-1}2,\frac{j+N}2;
j+\frac N2;\sin^2\theta\right).
\]
By the connection formula at $\sin^2\theta=1$, the ratio of its boundary value to its outward derivative is \eqref{eq:minimal-response-factor}.
This is the response factor on the unit hemisphere.
The normalized minimal bubble compactifies to the hemisphere of radius $1/2$, where the value-to-normal-derivative ratio is $\beta_j/2$.
Its boundary measure and overlaps are the same as \eqref{eq:boundary-harmonic-overlaps}.
Thus the combined prefactor is again $1/(4M_N)$.

The spectral expansion therefore proves the stated substitution in \eqref{eq:II-response-series}, with the Jacobi mode removed by projection.

To sum this minimal-boundary series, use $A_*=N-2+2q$, $C_a=N-1+q-a$, and $C_b=N-1+q-b$.
For a minimal-boundary pairing, write $n:=2k+n_0$, where the starting indices are $n_0=2,1$ for $q=1$ and $n_0=0,1$ for $q=2,3$.
Let $a_k$ be the summand without $K_q(a,b)$.
By the gamma recurrence, its successive-term ratio is $\varrho(k)$:
\begin{equation}
\begin{aligned}
\varrho(k):=\frac{a_{k+1}}{a_k}
={}&\frac{(n+A_*)(n+A_*+1)(n+a+q)(n+a+q+1)(n+b+q)(n+b+q+1)}
{(n+1)(n+2)(n+C_a)(n+C_a+1)(n+C_b)(n+C_b+1)}\\
&\times\frac{2n+A_*+4}{2n+A_*}
\frac{(n+q-1)(n+q+N)}{(n+q)(n+q+N+1)}.
\end{aligned}
\end{equation}
To find a telescoping expression for the sum, factor this rational ratio as $A_0(k)C_0(k+1)/(B_0(k)C_0(k))$, put $B_1(k):=B_0(k-1)$, and seek $P_*(k)\in\Q(N)[k]$ satisfying
\begin{equation}\label{eq:minimal-polynomial-telescoper}
A_0(k)P_*(k+1)-B_1(k)P_*(k)=C_0(k).
\end{equation}
Given such a polynomial, put $r(k):=-B_1(k)P_*(k)/C_0(k)$.
Substituting this expression and the factorization into \eqref{eq:minimal-polynomial-telescoper}, we obtain
\begin{equation}\label{eq:minimal-telescoping-identity}
a_kr(k)-a_{k+1}r(k+1)=a_k.
\end{equation}
By the gamma-ratio asymptotic, $a_k=O(k^{2(a+b+q)-N-1})$ as $k\to\infty$.
Whenever $r$ is defined for $k\ge0$ and satisfies $r(k)=O(k)$, the strict condition $N>2(a+b+q)$ therefore implies $a_kr(k)\to0$ as $k\to\infty$.
Summing \eqref{eq:minimal-telescoping-identity} and using this limit, we obtain
\[
\sum_{k=0}^{\infty}a_k=a_0r(0).
\]
This is the polynomial summation criterion used in the corrector-pairing calculation.
For $q=2,a=b=0$, the two rational functions are
\[
r_{\rm even}(k):=\frac{(N+2k-1)^2}{(N-4)(N+4k+2)},\qquad
r_{\rm odd}(k):=\frac{(N+2k)^2}{(N-4)(N+4k+4)}.\qedhere
\]

\end{proof}

The case $q=2$, $a=b=0$ also determines the contribution of $h=tA$ to the reduced energy.
Fix the minimal-boundary equation and $N\ge15$.
For $q=2$ and $a=b=0$, the prefactor and the first term in each parity are
\begin{equation}
\begin{aligned}
K_2(0,0)&=
\frac{(N-5)(N-3)(N+1)(N+2)}{4N(N-1)(N-2)}w_N,\\
a_0^{\rm even}&=\beta_2,\qquad
a_0^{\rm odd}=\frac{4(N+4)}{(N+1)^2}\beta_3,\\
\beta_2&=
\frac{N(N-2)(N-4)}{(N-5)(N-3)(N-1)(N+1)w_N},\\
\beta_3&=
\frac{(N-5)(N-3)(N-1)(N+1)}{N(N-2)(N-4)(N+2)}w_N.
\end{aligned}
\end{equation}
These formulas follow from \eqref{eq:II-response-series}, \eqref{eq:minimal-response-factor}, and the gamma recurrence.
Multiplying by $r_{\rm even}(0)$ and $r_{\rm odd}(0)$, respectively, we obtain
\begin{equation}
\begin{aligned}
\mathcal R_2(0,0)
&=K_2(0,0)\bigl(a_0^{\rm even}r_{\rm even}(0)
+a_0^{\rm odd}r_{\rm odd}(0)\bigr)\\
&=\frac14+
\frac{(N-5)^2(N-3)^2}{(N-4)^2(N-2)^2}w_N^2.
\end{aligned}
\end{equation}
Thus the even parity contributes the rational term $1/4$, and the odd parity contributes the term containing $w_N^2$.

For the row $h=tA$, with $|A|=1$, Lemma~\ref{lem:polynomial-corrector-solvability} yields
\[
X:=\frac{N-2}{4}tQ_A D^{-N/2},\qquad
\mathcal B_\partial X=-\frac{N-2}{4}Q_A(1+s)^{-N/2}.
\]
Here $X|_\partial=0$, so the local boundary integral in
\eqref{eq:response-reconstruction} vanishes. The raw energy term and
the interior corrector term cancel: regarding $h$ as a one-row profile, the moment formulas imply
\[
\frac{\mathbf M^{\rm raw}_{11}}{M_N}=\frac14,
\qquad
\frac{c_N^{-1}}{|\Sph^{N-2}|M_N}
\int_{\R^N_+}\widetilde S[h]X=\frac14.
\]
The remaining boundary response, with $\av Q_A^2=2/[(N-1)(N+1)]$, therefore yields
\begin{equation}\label{eq:minimal-response-worked-energy}
\frac{\widehat F_{tA}(0,1)}{|\Sph^{N-2}|M_N}
=-\frac{N-2}{2(N+1)}\mathcal R_2(0,0).
\end{equation}
Since $h$ has degree one, its contribution to the leading scale coefficient $a_{\mathrm{I,nu}}$ in Proposition~\ref{prop:four-scalar-minima} is twice \eqref{eq:minimal-response-worked-energy}.

\subsection{Matrix assembly and parameter dependence}
\label{sec:endpoint-matrix-assembly}
We insert the polynomial correctors, the moments of Lemma~\ref{lem:endpoint-moments}, and the boundary responses of Lemma~\ref{lem:endpoint-boundary-response} into \eqref{eq:coefficient-matrix-pairing} to compute the energy and translation matrices.

\begin{proposition}
\label{prop:endpoint-matrix-assembly}
Fix one of the four families with $|A|=1$ and homogeneous rows $h_i$.
The moment and response calculations above determine the entries
\[
\mathbf M_{ij}=\mathbf M^{\rm raw}_{ij}
-\frac{c_N^{-1}}{|\Sph^{N-2}|}\langle S_i,Z_j\rangle
\]
by the finite coefficient formula \eqref{eq:endpoint-matrix-explicit}.
The same reconstruction for the first and second translated rows determines $\mathbf T_0,\mathbf T_{1/2}$.
These matrices determine $\widehat F_h(0,\eps)$ and $\alpha(\eps),\gamma(\eps)$ through \eqref{a:eq:scalingmatrix} and \eqref{eq:explicit-alpha-gamma}.
\end{proposition}
\begin{proof}
We carry out the reconstruction in six steps, following the three terms of Lemma~\ref{lem:corrector-pairing-identity}.
\begin{enumerate}
\item Form the nonzero homogeneous rows
\[
h_i:=t^{p_i}S_N^{(k_i)}(b_{k_ip_i},e_{k_ip_i}),\qquad d_i:=2k_i+p_i,
\]
and keep their overall coefficients $c_{k_ip_i}$ in $\mathbf c$.
Use the assigned generating matrix for $\mathbf M$.
For the two reference translation matrices, retain all row coefficients and use
\[
A=A_{\mathrm r}=\frac1{\sqrt2}\diag(1,-1,0,\ldots,0),\qquad
A_{\mathrm r}^2=\diag(1/2,1/2,0,\ldots,0).
\]
The matrices $\mathbf T_0,\mathbf T_{1/2}$ correspond to the directions $\mathbf e_3,\mathbf e_1$, respectively.
Form $\partial_a h_i,\partial_{aa}h_i$ for $a=3,1$.

\item Evaluate the raw terms of \eqref{eq:coefficient-matrix-pairing} using the angular integrals \eqref{eq:common-spherical-moment} and the moments of Lemma~\ref{lem:endpoint-moments}.
For two divergence-free rows $t^pP_N^{(k)}$ and $t^{p'}P_N^{(l)}$, the last two raw terms vanish.
By the contractions in Lemma~\ref{lem:angular-gram},
\begin{equation}
\begin{aligned}
(\mathbf M^{\rm raw})_{(k,p),(l,p')}={}&
\frac{c_N^{-1}(N-2)^2}{2}R_{kl}M_{2(k+l)+2,p+p'}(N)\\
&-\frac14G_{kl}M_{2(k+l)-2,p+p'}(N-2)
-\frac{pp'}{4}C_{kl}M_{2(k+l),p+p'-2}(N-2).
\end{aligned}
\end{equation}
Here $G_{kl}$ comes from tangential differentiation, $pp'C_{kl}$ from normal differentiation, and $(N-2)^2R_{kl}$ from contraction with the bubble gradient.
The umbilic rows retain all four raw terms and use the same angular and radial--normal integration.
The moments are unnormalized: when using \eqref{eq:minimal-normalized-moments}, restore the factor $M_N^{\mathrm I}$.
Omit any term with zero coefficient before evaluating its moment.

\item Compute the centered and translated sources from \eqref{eq:full-scalar-centered-source}, \eqref{eq:full-scalar-first-source}, and \eqref{eq:full-scalar-second-source-projection}.
Apply Lemma~\ref{lem:polynomial-corrector-solvability} to each harmonic component.
In particular, the centered sources have degree two, so
\[
X_i=Q_A D^{-N/2}R_i,\qquad
JX_i=Q_A D^{-(N+2)/2}\mathcal T_2R_i=S_i.
\]
The coefficients of $R_i$ are determined by this equation.
To evaluate the boundary terms, write the residual from that lemma as
\[
\mathcal B_\partial X_i=-Q_A D^{-N/2}P_i(D),\qquad
P_i(D):=
\begin{cases}2R_{i,D}(D,1)+R_{i,\widehat t}(D,1),&\mathrm{II},\\
R_{i,t}(D,0),&\mathrm I.\end{cases}
\]

\item Evaluate the three terms of \eqref{eq:response-reconstruction}.
The first two use ordinary moments, and the last uses $\mathcal R_2$ from Lemma~\ref{lem:endpoint-boundary-response}.
Before writing the coefficient formula, express $R_i$ and $\mathcal T_2R_i$ in the independent variables $(D,t)$; in Type II this means substituting $\widehat t=t+1$ after applying $\mathcal T_2$.
For any polynomial $G$, our coefficient notation is
\[
G(D,t)=\sum_{a,p\ge0}g_{ap}D^at^p
\quad\Rightarrow\quad [D^at^p]G:=g_{ap}.
\]
For a polynomial in $D$ alone, $[D^a]G$ similarly denotes the coefficient of $D^a$.
Thus $a,b$ below are nonnegative integer powers of $D$, while $p$ is a nonnegative integer power of $t$; the row indices remain $i,j$.
In particular, the response contribution is
\[
\sum_{a,b\ge0}([D^a]P_i)([D^b]P_j)\mathcal R_2(a,b).
\]
All sums run only over the nonzero coefficients of the indicated polynomials.
Restriction to $t=0$ leaves $D$ as an independent variable until boundary integration.
Since $\av Q_A^2=q_A=2/[m(m+2)]$, the matrix entry is
\begin{equation}\label{eq:endpoint-matrix-explicit}
\begin{aligned}
\mathbf M_{ij}=\mathbf M^{\rm raw}_{ij}-c_N^{-1}q_A\biggl\{
&\frac12\sum_{a,p\ge0}[D^at^p]
\bigl[(\mathcal T_2R_i)R_j+(\mathcal T_2R_j)R_i\bigr]
M_{4,p}(N+1-a)\\
&+\frac14\sum_{a\ge0}[D^a]\bigl[(R_iP_j+R_jP_i)|_{t=0}\bigr]
B\left(\frac m2+2,N-a-\frac m2-2\right)\\
&+M_N\sum_{a,b\ge0}([D^a]P_i)([D^b]P_j)
\mathcal R_2(a,b)\biggr\}.
\end{aligned}
\end{equation}
All sums are finite.
Integrating the product $S_iX_j$ produces the moment $M_{4,p}(N+1-a)$ in the first line.
The minus sign in $\mathcal B_\partial X_j$ cancels the minus sign of the local boundary term in the identity; polarization and the boundary beta integral each contribute $1/2$, which explains the factor $1/4$ in the second line.
The last line restores the mass factor removed in the definition of $\mathcal R_2$.
Finally, subtracting the inverse pairing from the raw entry proves \eqref{eq:endpoint-matrix-explicit} in the normalization of \eqref{eq:coefficient-matrix-entries}.

\item Apply the same calculation to the translated source pairings in \eqref{eq:full-source-Hessian-pairing}.
The first translated sources have harmonic degrees one and three; only the degree-two component of the second translated source pairs with a centered source.
In degree $q$, replace the three quantities in the centered calculation by
\[
M_{2q,p}(N+1-a),\qquad
B\left(\frac m2+q,N-a-\frac m2-q\right),\qquad
\mathcal R_q(a,b),
\]
and use the squared spherical norm of its angular factor in place of $q_A$.
Writing $A^2=\diag(\lambda_1,\ldots,\lambda_m)$, the relevant norms are
\[
\av (A\bar x)_a^2=\frac{\lambda_a}{m},\qquad
\av C_a^2=\frac{q_A}{m+4}
\left(1+\frac{2m}{m+2}\lambda_a\right),\qquad
\av Q_A^2=q_A.
\]
Orthogonality eliminates pairings between different harmonic degrees.
By retaining both the first/first and second/zeroth source pairings and including their raw terms, we obtain the two reference matrices from \eqref{eq:translation-coefficient-matrices}.

\item Restore the scale by replacing $c_i$ with $\eps^{d_i}c_i$ and multiplying the translation quadratic forms by $\eps^{-2}$.
For $A^2=\diag(\lambda_1,\ldots,\lambda_m)$ and $\mathbf c(\ell)=(\eps^{d_i}c_i)_i$, the formulas in Section~\ref{subsec:polynomial-parameter-derivatives} read
\begin{align*}
\mathbf T_{\lambda_a}
&=(1-2\lambda_a)\mathbf T_0+2\lambda_a\mathbf T_{1/2},\\
\widehat F_h(0,\eps)
&=|\Sph^{N-2}|\mathbf c(\ell)^{\mathsf T}\mathbf M\mathbf c(\ell),\\
\alpha(\eps)&=|\Sph^{N-2}|\eps^{-2}
\mathbf c(\ell)^{\mathsf T}\mathbf T_0\mathbf c(\ell),\\
\gamma(\eps)&=2|\Sph^{N-2}|\eps^{-2}
\mathbf c(\ell)^{\mathsf T}(\mathbf T_{1/2}-\mathbf T_0)\mathbf c(\ell).
\end{align*}
\end{enumerate}
\end{proof}

Define the normalized energy and its first two logarithmic scale derivatives by
\[
E(\tau):=\frac{\widehat F_{h_{N,\tau}}(0,1)}{|\Sph^{N-2}|M_N},\qquad
p(\tau):=\frac{\partial_\ell\widehat F_{h_{N,\tau}}(0,1)}
{|\Sph^{N-2}|M_N},\qquad
K(\tau):=\frac{\partial_\ell^2\widehat F_{h_{N,\tau}}(0,1)}
{|\Sph^{N-2}|M_N}.
\]
Since the coefficient vector has the form $\mathbf c(\tau):=\mathbf c_0+\tau\mathbf c_1$, these are quadratic polynomials in $\tau$.
Write $p(\tau)=:a\tau^2+b\tau+c$; the scale-stationarity equation is $p(\tau)=0$.
Likewise, for a squared eigenvalue $\lambda$ of $A$, define the normalized translation value
\[
T_\lambda(\tau):=M_N^{-1}\mathbf c(\tau)^{\mathsf T}
\mathbf T_\lambda\mathbf c(\tau),\qquad
\mathbf T_\lambda=(1-2\lambda)\mathbf T_0+2\lambda\mathbf T_{1/2}.
\]
The vector $\mathbf c_0$ contains the unselected rows, and $\mathbf c_1$ contains the distinguished row with its defining normalization: $tA$ in the nonumbilic cases and $c_{02}t^2A$ in the umbilic cases.
The different umbilic values of $c_{02}$ are therefore retained in $\mathbf c_1$.
Its support has homogeneous degree $d_0=1$ in the nonumbilic cases and $d_0=2$ in the umbilic cases.
The dependence on $\tau$ and tangential reflection symmetry imply the following properties.

\begin{lemma}
\label{lem:endpoint-hessian-structure}
For each of the four scalar-flat and minimal-boundary families,
\begin{equation}\label{eq:coefficient-degree-identities}
[\tau^2]E(\tau)=\frac{a}{2d_0},\qquad
[\tau^2]K(\tau)=2d_0a.
\end{equation}
The translation gradient and the mixed translation--dilation block vanish at center zero at every scale.
For $A^2=\diag(\lambda_1,\ldots,\lambda_m)$, the normalized translation Hessian at unit scale has the form
\begin{equation}\label{eq:invariant-translation-hessian}
\begin{aligned}
T:=\frac{D_\xi^2\widehat F_h(0,1)}{|\Sph^{N-2}|M_N}
&=T_0I+2(T_{1/2}-T_0)A^2\\
&=\diag(T_{\lambda_1},\ldots,T_{\lambda_m}),\qquad
0\le\lambda_a\le\tfrac12.
\end{aligned}
\end{equation}
For each fixed $\lambda$, $T_\lambda$ is affine in $\tau$.
\end{lemma}
\begin{proof}
By the scaling formulas of Section~\ref{subsec:polynomial-parameter-derivatives}, with $\mathbf D=\diag(d_i)$, we have
\begin{align*}
E(\tau)&=M_N^{-1}\mathbf c(\tau)^{\mathsf T}
  \mathbf M\mathbf c(\tau),\\
p(\tau)&=M_N^{-1}\mathbf c(\tau)^{\mathsf T}
(\mathbf D\mathbf M+\mathbf M\mathbf D)\mathbf c(\tau),\\
K(\tau)&=M_N^{-1}\mathbf c(\tau)^{\mathsf T}
(\mathbf D^2\mathbf M+2\mathbf D\mathbf M\mathbf D
 +\mathbf M\mathbf D^2)\mathbf c(\tau).
\end{align*}
Since $\mathbf D\mathbf c_1=d_0\mathbf c_1$, the definitions imply
\[
a=2d_0M_N^{-1}\mathbf c_1^{\mathsf T}\mathbf M\mathbf c_1,
\qquad
[\tau^2]K(\tau)=4d_0^2M_N^{-1}
\mathbf c_1^{\mathsf T}\mathbf M\mathbf c_1.
\]
These are \eqref{eq:coefficient-degree-identities}.
The pure lowest block is a constant matrix times a normal monomial, so its corrected energy is independent of tangential translation.
The quadratic coefficient in $\tau$ of its translation Hessian therefore vanishes.

Each of the four scalar-flat and minimal-boundary families is even under $\bar x\mapsto-\bar x$.
By reflection invariance of the corrected functional, $F_h(-\xi,\eps)=F_h(\xi,\eps)$, proving both translation stationarity and the zero mixed block used in the coefficient criterion below.
By the affine matrix formula of Proposition~\ref{prop:scalar-translated-sources},
\[
T_{\lambda_a}(\tau)
=(1-2\lambda_a)T_0(\tau)+2\lambda_a T_{1/2}(\tau).
\]
For $A_{\mathrm r}$ and $A_{\mathrm c}$, one has $\lambda_a\in\{0,1/2\}$.
For $A_{\mathrm b}$,
\[
\lambda_1=\cdots=\lambda_{m_+}=\frac{m_-}{m_+m},\qquad
\lambda_{m_++1}=\cdots=\lambda_m=\frac{m_+}{m_-m}.
\]
Both values lie in $[0,1/2]$.
Thus every actual translation eigenvalue is a convex combination of $T_0,T_{1/2}$, also in the minimal umbilic case.

\end{proof}

\subsection{Selection of \texorpdfstring{$\tau$}{tau} and uniform minima}
The next lemma reduces stationarity and the energy and Hessian signs to polynomial inequalities in the coefficients.

\begin{lemma}\label{lem:eleven-signs}
Suppose the translation gradient vanishes at center zero at every scale, and the translation Hessian at unit scale has the form \eqref{eq:invariant-translation-hessian}.
Write $p(\tau)=a\tau^2+b\tau+c$, $\Delta:=b^2-4ac$.
For each $f\in\{-E,K,T_0,T_{1/2}\}$ write $f=:f_0+f_1\tau+f_2\tau^2$ and set
\begin{equation}
A_f:=f_1-\frac{f_2b}{a},\qquad
B_f:=f_0-\frac{f_2c}{a},\qquad
C_f:=A_f b-2aB_f .
\end{equation}
If the eleven quantities
\begin{equation}\label{eq:eleven-signs}
-a,\quad b,\quad\Delta,\quad
A_f,\ C_f\quad(f=-E,K,T_0,T_{1/2})
\end{equation}
are strictly positive, then the larger root
\begin{equation}\label{eq:larger-coefficient-root}
\tau_N:=\frac{-b-\sqrt{\Delta}}{2a}>0
\end{equation}
gives a negative nondegenerate full minimum at $(\xi,\eps)=(0,1)$.
\end{lemma}
\begin{proof}
By the identity $f-(f_2/a)p=A_f\tau+B_f$,
\begin{equation}
f(\tau_N)=\frac{C_f+A_f\sqrt{\Delta}}{-2a}>0.
\end{equation}
Thus $E<0$ and $K,T_0,T_{1/2}>0$ at the selected root.
Each translation eigenvalue is a convex combination of $T_0,T_{1/2}$.
Translation stationarity at every scale implies that the mixed block is zero, and $p(\tau_N)=0$ supplies scale stationarity.
Hence $(0,1)$ is a negative nondegenerate local minimum.
\end{proof}

\begin{proposition}\label{prop:four-scalar-minima}
The coefficient \eqref{eq:larger-coefficient-root} gives a negative nondegenerate full minimum at $(0,1)$ for each of the four scalar-flat and minimal-boundary families in its entire integer dimension range:
\[
\mathrm{II,nu}:N\ge15,\qquad \mathrm{II,u}:N\ge22,\qquad
\mathrm{I,nu}:N\ge15,\qquad \mathrm{I,u}:N\ge21.
\]
The tensor coefficients are held fixed under all bubble variations.
\end{proposition}
\begin{proof}
The symbolic reconstruction and sign certificates described below are the computer-assisted steps of this proposition.
Apply the six steps of Proposition~\ref{prop:endpoint-matrix-assembly} to the defining coefficients.
After division by $M_N$, the resulting matrices determine the five forms $E,p,K,T_0,T_{1/2}$.
The full first and second translated sources are included.
Their symmetry and coefficient dependence are given by Lemma~\ref{lem:endpoint-hessian-structure}.

For the degree-six nonumbilic profiles, the nine symmetric entries have $q=2$, $0\le a\le b\le2$, or $q=3$, $0\le a\le b\le1$.
For the two umbilic profiles, of degrees eight and nine, a common containing set has $q=1,2$, $0\le a\le b\le4$, or $q=3$, $0\le a\le b\le3$.
The largest $2(a+b+q)$ in this set is $20$.
Thus the required convergence conditions hold for $N\ge21$, including the minimal-boundary example; the scalar-flat umbilic profile is used for $N\ge22$.
Entries absent from the actual sources have zero coefficients.

For the minimal boundary pairings, we use the rational functions specified in Appendix~\ref{app:verification}.
The finite checks verify the polynomial summation identity, absence of poles on $k\ge0$, and growth condition used in the proof of Lemma~\ref{lem:endpoint-boundary-response}.
The forty unordered entries ($q=1,2$ with $0\le a\le b\le4$, and $q=3$ with $0\le a\le b\le3$) require eighty polynomial identities, whose coefficients are specified in Appendix~\ref{app:verification}.
Their verified pole and growth conditions, together with the summand decay in Lemma~\ref{lem:endpoint-boundary-response}, justify the terminal limit in \eqref{eq:minimal-telescoping-identity}.
At $N=15$ only the lower degree entries actually present in its profile are evaluated.
The gamma recurrence then puts the normalized parity sums in $\Q(N)+w_N^2\Q(N)$.
Combining them with the normalized bulk moments, we obtain corrected forms in $\Q(N)[\tau,w_N]$, of degree at most two in each variable.
The translation forms are affine in $\tau$.

The leading stationarity coefficients are, respectively,
\begin{align*}
a_{\mathrm{II,nu}}&=-\frac{N^2-2N-2}{2(N-1)(N+1)},\\
a_{\mathrm{II,u}}&=
-\frac{32(N-3)(3N^3-18N^2+21N+14)}
{3(N-6)(N-5)(N-2)^2(N-1)(N+1)},\\
a_{\mathrm{I,nu}}&=-\frac{N-2}{4(N+1)}
-\frac{(N-5)^2(N-3)^2}{(N-4)^2(N-2)(N+1)}w_N^2,\\
a_{\mathrm{I,u}}&=-\frac{4c_{02}^2(N-3)}{3(N-6)(N-2)}.
\end{align*}
For the scalar-flat families every quantity in \eqref{eq:eleven-signs} is rational in $N$.
Its numerator and denominator have nonnegative coefficients and positive constant terms after $N=15+X$ or $N=22+X$, respectively.
For example,
\[
-a_{\mathrm{II,nu}}
=\frac{X^2+28X+193}{2X^2+60X+448}>0
\quad(N=15+X,\ X\ge0).
\]
The other rational identities are specified in Appendix~\ref{app:sign-witnesses}, including a complete representative coefficient-positivity certificate.

For the minimal families the five forms are polynomials in the explicit gamma ratio $w_N$.
Lemma~\ref{lem:gamma-recurrence} bounds that ratio uniformly using its exact recurrence.
After substituting in the eleven quantities and expanding in the Bernstein basis, the coefficients are rational in $y=\sqrt{N-11/2}$.
Their numerator and denominator coefficients are nonnegative after $y=77/25+X$ in the nonumbilic case and $y=393/100+X$ in the umbilic case, with positive constant terms.
Since
\[
(77/25)^2<15-11/2,\qquad (393/100)^2<21-11/2,
\]
these identities cover both complete dimension ranges.
Lemma~\ref{lem:eleven-signs} proves all four assertions.

\end{proof}
\section{Negative mean curvature}\label{sec:TI-well}
For each $N\ge9$, we construct a degree-three family whose corrected energy has a negative nondegenerate full maximum for every sufficiently large $L:=-\kappa>0$.
It produces both umbilic and nonumbilic boundary.
The parameter $L$ is fixed during differentiation in the bubble parameters and along the blow-up sequence.

The restriction on polynomial degree explains the additional harmonic cubic.
At $N=9$, the localization condition $N>2d_*+2$ requires $d_*\le3$.
We first impose the boundary conditions used by our umbilic profile, $h(\bar x,0)=\partial_t h(\bar x,0)=0$.
In an $S$-profile these conditions remove the $p=0,1$ terms.
Since a nonzero row $t^pS_N^{(k)}$ has degree $2k+p$, the remaining rows of degree at most three have $k=0$ and therefore depend only on $t$.
Their corrected energy is independent of the tangential center, so they cannot give a nondegenerate full extremum.

To retain the factor $t^2$ and total degree three, we add the Hessian of a harmonic cubic.
This tensor has tangential degree one, so it introduces the required tangential dependence.

\subsection{The cubic family and whole-space calculation}
Set $m=N-1$ and use the rank-two generating matrix $A:=A_{\mathrm c}=\frac1{\sqrt2}\diag(0,1,-1,0,\ldots,0)$.
Its nonzero entries occupy the second and third coordinates; this ensures the orthogonality to the derivatives of the cubic tensor proved below.
Define
\begin{equation}\label{eq:scalar-cubic-seed}
\begin{aligned}
H_3(\bar x)&:=x_1s-\frac{m+2}{3}x_1^3,\\
B(\bar x)&:=\frac12\nabla_{\bar x}^2H_3
=x_1I+\mathbf e_1\otimes\bar x+\bar x\otimes\mathbf e_1
-(m+2)x_1\mathbf e_1\otimes\mathbf e_1,\\
\beta_N&:=\frac16\sqrt{\frac{N-8}{(N-2)(N+1)}}.
\end{aligned}
\end{equation}
Define the homogeneous tensor components
\[
h_1:=tA,\qquad h_2:=t^2A,\qquad
h_3:=-\frac{t^3}{3L}A+\beta_Nt^2B.
\]
Our common family is
\begin{equation}\label{a:eq:hLmain}
h^-_{N,L,\tau,\theta}
:=h_2+\tau h_3+\theta h_1
=t^2\left[\left(1-\frac{\tau t}{3L}\right)A
+\tau\beta_NB(\bar x)\right]+\theta tA.
\end{equation}
When $N,L$ are fixed, write $h^-_{\tau,\theta}:=h^-_{N,L,\tau,\theta}$; write $h^-$ when the coefficients are fixed as well.
\begin{lemma}
\label{lem:cubic-tensor-geometry}
For $N\ge9$, let $A,H_3,B$ be as above and fix $L>0$ and $\tau,\theta\in\R$.
The tensor $B$ has tangential degree one and satisfies
\begin{equation}
\Delta H_3=0,\quad \tr B=0,\quad \divg B=0,\quad
\Delta B=0,\quad \bar x^{\mathsf T}B\bar x=3H_3.
\end{equation}
With $B_a:=\partial_aB$, its derivative Gram matrix is
\begin{equation}\label{eq:scalar-cubic-Gram}
\langle A,B_a\rangle=0,\qquad
(\langle B_a,B_b\rangle)_{ab}
=\diag(m(m-1),2,\ldots,2).
\end{equation}
The common family \eqref{a:eq:hLmain} is tangential, trace free and divergence free, of degree at most three.
Its boundary jet is
\begin{equation}\label{eq:negative-boundary-jet}
h^-(\bar x,0)=0,\qquad
\partial_t h^-(\bar x,0)=\theta A,\qquad
\pi_{\exp(\mu h^-)}=-\frac{\mu\theta}{2}A,\quad H=0.
\end{equation}
Thus the boundary is totally geodesic when $\theta=0$ and nonumbilic when $\theta\ne0$ for every nonzero metric coefficient $\mu$.
At the same value of $\tau$, the umbilic and nonumbilic profiles are related by
\begin{equation}\label{eq:negative-boundary-class-relation}
h^-_{\tau,\theta}-h^-_{\tau,0}
=\theta tP_N^{(0)}
=\theta tS_N^{(0)}(0,0).
\end{equation}
\end{lemma}
\begin{proof}
The formula for $H_3$ implies $\Delta H_3=0$.
Its Hessian is twice $B$, so its trace and divergence vanish; $\Delta B=0$ follows also from its linearity.
By Euler's identity for a homogeneous cubic, $\bar x^{\mathsf T}B\bar x=3H_3$.
$B_1=\diag(-(m-1),1,\ldots,1)$, and for $a>1$ the only nonzero entries of $B_a$ are $(1,a)$ and $(a,1)$, both equal to one.
These formulas prove the identities for every $m\ge8$.

The factors depending only on $t$ preserve the tangential divergence and trace conditions.
Every term of $h^-$ vanishes at $t=0$, and only $\theta tA$ contributes to its first normal derivative there.
The product-coordinate formula $L_{ab}=-\tfrac12\partial_tg_{ab}$, together with $\tr A=0$, proves \eqref{eq:negative-boundary-jet}.
Finally, \eqref{eq:negative-boundary-class-relation} follows from $P_N^{(0)}=A=S_N^{(0)}(0,0)$.
\end{proof}

Each $B_a$ is a constant symmetric trace-free matrix, so the first translation derivatives use the calculation for $t^2A$ with $A$ replaced by $B_a$.

Recenter with $\widehat t:=t-L$ and set
\[
\Omega_L:=\{(\bar x,\widehat t)\in\R^N:\widehat t\ge-L\},\qquad
D:=1+s+\widehat t^2,\qquad \widetilde U:=D^{-(N-2)/2}.
\]
The rows $h_i$ are still the homogeneous polynomials in the original variable $t=\widehat t+L$.
Scaling is performed in the original half-space coordinate $t$; the shifted coordinate $\widehat t$ is used to evaluate the pairing at the centered unit bubble $z_0=(0,1)$.
Use the notation of Section~\ref{sec:exact-endpoint-pairings} for the interior and boundary operators:
\[
J:=-\Delta-N(N+2)D^{-2},\qquad
\mathcal B_\partial:=-\partial_{\widehat t}+\frac{NL}{D}
\quad\text{on }\partial\Omega_L.
\]
Their Jacobi form and projected weak operator are denoted by $\mathcal B_L$ and $\mathcal L_L$:
\[
\mathcal B_L(v,\psi)
:=\int_{\Omega_L}\bigl(\nabla v\cdot\nabla\psi-N(N+2)D^{-2}v\psi\bigr)
+NL\int_{\partial\Omega_L}D^{-1}v\psi.
\]
The inverse of $\mathcal L_L$ is taken on the complement of the translation and scale Jacobi fields.
For every row, write
\[
S_i:=\widetilde S[h_i]=(h_i)_{ab}\widetilde U_{ab},\qquad
Z_i:=\mathcal L_L^{-1}S_i,\qquad JX_i=S_i.
\]
The source has this simple form because all the rows are divergence free.
The exact corrector $Z_i$ satisfies $\mathcal B_\partial Z_i=0$.
We construct $X_i$ on the whole space and restrict it to $\Omega_L$; its boundary residual need not vanish.
All these sources and constructed correctors have tangential angular degree two or three.
The translation and scale Jacobi fields have degree one and zero, respectively, and the projection weights are tangentially radial.
The restrictions of $X_i$ therefore already satisfy the fixed-bubble orthogonality constraints.
In these angular sectors the projected equation is the displayed interior equation with the actual homogeneous Robin boundary condition.

Retain the positive normalization
\begin{equation}\label{eq:negative-energy-normalization}
\widehat F_h:=c_N^{-1}2^{2-N}F_h,\qquad
F_h=c_N2^{N-2}\widehat F_h.
\end{equation}
The geometric bubble is $2^{(N-2)/2}\widetilde U$, so this is exactly the normalization used for minimal Type I in Section~\ref{sec:algebra}.
Define $\mathcal M_L$ and $\mathbf M_L$ with the same sphere-area factor as \eqref{eq:coefficient-matrix-pairing}:
\begin{equation}\label{eq:negative-corrected-form}
\begin{aligned}
|\Sph^{N-2}|\mathcal M_L(h_i,h_j):={}&
\frac1{2c_N}\int_{\Omega_L}
(h_i\nabla\widetilde U)\cdot(h_j\nabla\widetilde U)
-\frac14\int_{\Omega_L}\langle\partial h_i,\partial h_j\rangle
              \widetilde U^2\\
&-c_N^{-1}\langle S_i,Z_j\rangle,\qquad
(\mathbf M_L)_{ij}:=\mathcal M_L(h_i,h_j).
\end{aligned}
\end{equation}
The first two integrals, divided by $|\Sph^{N-2}|$, give $\mathbf M^{\rm raw}_{L,ij}$.
The divergence terms in \eqref{eq:coefficient-matrix-pairing} vanish.
By homogeneous scaling,
\begin{equation}\label{eq:negative-matrix-scaling}
\widehat F_{h^-_{\tau,\theta}}(0,\eps)
=|\Sph^{N-2}|\mathbf c(\ell)^{\mathsf T}\mathbf M_L\mathbf c(\ell),
\qquad
\mathbf c(\ell):=
\begin{pmatrix}\theta\eps\\\eps^2\\\tau\eps^3\end{pmatrix},
\quad\ell=\log\eps.
\end{equation}
Write $\mathcal M_{\mathbb R^N}$ for the corresponding whole-space pairing and $\mathbf M_{\mathbb R^N}$ for its matrix on the same rows $h_1,h_2,h_3$.
In this pairing the inverse is the whole-space Jacobi inverse on the kernel complement.
The whole-space pairing is evaluated on these same $L$-dependent rows.
Define the positive constant
\[
C_N:=\frac{N-3}{3(N-6)(N-2)}
\int_{\mathbb R^N}\widetilde U^2.
\]

\begin{proposition}
\label{prop:cubic-whole-space-energy}
For $N>8$, the whole-space pairing on the cubic family is
\[
\frac{|\Sph^{N-2}|}{C_N}\mathbf M_{\mathbb R^N}
=-\begin{pmatrix}
0&0&0\\
0&1&-1\\
0&-1&\displaystyle\frac{49}{48}
+\frac{3N-7}{4(N-8)(N-3)L^2}
\end{pmatrix}.
\]
Consequently its centered energy after dilation is
\begin{equation}\label{a:eq:normalfullenergy}
|\Sph^{N-2}|\mathbf c(\ell)^{\mathsf T}
\mathbf M_{\mathbb R^N}\mathbf c(\ell)
=-C_N\left[\eps^4-2\tau\eps^5+
\left(\frac{49}{48}+\frac{3N-7}{4(N-8)(N-3)L^2}\right)
            \tau^2\eps^6\right].
\end{equation}
\end{proposition}
\begin{proof}
We first construct the correctors, then evaluate their pairings and the raw terms.
For the normal monomials in the recentered coordinate, the source equations are
\begin{align*}
\widetilde S[\widehat t^pA]&=N(N-2)\widehat t^pQ_A\widetilde U D^{-2},\\
\widetilde S[\widehat t^pB]&=N(N-2)\widehat t^p(3H_3)\widetilde U D^{-2}.
\end{align*}
Thus their tangential angular degrees are two and three, respectively.
In the notation of Lemma~\ref{lem:polynomial-corrector-solvability}, seek
\[
X_{\widehat t^pA}:=Q_A R_{2,p}\frac{\widetilde U}{D},\qquad
X_{\widehat t^pB}:=3H_3 R_{3,p}\frac{\widetilde U}{D},\qquad
\mathcal T_qR_{q,p}=N(N-2)\widehat t^p.
\]
Here the first index of $R_{q,p}$ is the harmonic degree, and the second is the normal power in the source.
The operator is \eqref{eq:I-polynomial-operator}, with $t$ replaced by $\widehat t$.
Lemma~\ref{lem:polynomial-corrector-solvability} gives unique polynomial solutions of weighted degree at most $p$.
Coefficient matching gives, for $q=2$ and $0\le p\le3$,
\begin{equation}
\begin{aligned}
R_{2,0}&=(N-2)/2,&R_{2,1}&=(N-2)\widehat t/4,\\
R_{2,2}&=[s+(2N-3)\widehat t^2+3]/12,&
R_{2,3}&=[s\widehat t+(N-1)\widehat t^3+2\widehat t]/8,
\end{aligned}
\end{equation}
and, for $q=3$ and $0\le p\le2$,
\begin{equation}
R_{3,0}=(N-2)/4,\qquad
R_{3,1}=(N-2)\widehat t/6,\qquad
R_{3,2}=[s+(3N-5)\widehat t^2+2]/24.
\end{equation}
In these formulas $s=D-1-\widehat t^2$, so they are polynomials in the independent variables $(D,\widehat t)$.
Their correctors are $O(r^{5-N})$ at worst, with gradients $O(r^{4-N})$, where $r:=(s+\widehat t^2)^{1/2}$.
They are smooth and have finite Dirichlet energy for $N>8$.
Their tangential angular degrees make them orthogonal to the whole-space Jacobi kernel, so these are the exact whole-space correctors.

For later use, expand $t=\widehat t+L$ explicitly.
By linearity, the approximations for the actual rows are
\begin{equation}\label{eq:negative-row-approximations}
\begin{aligned}
X_1&:=Q_A\frac{\widetilde U}{D}(R_{2,1}+LR_{2,0}),\\
X_2&:=Q_A\frac{\widetilde U}{D}(R_{2,2}+2LR_{2,1}+L^2R_{2,0}),\\
X_3&:=-\frac{Q_A\widetilde U}{3LD}
(R_{2,3}+3LR_{2,2}+3L^2R_{2,1}+L^3R_{2,0})\\
&\hspace{1cm}+3\beta_NH_3\frac{\widetilde U}{D}
(R_{3,2}+2LR_{3,1}+L^2R_{3,0}).
\end{aligned}
\end{equation}
Each satisfies $JX_i=S_i$ on the whole space and hence on $\Omega_L$.

To evaluate the integrals, put $V_N:=\int_{\mathbb R^N}\widetilde U^2$ within this calculation.
By beta integration,
\begin{equation}
\frac1{V_N}\int_{\mathbb R^N}s^a\widehat t^{2b}
 D^{-(N-2+\nu)}
=\frac{((N-1)/2)_a(1/2)_b((N-4)/2)_{\nu-a-b}}{(N-2)_\nu}.
\end{equation}
Here $a,b,\nu$ are nonnegative integers and the identity holds whenever the integral converges.
Negative Pochhammer indices denote reciprocal finite products.
Odd normal moments vanish.
All moments below converge for $N>8$.
For the cubic tensor, write $\bar x=\sqrt{s}\,\omega$ with $\omega\in\mathbb S^{m-1}$ and average over $\omega$ at fixed $s$:
\begin{align*}
\fint_{\mathbb S^{m-1}}|B|^2
&=\frac{(m-1)(m+2)}m s,&
|\nabla B|^2&=(m-1)(m+2),\\
\fint_{\mathbb S^{m-1}}|B\bar x|^2
&=\frac{2(m-1)}m s^2,&
\fint_{\mathbb S^{m-1}}(\bar x^{\mathsf T}B\bar x)^2
&=\frac{6(m-1)}{m(m+4)}s^3.
\end{align*}
These follow from \eqref{eq:scalar-cubic-seed} and the spherical moments in Lemma~\ref{lem:angular-gram}.
For the matrix family use the same moments with $|A|=1$.

On the whole space, the correction in
\eqref{eq:negative-corrected-form} is $\langle S_i,X_j\rangle$.
Substitution of the polynomial solutions and moments in that formula yields the complete corrected matrices on the indicated ordered bases:
\begin{align}
\frac{|\Sph^{N-2}|}{V_N}\bigl(\mathcal M_{\mathbb R^N}(\widehat t^iA,\widehat t^jA)\bigr)_{i,j=0}^3
&=\diag\left(0,0,-\frac{N-3}{3(N-6)(N-2)},
-\frac{3(3N-7)}{4(N-8)(N-6)(N-2)}\right),
\\
\frac{|\Sph^{N-2}|}{V_N}\bigl(\mathcal M_{\mathbb R^N}(\widehat t^iB,\widehat t^jB)\bigr)_{i,j=0}^2
&=\diag\left(0,0,-\frac{(N-3)(N+1)}{4(N-8)(N-6)}\right).
\end{align}
Cross pairings between the two bases vanish by tangential parity.
Within each basis, the raw terms and the scalar correction cancel for constant and linear normal powers, so those terms introduced by recentering contribute zero to the corrected pairing.

The degree-two normal term contributes $-C_N$.
By the definition of $\beta_N$ and the second matrix,
\begin{equation}
-\frac{|\Sph^{N-2}|\beta_N^2}{C_N}
\mathcal M_{\mathbb R^N}(t^2B,t^2B)=\frac1{48}.
\end{equation}
The remaining contribution to the $(3,3)$ entry, divided by $-C_N$, is
\[
-\frac{|\Sph^{N-2}|}{9C_NL^2}
\mathcal M_{\mathbb R^N}(\widehat t^3A,\widehat t^3A)
=\frac{3N-7}{4(N-8)(N-3)L^2}.
\]
Expanding $t=\widehat t+L$ in the two matrices proves the stated matrix on $h_1,h_2,h_3$.
Inserting $\mathbf c(\ell)$ proves
\eqref{a:eq:normalfullenergy}.
\end{proof}

\subsection{Corrector comparison and matrix coefficients}
\label{sec:cubic-comparison}
We now compare $\mathbf M_L$ with the whole-space matrix.
The restrictions of \eqref{eq:negative-row-approximations} solve the interior equations exactly, and
\[
J(Z_i-X_i)=0,\qquad
\mathcal B_\partial(Z_i-X_i)=-\mathcal B_\partial X_i.
\]
Lemma~\ref{lem:corrector-pairing-identity} therefore applies with $\mathcal B=\mathcal B_L$.
Here we bound its error pairing using uniform coercivity and the boundary trace inequality.

\begin{lemma}
\label{a:lem:uniformnegative}
For $L\ge0$ and a Dirichlet-space function $v$ of tangential harmonic degree $q\ge2$,
\begin{equation}
\mathcal B_L(v,v)\ge\frac{2(q-1)}{N+2q}
 \int_{\Omega_L}|\nabla v|^2.
\end{equation}
If $JX_i=S_i$ and $Z_i=\mathcal L_L^{-1}S_i$ in these angular sectors, then
\[
\|Z_i-X_i\|_{\mathcal B_L}
\le C(N)\|\mathcal B_\partial X_i\|_{L^{2(N-1)/N}(\partial\Omega_L)}.
\]
Consequently,
\[
|\mathcal B_L(Z_i-X_i,Z_j-X_j)|
\le C(N)\|\mathcal B_\partial X_i\|_{L^{2(N-1)/N}}
\|\mathcal B_\partial X_j\|_{L^{2(N-1)/N}}.
\]
The constants are independent of $L$.
The error bounds also hold for finite sums of components of degrees at least two.
\end{lemma}
\begin{proof}
The positive radial function $\phi:=s^{q/2}D^{-N/2}$ satisfies
\[
J_q\phi=2N(q-1)D^{-1}\phi,\qquad
\mathcal B_\partial\phi=0,
\]
where $J_q$ is the radial operator induced by $J$ in angular degree $q$.
By the ground-state identity,
\[
\mathcal B_L(v,v)\ge2N(q-1)\int_{\Omega_L}v^2/D.
\]
Since $D\ge1$ and the Robin term is nonnegative,
\[
\int|\nabla v|^2
\le\mathcal B_L(v,v)+N(N+2)\int v^2/D^2
\le\frac{N+2q}{2(q-1)}\mathcal B_L(v,v).
\]
Apply the identity first away from $s=0$ and then pass to the limit by density.
The estimate is uniform in $L$, including $L=0$.

For the error, Green's formula implies
\[
\mathcal B_L(Z_i-X_i,\psi)
=-\int_{\partial\Omega_L}(\mathcal B_\partial X_i)\psi.
\]
Take $\psi=Z_i-X_i$.
H\"older's inequality and the half-space trace inequality bound the right side by
\[
C(N)\|\mathcal B_\partial X_i\|_{L^{2(N-1)/N}}
\|\nabla(Z_i-X_i)\|_{L^2}.
\]
Coercivity proves the norm estimate, and Cauchy--Schwarz in $\mathcal B_L$ proves the pairing estimate.
For sums of angular components, orthogonality and the lower bound $2(q-1)/(N+2q)\ge2/(N+4)$ give the same conclusion.
\end{proof}

For the umbilic family, only the $(2,3)$ block of $\mathbf M_L$ is needed.
Define its normalized coefficients by
\begin{equation}\label{eq:cubic-finite-coefficients}
a_L:=-\frac{|\Sph^{N-2}|}{C_N}(\mathbf M_L)_{22},\qquad
b_L:=\frac{|\Sph^{N-2}|}{C_N}(\mathbf M_L)_{23},\qquad
c_L:=-\frac{|\Sph^{N-2}|}{C_N}(\mathbf M_L)_{33}.
\end{equation}
Thus the exact finite-half-space energy is
\[
\widehat F_{h^-_{\tau,0}}(0,\eps)
=-C_N(a_L\eps^4-2b_L\tau\eps^5+c_L\tau^2\eps^6).
\]

\begin{lemma}
\label{lem:cubic-half-space-comparison}
For each fixed $N\ge9$, as $L\to\infty$,
\begin{align*}
a_L&=1+O_N(L^{6-N}),\qquad b_L=1+O_N(L^{6-N}),\\
c_L&=\frac{49}{48}+\frac{3N-7}{4(N-8)(N-3)L^2}
+O_N(L^{6-N}+L^{8-N}).
\end{align*}
The comparison with \eqref{a:eq:normalfullenergy} holds uniformly for bounded $\tau$ and $\eps$ in compact subintervals of $(0,\infty)$, also after two derivatives in $\ell=\log\eps$.
All errors tend to zero for fixed $N\ge9$; the slowest rate at $N=9$ is $L^{-1}$.
\end{lemma}
\begin{proof}
We estimate the boundary contribution to each corrector pairing and then the integrals over the omitted region $\widehat t<-L$.
For a polynomial factor $R(D,\widehat t)$, cancellation of the bubble derivative with the Robin term yields
\[
\mathcal B_\partial\left(H_qR\frac{\widetilde U}{D}\right)
=-H_qD^{-N/2}\bigl(2\widehat tR_D+R_{\widehat t}\bigr)
  \big|_{\widehat t=-L}.
\]
In particular, the cubic part of $X_3$ is $\beta_NX_{t^2B}$, where
\[
X_{t^2B}:=3H_3\frac{\widetilde U}{D}
(R_{3,2}+2LR_{3,1}+L^2R_{3,0}).
\]
Its boundary residual is
\begin{equation}
\mathcal B_\partial X_{t^2B}
=-\frac{N-3}{12}L(3H_3)(1+s+L^2)^{-N/2}.
\end{equation}
For the two normal-matrix terms, the same differentiation yields
\begin{align*}
\mathcal B_\partial X_{t^2A}
&=-\frac{N-3}{6}LQ_A(1+s+L^2)^{-N/2},\\
\mathcal B_\partial X_{t^3A}
&=-\frac18\bigl[s+(N-3)L^2+2\bigr]Q_A(1+s+L^2)^{-N/2}.
\end{align*}
Here $X_{t^2A}=X_2$ and $X_{t^3A}$ is obtained from the first term of $X_3$ by removing its coefficient $-1/(3L)$.

For $L\ge1$, the substitution $\bar x=L\zeta$ in these formulas and their boundary traces yields
\begin{align*}
\|\mathcal B_\partial X_{t^2B}\|_{L^{2(N-1)/N}(\partial\Omega_L)}
+\|X_{t^2B}\|_{L^{2(N-1)/(N-2)}(\partial\Omega_L)}
&=O_N(L^{4-N/2}),\\
\|\mathcal B_\partial X\|_{L^{2(N-1)/N}(\partial\Omega_L)}
+\|X\|_{L^{2(N-1)/(N-2)}(\partial\Omega_L)}
&=O_N(L^{3-N/2})
\end{align*}
for $X=X_{t^2A}$ and $X=(3L)^{-1}X_{t^3A}$.
The rescaled traces are dominated by polynomial multiples of $(1+|\zeta|^2)^{-N/2}$; their indicated powers are integrable for $N>8$.

For each of these components, put $Z_h:=\mathcal L_L^{-1}\widetilde S[h]$.
Lemma~\ref{lem:corrector-pairing-identity} gives the diagonal reconstruction
\[
\langle \widetilde S[h],Z_h\rangle-\int_{\Omega_L}\widetilde S[h]X_h
=-\int_{\partial\Omega_L}X_h\mathcal B_\partial X_h
+\mathcal B_L(Z_h-X_h,Z_h-X_h).
\]
The local boundary integral is controlled by the two trace norms.
Lemma~\ref{a:lem:uniformnegative} controls the error pairing by the square of the residual norm.
Thus the total boundary correction is $O_N(L^{8-N})$ for $t^2B$ and $O_N(L^{6-N})$ for the normal-matrix terms.
The bilinear identity gives the same bounds for cross terms within each angular sector; the two sectors are orthogonal.

It remains to compare the interior integrals with their whole-space values.
On $\widehat t<-L$, put $r=(s+\widehat t^2)^{1/2}\ge L$.
The cubic tensor is $O(r^3)$, its gradient $O(r^2)$, its source $O(r^{3-N})$, and its corrector $O(r^{5-N})$.
Hence each omitted quadratic integral is bounded by
\[
C(N)\int_L^\infty r^{7-N}\dd r=O_N(L^{8-N}).
\]
For $t^2A$ and $t^3A/(3L)$, the corresponding bounds have an additional factor $L^{-1}$, since $r\ge L$; their quadratic tails are therefore $O_N(L^{6-N})$.

Combining the boundary and tail estimates with Proposition~\ref{prop:cubic-whole-space-energy} proves the three coefficient estimates in \eqref{eq:cubic-finite-coefficients}.
Finally, \eqref{eq:negative-matrix-scaling} puts all scale dependence in the polynomial vector $\mathbf c(\ell)$, with the operator and domain fixed.
The same estimates therefore hold after two scale derivatives on the stated compact sets.
\end{proof}

The translation calculation uses the derivative matrices $B_a:=\partial_aB$ from \eqref{eq:scalar-cubic-Gram}.
Write their Gram matrix as
\[
G^B:=(\langle B_a,B_b\rangle)_{a,b=1}^m
=\diag(m(m-1),2,\ldots,2).
\]
Here $a,b$ are tangential coordinate indices; the indices $i,j$ of $h_i$ and $(\mathbf M_L)_{ij}$ continue to label the three tensor rows.
\begin{lemma}
\label{lem:cubic-translation-block}
For fixed $N\ge9$, $L>0$, and $\tau,\theta\in\mathbb R$, the family
\eqref{a:eq:hLmain} has the exact translation dependence
\begin{equation}\label{eq:cubic-exact-translation-energy}
\widehat F_{h^-_{\tau,\theta}}(\xi,\eps)
=\widehat F_{h^-_{\tau,\theta}}(0,\eps)
-C_Na_L\tau^2\beta_N^2\eps^4\xi^{\mathsf T}G^B\xi.
\end{equation}
It is stationary in translation at every centered scale, and its mixed translation--scale block vanishes.
Its corrected translation Hessian at scale one is
\begin{equation}\label{a:eq:KphysicalT}
D_\xi^2\widehat F_{h^-_{\tau,\theta}}(0,1)
=-2C_Na_L\tau^2\beta_N^2\diag(m(m-1),2,\ldots,2).
\end{equation}
It is negative definite if $a_L>0$ and $\tau\ne0$.
\end{lemma}
\begin{proof}
Orthogonal invariance makes the corrected pairing of fixed normal profiles times constant trace-free matrices proportional to their Frobenius inner product.
Since $|A|=1$,
\[
\mathcal M_L(t^2B_a,t^2B_b)
=\langle B_a,B_b\rangle(\mathbf M_L)_{22}.
\]
Tangential parity eliminates pairings of a normal-matrix profile with the linear tensor $B$.
The exact translation formula is
\[
h^-(\bar x+\xi,t)=h^-(\bar x,t)
+\tau\beta_Nt^2\sum_a\xi_aB_a.
\]
After rescaling to the unit bubble, the added term is $\tau\beta_N\eps^2t^2\sum_a\xi_aB_a$.
Its cross pairings with the normal-matrix terms vanish because $\langle A,B_a\rangle=0$, and its cross pairing with the linear tensor $B$ vanishes by parity.
Its quadratic pairing is therefore
\[
|\Sph^{N-2}|\tau^2\beta_N^2\eps^4
(\mathbf M_L)_{22}\,\xi^{\mathsf T}G^B\xi.
\]
Since $|\Sph^{N-2}|(\mathbf M_L)_{22}=-C_Na_L$, this proves
\eqref{eq:cubic-exact-translation-energy}.
Differentiation proves stationarity, the zero mixed block, and \eqref{a:eq:KphysicalT}.
\end{proof}

\subsection{Selection of \texorpdfstring{$\tau$}{tau} and strict local maximum}
As in Section~\ref{sec:endpoint-matrix-assembly}, let $E,p,K$ denote the energy and its first two logarithmic scale derivatives, now divided by the positive factor $C_N$:
\[
E_\theta(\tau):=\frac{\widehat F_{h^-_{\tau,\theta}}(0,1)}{C_N},\qquad
p_\theta(\tau):=\frac{\partial_\ell\widehat F_{h^-_{\tau,\theta}}(0,1)}{C_N},\qquad
K_\theta(\tau):=\frac{\partial_{\ell\ell}\widehat F_{h^-_{\tau,\theta}}(0,1)}{C_N}.
\]
Here the derivatives are evaluated at $\ell=0$.
With $\mathbf c:=(\theta,1,\tau)^{\mathsf T}$ and $\mathbf D:=\diag(1,2,3)$, they are the quadratic forms in
\eqref{a:eq:scalingmatrix}, multiplied by $|\Sph^{N-2}|/C_N$.
For the umbilic family this reads
\begin{align*}
E_0(\tau)&=-(a_L-2b_L\tau+c_L\tau^2),\\
p_0(\tau)&=-2(2a_L-5b_L\tau+3c_L\tau^2),\\
K_0(\tau)&=-(16a_L-50b_L\tau+36c_L\tau^2).
\end{align*}
Thus selecting scale one reduces to solving $p_0(\tau)=0$.
For sufficiently large $L$, the comparison lemma makes $c_L>0$ and $\Delta_L:=25b_L^2-24a_Lc_L>0$.
Define the larger root by
\begin{equation}\label{eq:cubic-unit-coefficient}
\tau_0:=\frac{5b_L+\sqrt{\Delta_L}}{6c_L}
\to\frac{8(5+1/\sqrt2)}{49}
\quad\text{as }L\to\infty.
\end{equation}

\begin{theorem}\label{thm:cubic-maximum}
For every $N\ge9$ there is $L_N>0$ such that, for $L\ge L_N$, the family \eqref{a:eq:hLmain} with $\theta=0$ and coefficient
\eqref{eq:cubic-unit-coefficient} has a strict nondegenerate local maximum of negative value at $(0,1)$:
\[
\widehat F_{h^-_{\tau_0,0}}(0,1)<0,\qquad
D\widehat F_{h^-_{\tau_0,0}}(0,1)=0,\qquad
D^2\widehat F_{h^-_{\tau_0,0}}(0,1)\prec0.
\]
Its exponential metric has totally geodesic boundary.
\end{theorem}
\begin{proof}
Lemma~\ref{lem:cubic-half-space-comparison} allows us to choose $L_N$ so that for every $L\ge L_N$,
\begin{equation}
|a_L-1|,\quad |b_L-1|,\quad |c_L-49/48|\le1/1000.
\end{equation}
Elementary rational inequalities on this box give
\[
a_Lc_L-b_L^2\ge269/16000>0,\qquad
\Delta_L\ge401501/10^6>0.
\]
The selected coefficient satisfies $p_0(\tau_0)=0$, and
\[
E_0(\tau_0)=-\frac{a_L-b_L\tau_0}{3}<0,\qquad
K_0(\tau_0)=-2\tau_0\sqrt{\Delta_L}<0.
\]
Indeed $a_Lc_L>b_L^2$ implies $\sqrt{\Delta_L}<b_L$ and $\tau_0<b_L/c_L$, so $a_L-b_L\tau_0>0$.
Since $a_L>0$ and $\tau_0>0$, Lemma~\ref{lem:cubic-translation-block} gives negative definiteness of the translation block and vanishing of the mixed block.
Together with the scale calculation, this proves the strict nondegenerate local maximum of negative value.

Lemma~\ref{lem:cubic-tensor-geometry} gives the boundary assertion.
At the origin the tensor and its first derivatives vanish, and its degree-two term is $t^2A$.
Hence
\[
W_{\exp(\mu h^-)}{}_{aNbN}(0)
=-\mu\frac{N-3}{N-2}A_{ab}\ne0\qquad(\mu\ne0).
\]
The degree-three terms do not change this identity.
\end{proof}

The same matrix determines the stationary equation when $\theta\ne0$:
\begin{equation}\label{eq:negative-perturbed-quadratic}
\begin{aligned}
p_\theta(\tau)=\frac{|\Sph^{N-2}|}{C_N}\bigl\{
&6(\mathbf M_L)_{33}\tau^2
+(10(\mathbf M_L)_{23}+8\theta(\mathbf M_L)_{13})\tau\\
&+4(\mathbf M_L)_{22}+6\theta(\mathbf M_L)_{12}
      +2\theta^2(\mathbf M_L)_{11}\bigr\}.
\end{aligned}
\end{equation}
This follows directly from the degrees $1,2,3$ and
\eqref{eq:negative-matrix-scaling}. Denote by $\tau(\theta)$ the
nearby larger root that equals $\tau_0$ at $\theta=0$.

\begin{corollary}\label{thm:normal-odd-maximum}
For each fixed $N\ge9$ and $L\ge L_N$, there is $\theta_0(N,L)>0$ such that every $0<|\theta|<\theta_0(N,L)$, with $\tau=\tau(\theta)$ from \eqref{eq:negative-perturbed-quadratic}, gives a strict nondegenerate local
maximum of negative value at $(0,1)$:
\[
\widehat F_{h^-_{\tau(\theta),\theta}}(0,1)<0,\qquad
D\widehat F_{h^-_{\tau(\theta),\theta}}(0,1)=0,\qquad
D^2\widehat F_{h^-_{\tau(\theta),\theta}}(0,1)\prec0.
\]
The boundary is nonumbilic throughout the uncut boundary core.
\end{corollary}
\begin{proof}
At $\theta=0$ the stationary root is simple:
\[
p'_0(\tau_0)=-2\sqrt{\Delta_L}\ne0.
\]
The implicit-function theorem provides a smooth nearby larger root $\tau(\theta)$, still selecting scale one.
The leading quadratic coefficient in \eqref{eq:negative-perturbed-quadratic} is independent of $\theta$ and negative.
The quadratic forms defined by $\mathbf M_L$ depend continuously on $(\tau,\theta)$, so $E_\theta(\tau(\theta))<0$ and $K_\theta(\tau(\theta))<0$ for sufficiently small $|\theta|$.
Lemma~\ref{lem:cubic-translation-block} remains exact for every $\theta$, so the translation block is negative definite and the mixed block is zero.
This proves the strict local maximum.

Nonumbilicity follows from \eqref{eq:negative-boundary-jet}.
The degree-one perturbation has zero second derivatives, so the nonzero linearized Weyl component persists for small metric coefficient.
Freeze a nonzero $\theta$ and its selected $\tau(\theta)$ before forming the concentrating sequence.
\end{proof}
Appendix~\ref{app:normal-weyl-bounds} gives explicit sufficient curvature bounds by a separate construction.

\section{Localization on a fixed metric}\label{sec:common-analytic}
The placement of disjoint perturbations at shrinking scales follows \cite{Brendle,BrendleMarques,AlmarazBlowup}.
We give the estimates for the polynomial profiles used here, with their nonzero scalar corrections; the same argument applies to strict minima and strict maxima.

\subsection{The projected equation}
We use the Lyapunov--Schmidt argument of \cite[Propositions~2.8--2.9]{AlmarazBlowup}.
We record the estimates needed for the interior critical term and either sign of the boundary parameter.

Fix $(\sigma,\eta)$ and use the energy \eqref{eq:common-functional-expanded} on $\R^N_+$.
For block metrics of determinant one and mean curvature zero, define the nonlinear weak residual $\mathscr R_g(u)\in\Sigma^*$ by
\begin{align*}
\mathscr R_g(u)[\psi]
:={}&\int_{\R^N_+}
(\inner{du}{d\psi}_g+c_NR_gu\psi)\dd x
-\frac{N(N-2)\sigma}{4}\int_{\R^N_+}(u_+)^{(N+2)/(N-2)}\psi\dd x\\
&-\frac{(N-2)\eta}{2}\int_{\pa\R^N_+}(u_+)^{N/(N-2)}\psi\dd\bar x.
\end{align*}
Thus
\[
\mathscr R_g=\tfrac12DE_g,\qquad \mathscr R_{g_{\mathrm{euc}}}(U_z)=0.
\]
Constants may depend on the fixed curvature parameter $\kappa$.

The full nonlinear weak equation is
\begin{equation}\label{eq:full-nonlinear-weak-equation}
\mathscr R_g(u)[\psi]=0\qquad\text{for every }\psi\in\Sigma.
\end{equation}
For $u\in\Sigma$, define the symmetric bilinear form
\begin{align*}
\cB_{g,u}(v,\psi)
:={}&\frac12D^2E_g(u)[v,\psi]
=\int_{\R^N_+}
(\inner{dv}{d\psi}_g+c_NR_gv\psi)\dd x\\
&-\frac{N(N+2)\sigma}{4}\int_{\R^N_+}
(u_+)^{4/(N-2)}v\psi\dd x
-\frac{N\eta}{2}\int_{\pa\R^N_+}
(u_+)^{2/(N-2)}v\psi\dd\bar x,
\end{align*}
for $v,\psi\in\Sigma$.
At a flat bubble, $\cB_{g_{\mathrm{euc}},U_z}=\cB_z$.
Its restriction to $\Sigma_z$ is the linearization on the complement.

Throughout this subsection, let $N>6$ and $\mathcal O\Subset\R^{N-1}\times(0,\infty)$.
Assume
\[
\sup_{z\in\overline{\mathcal O}}\|\cL_z^{-1}\|_{\mathrm{op}}<\infty.
\]
Let $g$ be a smooth uniformly elliptic block metric, Euclidean outside a compact set, with $\det g=1$, $H_g=0$, and
\begin{equation}\label{eq:metric-smallness-natural-constraint}
\|g-g_{\mathrm{euc}}\|_{L^\infty}+\|R_g\|_{L^{N/2}}\le\alpha_1.
\end{equation}
For a projected family $V_{g,z}:=U_z+w_{g,z}$, write
\begin{equation*}
\mathcal F_g(z):=E_g(V_{g,z}).
\end{equation*}

\begin{proposition}\label{prop:precise-natural-constraint}
For sufficiently small $\alpha_1$, there is a locally unique family
\begin{equation}\label{eq:nonlinear-projected-family}
V_{g,z}=U_z+w_{g,z},\qquad w_{g,z}\in\Sigma_z,\qquad
\sup_{z\in\mathcal O}\|w_{g,z}\|_\Sigma\le C\alpha_1,
\end{equation}
continuous on $\overline{\mathcal O}$ and $C^1$ on $\mathcal O$, satisfying
\begin{equation*}
\mathscr R_g(V_{g,z})[\psi]=0\qquad(\psi\in\Sigma_z).
\end{equation*}
Every critical point of $\mathcal F_g$ gives a smooth positive solution of \eqref{eq:full-nonlinear-weak-equation}.
\end{proposition}

\begin{proof}
By the Sobolev and trace inequalities, uniformly on $\overline{\mathcal O}$,
\begin{align*}
\|\mathscr R_g(U_z)\|_{\Sigma^*}&\le C\alpha_1,\\
\bigl|\cB_{g,U_z+w}(v,\psi)-\cB_z(v,\psi)\bigr|
&\le C\bigl(\alpha_1+\|w\|_\Sigma^{2/(N-2)}\bigr)
\|v\|_\Sigma\|\psi\|_\Sigma.
\end{align*}
Here $w,v,\psi\in\Sigma_z$ and $\|w\|_\Sigma\le1$.
The critical power maps are $C^1$ between the corresponding Sobolev and dual spaces.
The uniform inverse bound makes
\[
w\mapsto w-\cL_z^{-1}\mathscr R_g(U_z+w)
\]
a contraction on a ball of radius $C\alpha_1$.
This proves existence, local uniqueness and \eqref{eq:nonlinear-projected-family}; the linearized projected inverse remains uniformly bounded.
Identifying the spaces $\Sigma_z$ by \eqref{eq:explicit-Sigma-projection}, the asserted parameter dependence follows from the implicit-function theorem.

For the natural constraint, write $\mathscr R_g(V_{g,z})=\sum_\alpha b_\alpha\mathfrak l_{z,\alpha}$.
Differentiating $\mathfrak l_{z,\alpha}(w_{g,z})=0$, we find
\[
\pa_{z_\beta}\mathcal F_g
=2\sum_\alpha b_\alpha
\bigl[(\mathbf G_z)_{\alpha\beta}-(\pa_{z_\beta}\mathfrak l_{z,\alpha})(w_{g,z})\bigr].
\]
The matrix in brackets is $\mathbf G_z+O(\alpha_1)$, hence invertible.
Thus every critical point satisfies $b=0$ and \eqref{eq:full-nonlinear-weak-equation}.

By the smallness condition \eqref{eq:metric-smallness-natural-constraint}, there is $c>0$ such that
\begin{equation}\label{eq:common-coercivity}
\int_{\R^N_+}(|\nabla v|_g^2+c_NR_gv^2)\dd x
\ge c\|v\|_\Sigma^2\qquad(v\in\Sigma).
\end{equation}
At a critical point, set $V:=V_{g,z}$ and $V_-:=\max\{-V,0\}$.
Both nonlinear terms vanish against $V_-$ because they involve $V_+$.
Testing the full equation with $-V_-$ yields
\[
0=\int_{\R^N_+}(|\nabla V_-|_g^2+c_NR_gV_-^2)\dd x
\ge c\|V_-\|_\Sigma^2.
\]
Hence $V\ge0$, independently of the signs of the curvature coefficients, and closeness to $U_z$ excludes $V\equiv0$.
Since
\[
V^{4/(N-2)}\in L^{N/2}_{\mathrm{loc}},\qquad
V^{2/(N-2)}|_{\pa\R^N_+}\in L^{N-1}_{\mathrm{loc}},
\]
cutoff power iteration, absorbing the small-norm tails of these coefficients, establishes every finite local integrability exponent.
Local elliptic estimates then imply $C^{1,\alpha}$ regularity.
The strong maximum principle and boundary point lemma imply $V>0$ up to the boundary; smoothness follows by elliptic bootstrapping.
\end{proof}

\subsection{Localization and concentration}
\label{subsec:fixed-metric-transfer}
Fix
\[
(\sigma,\eta)=(1,\kappa),\quad \kappa\in\R,
\qquad\text{or}\qquad (\sigma,\eta)=(0,2).
\]
Let $h$ be a tangential symmetric polynomial tensor on $\R^N_+$ with
\[
h_{iN}=0,\qquad \tr h=0,\qquad
\deg h\le d_*,\qquad N>\max\{6,2d_*+2\}.
\]
Fix a compact $K\Subset\R^{N-1}\times(0,\infty)$ and an open relatively compact neighborhood $\mathcal O$ of $K$ on which to apply Proposition~\ref{prop:precise-natural-constraint}.
For $\mu_0>0$ and integers $j\ge j_0$, set
\begin{equation*}
\begin{gathered}
p_j:=2^{-j}\mathbf e_1,\quad \rho_j:=2^{-j-5},\quad
\lambda_j:=\exp(-j^3),\quad
\mu_j:=\mu_0\exp(-j^2)\lambda_j^{d_*},\\
h_j^{\mathrm{loc}}(x):=\mu_j\chi(2|x-p_j|/\rho_j)
h((x-p_j)/\lambda_j),\qquad
h^{\mathrm{loc}}:=\sum_jh_j^{\mathrm{loc}},\qquad g:=\exp(h^{\mathrm{loc}}),
\end{gathered}
\end{equation*}
Define the dilation and rescaled metric by
\[
D_j(x):=p_j+\lambda_jx,\qquad g_j:=\lambda_j^{-2}D_j^*g.
\]
Thus the coefficient matrices satisfy $g_j(x)=g(p_j+\lambda_jx)$.
The constants $\mu_0$ and $j_0$ will be chosen depending on $K$ and then held fixed as $j\to\infty$.
In particular,
\[
R_{g_j}(x)=\lambda_j^2R_g(p_j+\lambda_jx),\qquad
\|R_{g_j}\|_{L^{N/2}}=\|R_g\|_{L^{N/2}}.
\]

\begin{lemma}\label{lem:localized-energy}
For sufficiently small $\mu_0$ and large $j_0$, the metric $g$ is smooth and minimal, with positive conformal quadratic form.
Its perturbation is compactly supported and vanishes to infinite order at the accumulation point.
The projected solutions $U_z+w_{j,z}$ from Proposition~\ref{prop:precise-natural-constraint} satisfy
\begin{equation}\label{eq:fixed-corrector-convergence}
\sup_{z\in K}\|\mu_j^{-1}w_{j,z}-v_z[h]\|_\Sigma\to0\qquad\text{as }j\to\infty
\end{equation}
and
\begin{equation}\label{eq:fixed-reduced-convergence}
E_{g_j}(U_z+w_{j,z})=E_{g_{\mathrm{euc}}}(U_z)+\mu_j^2F_h(z)+o(\mu_j^2)
\quad\text{as }j\to\infty,\quad\text{uniformly on }K.
\end{equation}
\end{lemma}
\begin{proof}

Here and below the normal components of $h^{\mathrm{loc}}$ are zero.
The supports are disjoint.
Since $h$ has degree at most $d_*$ and $\lambda_j\le\rho_j$,
\begin{equation*}
\|h_j^{\mathrm{loc}}\|_{C^k}\le C_k\mu_0\exp(-j^2)\rho_j^{d_*-k}\qquad(k\ge0).
\end{equation*}
For every fixed $k,q$,
\[
\|h_j^{\mathrm{loc}}\|_{C^k}=o(|p_j|^q)\qquad\text{as }j\to\infty.
\]
Thus $h^{\mathrm{loc}}$ extends smoothly by zero at the accumulation point and vanishes there to infinite order.
Since $h^{\mathrm{loc}}$ is compactly supported and tangential trace-free,
\[
\det g=1,\qquad H_{g}=0,\qquad
g=g_{\mathrm{euc}}\quad\text{outside a compact set}.
\]
By the curvature expansion used in Proposition~\ref{prop:metric-bubble-expansion} and the disjoint supports,
\[
\|g-g_{\mathrm{euc}}\|_\infty+\|R_{g}\|_{L^{N/2}}
\le C\mu_0.
\]
Indeed,
\[
\|\pa^2h_j^{\mathrm{loc}}\|_{L^{N/2}}\le C\mu_0\exp(-j^2)\rho_j^{d_*},
\]
and the quadratic derivative terms are smaller.
The Sobolev inequality then implies \eqref{eq:common-coercivity} for sufficiently small $\mu_0$.

After the rescaling $x\mapsto p_j+\lambda_jx$, the metric coefficients are $g_j(x)=g(p_j+\lambda_jx)$.
Put
\begin{equation}
\begin{aligned}
R_j&:=\rho_j/\lambda_j,\qquad
\zeta_0:=(N-2)/2-d_*>0,\\
a_j&:=\mu_jR_j^{d_*}=\mu_0\exp(-j^2)\rho_j^{d_*},\qquad
b_j:=R_j^{-(N-2)/2}/\mu_j.
\end{aligned}
\end{equation}
Here $R_j$ is the rescaled cutoff radius, $a_j$ bounds the metric perturbation on its support, and $\mu_jb_j$ bounds the remote bubble tail.
The strict degree inequality implies
\[
b_j=\mu_0^{-1}\rho_j^{-(N-2)/2}\exp(j^2-\zeta_0j^3)\to0
\qquad\text{as }j\to\infty.
\]
Then
\[
R_j\to\infty,\qquad a_j\to0,\qquad
b_j\to0\qquad\text{as }j\to\infty.
\]
Moreover,
\[
g_j=\exp(\mu_jh)\quad\text{in }|x|<R_j/2.
\]
Every other perturbation lies in $|x|\ge R_j$.

Uniformly for $z\in K$,
\begin{equation}\label{eq:fixed-residual-expansion}
\mathscr R_{g_j}(U_z)=\mu_j\Lambda_z[h]+o(\mu_j)
\quad\text{in }\Sigma^*\quad\text{as }j\to\infty.
\end{equation}
Here $\Lambda_z[h]$ includes all three terms of \eqref{eq:polynomial-source-functional}.
More precisely,
\[
\sup_{z\in K}\|\mathscr R_{g_j}(U_z)-\mu_j\Lambda_z[h]\|_{\Sigma^*}
\le C\mu_j(R_j^{-\zeta_0}+a_j+b_j)=o(\mu_j).
\]
The source and cutoff-derivative terms in \eqref{eq:cutoff-source-difference} have dual norm at most $C\mu_jR_j^{-\zeta_0}$ by \eqref{eq:cutoff-source-dual-convergence}, uniformly on $K$.
The nonlinear metric terms on the active support have dual norm at most $C\mu_ja_j$.
Uniformly on $K$, the bubble tails obey
\[
\|U_z\|_{L^{2N/(N-2)}(|x|>R)}
+\|\nabla U_z\|_{L^2(|x|>R)}
\le CR^{-(N-2)/2}.
\]
Together with the invariant curvature norm these control the remote residual by $CR^{-(N-2)/2}$ and its quadratic energy by $CR^{2-N}$.
Off the active support, Sobolev and H\"older inequalities therefore give $CR_j^{-(N-2)/2}=C\mu_jb_j$.
The metric boundary residual is zero, since $g_{NN}=1$, $g_{aN}=0$ and $H_g=0$.

The linearized operators satisfy
\[
\sup_{z\in K}\|D\mathscr R_{g_j}(U_z)-D\mathscr R_{g_{\mathrm{euc}}}(U_z)\|_{\mathrm{op}}
\le C\mu_0,
\]
and, for every $v\in\Sigma$,
\[
\sup_{z\in K}
\|[D\mathscr R_{g_j}(U_z)-D\mathscr R_{g_{\mathrm{euc}}}(U_z)]v\|_{\Sigma^*}
\to0\qquad\text{as }j\to\infty.
\]
The latter follows first for compactly supported smooth $v$ from local coefficient convergence, and then by density and the uniform operator bound.
The convergence is uniform on compact subsets of $\Sigma$.
The small operator-norm bound preserves invertibility, while strong convergence identifies the limiting corrector.

Proposition~\ref{prop:precise-natural-constraint}, applied to the entire rescaled metric $g_j$, gives $V_{j,z}=U_z+w_{j,z}$ with
\[
\sup_{z\in K}\|w_{j,z}\|_\Sigma=O(\mu_j).
\]
This follows from \eqref{eq:fixed-residual-expansion} and the uniform contraction constant.
The map
\[
v_z[h]=-\cL_z^{-1}(\Lambda_z[h]|_{\Sigma_z})
\]
is continuous on $K$ with compact image in $\Sigma$.
Define the restricted linearization by
\[
\cL_{j,z}v:=\bigl(D\mathscr R_{g_j}(U_z)[v]\bigr)|_{\Sigma_z},
\qquad v\in\Sigma_z.
\]
Since $\cL_zv_z[h]=-\Lambda_z[h]|_{\Sigma_z}$,
\begin{equation}
\begin{aligned}
&\sup_{z\in K}\|\mathscr R_{g_j}(U_z+\mu_jv_z[h])|_{\Sigma_z}\|_{\Sigma_z^*}
\le\sup_{z\in K}\|\mathscr R_{g_j}(U_z)-\mu_j\Lambda_z[h]\|_{\Sigma^*}\\
&\quad+\mu_j\sup_{z\in K}\|(\cL_{j,z}-\cL_z)v_z[h]\|_{\Sigma_z^*}
+C\mu_j^{1+2/(N-2)}
=o(\mu_j)\quad\text{as }j\to\infty.
\end{aligned}
\end{equation}
The first term tends to zero after division by $\mu_j$ by \eqref{eq:fixed-residual-expansion}.
Strong convergence is uniform on the compact image of $z\mapsto v_z[h]$ and controls the second term; the last term is the critical-power remainder.
The contraction estimate with its uniform inverse bound yields
\[
\sup_{z\in K}\|w_{j,z}-\mu_jv_z[h]\|_\Sigma=o(\mu_j)
\qquad\text{as }j\to\infty,
\]
which proves \eqref{eq:fixed-corrector-convergence}.

By Proposition~\ref{prop:metric-bubble-expansion}, the cutoff tensor has no first-order energy term.
Estimating the quadratic polynomial tail, cubic metric terms and exterior contribution separately, we obtain
\[
\sup_{z\in K}
|E_{g_j}(U_z)-E_{g_{\mathrm{euc}}}(U_z)-\mu_j^2\mathcal Q[h](z)|
\le C\mu_j^2(R_j^{-2\zeta_0}+a_j+b_j^2).
\]
The estimate includes all four terms of $\mathcal Q[h]$.
Indeed their exterior radial integrals are bounded by $C\int_{R_j/2}^\infty r^{2d_*+1-N}\dd r =O(R_j^{-2\zeta_0})$.
Cutoff derivative contributions have the same annular order, including the two divergence terms; cubic metric terms gain the factor $a_j$.
Hence
\[
E_{g_j}(U_z)=E_{g_{\mathrm{euc}}}(U_z)+\mu_j^2\mathcal Q[h](z)+o(\mu_j^2)
\]
as $j\to\infty$, uniformly on $K$.
For $\|w\|_\Sigma\le1$, the power inequalities imply the Taylor estimate
\[
\left|E_{g_j}(U_z+w)-E_{g_j}(U_z)-2\mathscr R_{g_j}(U_z)[w]
-(D\mathscr R_{g_j}(U_z)[w])(w)\right|
\le C\|w\|_\Sigma^{2+2/(N-2)}.
\]
Together with \eqref{eq:fixed-corrector-convergence} and strong operator convergence, this yields \eqref{eq:fixed-reduced-convergence}.
Indeed,
\[
2\Lambda_z[h](v_z[h])+\cB_z(v_z[h],v_z[h])
=-\Lambda_z[h](\cL_z^{-1}\Lambda_z[h]).
\]
The convergence in \eqref{eq:fixed-reduced-convergence} is uniform in $z$; no convergence of the reduced-energy Hessians is required below.

\end{proof}

\begin{theorem}[Transfer to a fixed metric]\label{thm:fixed-metric-transfer}
For $h$ and the curvature pair fixed above, suppose $F_h$ has a strict isolated local extremum at $z_*$ with $F_h(z_*)<0$.
Then a single smooth metric $g$ on the closed ball $B^N$, with positive conformal quadratic form, admits positive smooth functions $u_j$ satisfying
\[
R_{u_j^{4/(N-2)}g}=N(N-1)\sigma,\qquad
H_{u_j^{4/(N-2)}g}=\eta,\qquad
\max_{\pa B^N}u_j\to\infty\quad\text{as }j\to\infty.
\]
For large $j$, their energy values $E_g(u_j)$ are below the energy of the corresponding flat bubble.
The background may be chosen minimal, and totally geodesic if $\pa_th_{ab}=0$ on the boundary.
If $\exp(\mu h)$ has nonzero trace-free second fundamental form, respectively nonzero ambient Weyl curvature, at a fixed rescaled point for every sufficiently small $\mu\ne0$, the background has the same property near
each concentration center.
\end{theorem}
\begin{proof}
Choose a closed ball $K$ about $z_*$ such that
\[
\inf_{\pa K}F_h>F_h(z_*)\quad\text{for a minimum},\qquad
\sup_{\pa K}F_h<F_h(z_*)\quad\text{for a maximum},
\]
and $z_*$ is the unique point of $K$ with this extremal value.
Apply Lemma~\ref{lem:localized-energy} to this $K$.
The flat-bubble energy is independent of $z$, so the uniform expansion \eqref{eq:fixed-reduced-convergence} and the strict boundary gap imply that the exact reduced energy has an interior extremum at $z_j\in K$.
Every subsequential limit of these extremizers attains the extremal value of $F_h$ on $K$.
Since $z_*$ is the unique such point, $z_j\to z_*$ as $j\to\infty$.
Proposition~\ref{prop:precise-natural-constraint} then implies
\[
\nabla_z\mathcal F_{g_j}(z_j)=0,\qquad
\mathscr R_{g_j}(V_{j,z_j})=0,\qquad V_{j,z_j}>0.
\]
Rescaling, we obtain
\[
v_j(x):=\lambda_j^{-(N-2)/2}
V_{j,z_j}((x-p_j)/\lambda_j)
\]
on the same fixed metric $g$.
The negative value of $F_h(z_*)$ makes its energy smaller than the flat bubble energy for large $j$.

Set
\[
\varpi(x):=\left(\frac2{1+|x|^2}\right)^{(N-2)/2},\qquad
\bar g:=\varpi^{4/(N-2)}g,\qquad u_j:=\varpi^{-1}v_j.
\]
Since $g$ is Euclidean near infinity, $\bar g$ extends smoothly to the hemisphere, which we identify smoothly with $B^N$.
By conformal covariance, $u_j$ solves the prescribed-curvature equation away from the omitted boundary point.
The Hardy estimates in the proof of Proposition~\ref{prop:curvature-completion}, with $\varpi$ in place of $U_z$, imply $u_j\in H^1$ on the compactification.
Cutoffs near the omitted point have squared gradient integral $O(r^{N-2})$, so its zero $H^1$ capacity extends the weak equation across that point.
In a smooth chart for $\bar g$, the local iteration and boundary regularity argument in Proposition~\ref{prop:precise-natural-constraint} imply smoothness and positivity there.
Coercivity and umbilicity are conformally invariant, and $\pa_t\varpi=0$ on the boundary preserves minimality.
The metric $\bar g$ is the metric denoted by $g$ in the theorem.

As $j\to\infty$, the trace of $V_{j,z_j}$ converges in $L^{2(N-1)/(N-2)}$ on every fixed boundary ball to the trace of $U_{z_*}$.
Since $U_{z_*}$ is positive, the $L^{2(N-1)/(N-2)}$ norm of this trace on one fixed boundary ball is bounded below by a positive constant.
By rescaling and using the bounds for $\varpi$ near $p_j$, we obtain
\[
\max_{\pa B^N}u_j\ge c\lambda_j^{-(N-2)/2}\to\infty\qquad\text{as }j\to\infty.
\]
The cutoff is constant on each fixed rescaled compact set for large $j$; hence the assumed local tensor and curvature nonvanishing properties hold for the exact metric.
In particular, at the boundary centers of the rescaled perturbations, if $\pa_th(\xi_*,0)\ne0$, then $\xi_j\to\xi_*$ as $j\to\infty$ implies, for all sufficiently large $j$,
\[
\pa_th^{\mathrm{loc}}(p_j+\lambda_j(\xi_j,0))
=\frac{\mu_j}{\lambda_j}\pa_th(\xi_j,0)\ne0
\qquad(j\gg1).
\]
Injectivity of the differential of the matrix exponential and $H_g=0$ therefore give $\pi_g(p_j+\lambda_j(\xi_j,0))\ne0$.
On the boundary,
\[
\pa_t\chi(2|x-p_j|/\rho_j)=0,\qquad
\pa_th=0\quad\Rightarrow\quad\pa_th_j^{\mathrm{loc}}=0.
\]
This proves the geometric assertions.
\end{proof}
\section{Proofs of the main results}\label{sec:applications}
The following table summarizes the transfer hypotheses verified in the proofs below.
All quantities in a row are evaluated in the same ambient dimension $N$.
Corollary~\ref{cor:explicit-negative-bounds} is proved independently in Appendix~\ref{app:normal-weyl-bounds}.
\begin{table}[ht]
\centering\small\setlength{\tabcolsep}{4pt}
\begin{tabular}{@{}lcccccc@{}}\toprule
$(\sigma,\eta)$; boundary & Dimension & $d_*$ & $N-2d_*-2$ & $\partial_t h(0,0)$ & Extremum & $W_g$\\\midrule
$(0,2)$, nonumbilic & $N\ge15$ & $6$ & $1$ & $\tau_NA$ & minimum & $\ne0$\\
$(1,0)$, nonumbilic & $N\ge15$ & $6$ & $1$ & $\tau_NA$ & minimum & $\ne0$\\
$(0,2)$, umbilic & $N\ge22$ & $8$ & $4$ & $0$ & minimum & $\ne0$\\
$(1,0)$, umbilic & $N\ge21$ & $9$ & $1$ & $0$ & minimum & $\ne0$\\
$(1,-L)$, umbilic & $N\ge9$ & $3$ & $1$ & $0$ & maximum & $\ne0$\\
$(1,-L)$, nonumbilic & $N\ge9$ & $3$ & $1$ & $\theta A_{\mathrm c}$ & maximum & $\ne0$\\\bottomrule
\end{tabular}
\caption{Hypotheses of the fixed-metric construction.
All extrema have negative value and use the corrected coefficient and projected inverse of Proposition~\ref{prop:scalar-elimination}.
The four minima use the positive normalization identity \eqref{eq:complete-normalization-bridge}; cubic maxima use \eqref{eq:negative-corrected-form}.
The fourth column lists the least value of $N-2d_*-2$ in the indicated dimension range; $W_g$ is the ambient Weyl tensor at the point used in the curvature calculation.
In the last two rows, $L\ge L_N$ is fixed.
In each umbilic row the first normal jet vanishes on the entire boundary.}
\label{tab:transfer-interface}
\end{table}
\begin{proof}[Proof of cases~\textup{(i)--(ii)} of Theorem~\ref{thm:main}]
Use Proposition~\ref{prop:four-scalar-minima} with the corresponding tensor.
Table~\ref{tab:transfer-interface} verifies $N>\max\{6,2d_*+2\}$.
For the nonumbilic families $\partial_t h(0,0)=\tau_NA\ne0$; for the umbilic families $\partial_t h(\bar x,0)=0$ identically.
The exponential metric therefore has umbilic or nonumbilic boundary as asserted.
In the nonumbilic case continuity of the first normal jet also verifies the active-center condition when the concentrating critical parameters move slightly during localization.

We check nonzero ambient Weyl curvature independently of the energy signs.
With $m=N-1$, the leading normal Weyl component of a degree-six nonumbilic tensor is a positive multiple of $A$, since its coefficient is
\[
\frac{c_{10}(m+2b_1)-(m-2)c_{02}}{m-1}>0,
\qquad m+2b_1=\frac{m(m+4)(m-2)}{m^2+2m-4}>0.
\]
Both coefficient columns have $c_{10}>0$ and $c_{02}<0$.
For the scalar-flat umbilic tensor it is $2\tau_N(N-3)/(N-2)>0$; its terms of degree at least three do not affect the curvature at the origin.
For the minimal umbilic tensor it is $-c_{02}\tau_N(N-3)/(N-2)>0$.
In the umbilic cases the first metric jet vanishes, so the curvature of $\exp(\mu h)$ at the origin is exactly linear in $\mu$.
In the nonumbilic cases the nonzero linear terms dominate for sufficiently small $\mu\ne0$.

By Theorem~\ref{thm:fixed-metric-transfer}, there is a fixed smooth background and a positive blow-up sequence with all the stated properties.
\end{proof}

\begin{proof}[Proof of Theorem~\ref{thm:main}, case~\textup{(iii)}]
For the nonumbilic case use Corollary~\ref{thm:normal-odd-maximum} with one fixed nonzero $\theta$ and its selected $\tau(\theta)$.
For the umbilic case use Theorem~\ref{thm:cubic-maximum}.
In both cases $d_*=3$, the full critical point is a negative nondegenerate maximum and $N>8$.
Their boundary and Weyl conditions were proved with the tensors.
Fix $N$ and then any $L\ge L_N$, where $L_N$ is chosen in Theorem~\ref{thm:cubic-maximum}.
By the fixed-parameter inverse and Theorem~\ref{thm:fixed-metric-transfer}, we obtain the asserted sequence with $\kappa=-L$.
For the negative nonumbilic family, \eqref{eq:negative-boundary-jet} and the rescaling in Section~\ref{subsec:fixed-metric-transfer} imply, throughout the uncut boundary core,
\[
\pi=-\frac{\theta\mu_j}{2\lambda_j}A_{\mathrm c}\ne0 .
\]
It remains nonzero under the small tangential displacement of the concentration center caused by the other localized perturbations.
The fixed coefficients $L,\theta,\tau(\theta)$ do not vary with $j$.

\end{proof}

\begin{proposition}\label{prop:curvature-persistence}
Suppose $z_*$ is a negative nondegenerate local extremum of $F_h$:
\[
F_h(z_*)<0,\qquad DF_h(z_*)=0,\qquad
D^2F_h(z_*)\succ0\ \text{or}\ D^2F_h(z_*)\prec0.
\]
Such an extremum persists under a sufficiently small change of a fixed finite $\kappa$.
An extremum at the scalar-flat normalization $(0,2)$ also persists for the curvature pair $(\sigma,2)$ with all sufficiently small $\sigma>0$.
The metric construction of Theorem~\ref{thm:fixed-metric-transfer} applies for each fixed parameter in these intervals, while preserving whether the boundary of its exponential metric is umbilic or nonumbilic.
\end{proposition}
\begin{proof}
Fix a finite parameter $\kappa_0$ at which the stated extremum exists.
The bubbles, Jacobi fields and their weighted Gram matrices depend smoothly on $(\kappa,\xi,\eps)$.
On a compact parameter set the restricted Jacobi operators are continuous in operator norm.
Invertibility at the fixed parameter and a Neumann series ensure bounded inverses nearby, after identifying the complements by their explicit projections.
Polynomial sources and all their parameter derivatives through order two are dominated in the relevant dual norm by the same integrable powers when $N>2d_*+2$.
The raw quadratic integrals obey the same domination.
Resolvent identities, or two differentiations of the augmented inverse, therefore establish $C^2$ dependence of $F_h$ on $(\xi,\eps)$ and continuous dependence of these derivatives on $\kappa$.
The implicit-function theorem provides a critical point $z_*(\kappa)$ near $z_*(\kappa_0)=z_*$ and preserves its negative value and Hessian inertia.
The tensor and its selected coefficient remain fixed; only $\kappa$ and the critical bubble parameter vary.
To continue from the scalar-flat equation, define
\begin{equation*}
V_{\sigma,\xi,\eps}(x):=
\left(\frac{\eps}{|\bar x-\xi|^2+(t+\eps)^2+\sigma\eps^2/4}\right)^{(N-2)/2}.
\end{equation*}
Direct differentiation gives
\[
\begin{cases}
-\Delta V_\sigma=\dfrac{N(N-2)\sigma}{4}V_\sigma^{(N+2)/(N-2)}
&\text{in }\R^N_+,\\
\pa_\nu V_\sigma=(N-2)V_\sigma^{N/(N-2)}
&\text{on }\pa\R^N_+,
\end{cases}
\qquad V_0=U^{\mathrm{sf}}.
\]
The curvature pair is $(\sigma,2)$.
The scalar-flat boundary-weighted constraints have a positive-definite Gram matrix on the nearby Jacobi fields.
Operator perturbation, polynomial domination and the resolvent identities prove $C^2$ convergence of the quadratic functional as $\sigma\downarrow0$.
The implicit-function theorem preserves the critical point, its negative value and its Hessian inertia.
For $\sigma>0$,
\[
\sigma^{(N-2)/4}V_{\sigma,\xi,\eps}
=U_{2/\sqrt{\sigma},\,\xi,\,\sqrt{\sigma}\eps/2}.
\]
Under this normalization the energy is multiplied by $\sigma^{(N-2)/2}$.
By Corollary~\ref{cor:boundary-gauge-directions}, the scalar correction pairing is independent of the choice of complement.
Thus Theorem~\ref{thm:fixed-metric-transfer} applies after rescaling $\eps$; rescaling the conformal factors back gives the curvature pair $(\sigma,2)$.
Umbilicity or nonumbilicity is preserved because the tensor and its first normal derivative are fixed.
\end{proof}

\begin{proof}[Proof of Corollary~\ref{thm:positive-end-main}]
Proposition~\ref{prop:curvature-persistence} continues each negative nondegenerate scalar-flat minimum used in Theorem~\ref{thm:main}, case~\textup{(i)} to
\[
\eta=2,\qquad 0<\sigma<\sigma_0,
\]
with the same polynomial degree and the same umbilicity or nonumbilicity condition.
For each fixed $\sigma$, Theorem~\ref{thm:fixed-metric-transfer} gives a fixed background and a blow-up sequence.
Rescaling a resulting conformal metric $\hat{g}$ gives
\[
R_{\sigma\hat{g}}=\sigma^{-1}R_{\hat{g}}=N(N-1),\qquad
H_{\sigma\hat{g}}=\sigma^{-1/2}H_{\hat{g}}=\frac2{\sqrt \sigma}.
\]
Hence every fixed $\kappa>2/\sqrt{\sigma_0}$ occurs, in every $N\ge15$ in the nonumbilic class and every $N\ge22$ in the umbilic class.

For the minimal profiles, the same proposition gives an interval of finite $\kappa$ about zero and preserves the negative nondegenerate minimum and whether the boundary is umbilic or nonumbilic.
Applying the fixed-metric theorem for each fixed parameter proves the first assertion, including umbilic dimension $21$.
The parameter neighborhoods may depend on the dimension and on whether the boundary is umbilic or nonumbilic; the background may also depend on the prescribed parameter.
\end{proof}
\appendix
\section{Defining coefficients}\label{a:app:coefficients}
All profiles use the rows in \eqref{eq:full-scalar-profile}.
Table~\ref{a:tab:N15coeff} lists $c_{kp}$ for the two nonumbilic profiles; $b_{kp}=b_k$ and $e_{kp}=e_k$ are fixed by \eqref{a:eq:Achain}.
Table~\ref{tab:umbilic-coefficients} lists the products $c_{kp}/2$, $c_{kp}b_{kp}/2$, and $c_{kp}e_{kp}$ for the umbilic profiles.
For each nonzero row these three numbers determine $c_{kp},b_{kp},e_{kp}$ uniquely.
A zero row means $c_{kp}=0$ and contributes no tensor.
Daggers in the boundary rows denote the products determined by \eqref{eq:longitudinal-boundary-rule}, with both products zero when $c_{k0}=0$.
All displayed decimals are exact rational numbers, and the selected coefficient $\tau_N$ is fixed under bubble variations.

\begin{table}[ht]\centering\small
\begin{tabular}{@{}rrll@{}}\toprule
$k$&$p$&$c^{\mathrm{II},15}_{kp}$&$c^{\mathrm I,15}_{kp}$\\\midrule
0&1&$1$&$1$\\
1&0&$0.600805704576$&$0.588694426561$\\
0&2&$-0.688195625242$&$-0.674322706788$\\
1&1&$-0.332840224149$&$-0.332609859528$\\
0&3&$0.857655983231$&$0.74725306618$\\
2&0&$-0.001642472821$&$-0.002929852861$\\
1&2&$0.297652759311$&$0.315801923682$\\
0&4&$-0.349568862318$&$-0.256816464616$\\
2&1&$0.013035810584$&$0.012536223758$\\
1&3&$-0.116353102596$&$-0.121177058855$\\
0&5&$0.109154660929$&$0.070875477698$\\
3&0&$-0.000020710015$&$-0.000023805009$\\
2&2&$-0.002612307955$&$-0.00323049507$\\
1&4&$0.012088121034$&$0.014963018057$\\
0&6&$-0.009486633223$&$-0.008850141533$\\
\bottomrule\end{tabular}
\caption{The two degree-six nonumbilic coefficient columns.
The $(0,1)$ entry is replaced by $\tau$.}
\label{a:tab:N15coeff}
\end{table}

\begin{table}[p]\centering\small
\setlength{\tabcolsep}{3pt}\renewcommand{\arraystretch}{1.13}
\begin{tabular}{@{}rr rrr rrr@{}}\toprule
& & \multicolumn{3}{c}{Type II, umbilic}
& \multicolumn{3}{c}{Minimal Type I, umbilic}\\
\cmidrule(lr){3-5}\cmidrule(l){6-8}
$k$&$p$&$c_{kp}/2$&$c_{kp}b_{kp}/2$&$c_{kp}e_{kp}$
&$c_{kp}/2$&$c_{kp}b_{kp}/2$&$c_{kp}e_{kp}$\\\midrule
0&2&$-1$&$0$&$0$&$-0.6079397023$&$0$&$0$\\
0&3&$1$&$0$&$0$&$1$&$0$&$0$\\
0&4&$-1.1566$&$0$&$0$&$-0.7023612638$&$0$&$0$\\
1&2&$0.2041$&$0.0418$&$0$&$0.3081795754$&$0.2052085285$&$0$\\
2&0&$0$&$\dagger$&$\dagger$&$-0.001153284$&$\dagger$&$\dagger$\\
0&5&$0.8322$&$0$&$0$&$0.1916996655$&$0$&$0$\\
1&3&$-0.3857$&$-0.0396$&$0$&$-0.4429846842$&$-0.1820276564$&$0$\\
0&6&$-0.4163$&$0$&$0$&$-0.0047752815$&$0$&$0$\\
1&4&$0.3253$&$0.02$&$0$&$0.3021109048$&$0.0956548791$&$0$\\
2&2&$-0.0196$&$-0.0113$&$-0.0032$&$-0.0316493463$&$-0.0501286591$&$-0.0489756125$\\
3&0&$0$&$\dagger$&$\dagger$&$0.0001828883$&$\dagger$&$\dagger$\\
0&7&$0.1398$&$0$&$0$&$0.0055544718$&$0$&$0$\\
1&5&$-0.1449$&$-0.0082$&$0$&$-0.1166127162$&$-0.0346667298$&$0$\\
2&3&$0.0209$&$0.0087$&$0.0001$&$0.0290833177$&$0.0344969346$&$0.021109162$\\
0&8&$-0.0235$&$0$&$0$&$-0.0064656888$&$0$&$0$\\
1&6&$0.0286$&$0.0042$&$0$&$0.0256527482$&$0.0102167471$&$0$\\
2&4&$-0.0067$&$-0.0035$&$-0.0007$&$-0.0100109297$&$-0.0119194289$&$-0.0067485925$\\
3&2&$0.0002$&$0.0003$&$0.0004$&$0.0004314443$&$0.0014244969$&$0.0032292081$\\
4&0&$0$&$\dagger$&$\dagger$&$-0.000004563$&$\dagger$&$\dagger$\\
0&9&$0$&$0$&$0$&$0.0008075652$&$0$&$0$\\
1&7&$0$&$0$&$0$&$-0.0019941292$&$-0.001247346$&$0$\\
2&5&$0$&$0$&$0$&$0.0009633182$&$0.0014344316$&$0.0008961195$\\
3&3&$0$&$0$&$0$&$-0.0000891208$&$-0.0002884905$&$-0.0005158777$\\
\bottomrule\end{tabular}
\caption{Exact coefficient products for the umbilic profiles in \eqref{eq:fixed-umbilic-families}.
Only the $(0,2)$ row is multiplied by $\tau$.
The scalar-flat profile has sixteen nonzero rows of degree at most eight, and the minimal-boundary profile has twenty-three of degree at most nine.}
\label{tab:umbilic-coefficients}
\end{table}
\FloatBarrier
\section{Uniform sign certificates}\label{app:sign-witnesses}
The eleven algebraic signs in \eqref{eq:eleven-signs} are proved on each complete dimension range.
The scalar-flat formulas are rational in $N$.
The minimal formulas contain one gamma ratio, which we first bound uniformly from its recurrence.

\subsection{Gamma-ratio estimates}
\begin{lemma}\label{lem:gamma-recurrence}
For every real $N>6$, put $d:=N-11/2$ and let $w_N$ be \eqref{eq:minimal-gamma-parameter}.
Then
\begin{equation}\label{eq:gamma-recurrence-inequality}
\frac{d}{d^2+1/8}<\frac{\pi}{2}w_N^2<\frac1d.
\end{equation}
For $N\ge15$ and $y:=\sqrt d$, this implies
\begin{equation}\label{eq:gamma-rational-enclosure}
l(y):=c_-\frac{1-1/(16y^4)}y<w_N<\frac{c_+}y=:u(y),
\quad
c_-:=\frac{7978845608}{10^{10}},\quad
c_+:=\frac{7978845609}{10^{10}}.
\end{equation}
\end{lemma}
\begin{proof}
The functional equation for $\Gamma$ implies the exact recurrence
\[
w_{N+2}=\frac{N-5}{N-4}w_N.
\]
For $f_N:=\pi w_N^2/2$ it becomes $f_{N+2}/f_N=((d+1/2)/(d+3/2))^2$.
By log convexity of $\Gamma$, with $x:=(N-5)/2>1/2$,
\[
\frac1{\pi x}\le w_N^2\le\frac1{\pi(x-1/2)} .
\]
Multiplying by $\pi d/2$ and applying the squeeze theorem, we obtain $df_N\to1$ as $N\to\infty$.
Define $U(d):=1/d$ and $L(d):=d/(d^2+1/8)$.
Direct algebra shows that
\begin{align*}
\frac{f_{N+2}/U(d+2)}{f_N/U(d)}
&=1+\frac2{d(2d+3)^2}>1,\\
\frac{f_{N+2}/L(d+2)}{f_N/L(d)}
&=1-\frac{18}{(d+2)(2d+3)^2(8d^2+1)}<1 .
\end{align*}
For each fixed initial $N>6$, both normalized sequences tend to one along $N+2k$ as $k\to\infty$.
The first increases strictly and the second decreases strictly, proving \eqref{eq:gamma-recurrence-inequality}.
This argument applies from every real $N>6$, so it requires no separate initial checks for the two integer parities.

Taking square roots and applying $(1+t)^{-1/2}>1-t/2$ for $t>0$ proves \eqref{eq:gamma-rational-enclosure}.
For the fixed constants, put
\[
S_m(x):=\sum_{j=0}^{m}\frac{(-1)^jx^{2j+1}}{2j+1},\qquad
\pi_- :=\frac{314159265358979}{10^{14}},\quad
\pi_+ :=\frac{314159265358980}{10^{14}}.
\]
The alternating-series bounds and Machin's identity $\pi=16\arctan(1/5)-4\arctan(1/239)$ give
\[
\pi_-<16S_9(1/5)-4S_2(1/239)<\pi
<16S_{10}(1/5)-4S_3(1/239)<\pi_+.
\]
The two outer inequalities follow by clearing the denominators of the displayed finite sums.
Finally, the integer inequalities
\[
7978845608^2\cdot314159265358980
<2\cdot10^{34}
<7978845609^2\cdot314159265358979
\]
give $c_-^2\pi_+<2<c_+^2\pi_-$ and hence $c_-<\sqrt{2/\pi}<c_+$.
This is a finite arithmetic proof of the constant bounds and requires no supplementary verification file.
\end{proof}
The gamma ratio arises from exactly evaluated integrals and corrector pairings.

\subsection{Polynomial sign certificates}
\begin{lemma}
\label{lem:polynomial-positivity-criterion}
Let $P/Q$ be a rational function.
If, after $y=y_0+X$, both $P$ and $Q$ have nonnegative coefficients and positive constant terms, then
\[
\frac{P(y)}{Q(y)}>0\qquad(y\ge y_0).
\]
More generally, let $l(y)<u(y)$ and let $f(y,w)$ be polynomial in $w$.
Write $f(y,l(y)+v(u(y)-l(y)))=:\sum_{i=0}^d c_i(y)v^i$.
If every rational function
\[
B_j(y):=\sum_{i=0}^j c_i(y)\frac{\binom ji}{\binom di},\qquad 0\le j\le d,
\]
has such positive shifted numerator and denominator coefficients, then
\[
f(y,w)>0\qquad(y\ge y_0,\ l(y)\le w\le u(y)).
\]
\end{lemma}
\begin{proof}
Every polynomial with nonnegative coefficients and positive constant term is positive on $X\ge0$, proving the first assertion.
The monomial-to-Bernstein identity gives
\begin{equation}\label{eq:gamma-Bernstein-tests}
\sum_{i=0}^d c_i(y)v^i
=\sum_{j=0}^d B_j(y)\binom djv^j(1-v)^{d-j}.
\end{equation}

Indeed $v^i=\sum_{j=i}^d\binom ji\binom di^{-1}
\binom djv^j(1-v)^{d-j}$ by the binomial theorem.
For $0\le v\le1$, these basis polynomials are nonnegative and sum to one.
The positive coefficients $B_j(y)$ therefore give the strict positivity asserted in the second part.
\end{proof}

\begin{table}[ht]\centering\small
\begin{tabular}{@{}lclc@{}}\toprule
Family & Full range & Substitution, $X\ge0$ & Rational pairs\\\midrule
Scalar-flat nonumbilic & $N\ge15$ & $N=15+X$ & $11$\\
Scalar-flat umbilic & $N\ge22$ & $N=22+X$ & $11$\\
Minimal nonumbilic & $N\ge15$ & $y=77/25+X$ & $43$\\
Minimal umbilic & $N\ge21$ & $y=393/100+X$ & $26$\\\bottomrule
\end{tabular}
\caption{Uniform certificates, with $y=\sqrt{N-11/2}$.
Each pair consists of a numerator and denominator with nonnegative shifted coefficients and strictly positive constant terms.}
\label{tab:sign-verification-ranges}
\end{table}
In the computer-assisted proof of Proposition~\ref{prop:four-scalar-minima}, each scalar-flat rational expression in \eqref{eq:eleven-signs} is reduced with a positive leading denominator coefficient.
The finite checks establish the shifted coefficient conditions in Lemma~\ref{lem:polynomial-positivity-criterion} for the $22$ rational pairs in the first two rows of Table~\ref{tab:sign-verification-ranges}.

For the minimal families, \eqref{eq:coefficient-degree-identities} gives
\[
f_2/a=-1/(2d_0),\ 2d_0,\ 0,\ 0
\quad(f=-E,K,T_0,T_{1/2}).
\]
Thus the eleven expressions have degree at most four in $w$, with no new gamma denominators.
Substitute $N=y^2+11/2$ and the interval from Lemma~\ref{lem:gamma-recurrence}.
The same declared finite checks verify the shifted coefficient conditions for all $43+26$ Bernstein pairs.
The strict inequalities $(77/25)^2<19/2$ and $(393/100)^2<31/2$ give coverage from $N=15$ and $N=21$.
Their locations and reconstruction conventions are in Appendix~\ref{subsec:certificate-index}.

For the first minimal nonumbilic test, the Bernstein conversion is explicit:
\[
f(w):=-a_{\mathrm{I,nu}}=u_0+u_2w^2,\qquad
u_0:=\frac{N-2}{4(N+1)},\qquad
u_2:=\frac{(N-5)^2(N-3)^2}{(N-4)^2(N-2)(N+1)}.
\]
Write $l:=l(y)$ and $u:=u(y)$, with $N=y^2+11/2$.
Substituting $w=l+v(u-l)$ in \eqref{eq:gamma-Bernstein-tests}, we obtain
\[
B_0=u_0+u_2l^2,\qquad B_1=u_0+u_2lu,\qquad B_2=u_0+u_2u^2.
\]
Indeed $f(l+v(u-l))=B_0(1-v)^2+2B_1v(1-v)+B_2v^2$.
All three coefficients are positive for $N\ge15$, since $u_0,u_2,l,u>0$.
The remaining tests are certified by the shifted arrays described above.

The following is one complete printed certificate; the other exact arrays remain part of the electronic supplement.
For the degree-eight scalar-flat umbilic family, set $N=22+X$.
The first of the eleven tests has the complete certificate
\[
-a=\frac{P(X)}{Q(X)},\qquad X\ge0,
\]
\begin{align*}
P(X)&:=96 X^{4} + 7584 X^{3} + 224160 X^{2}\\
&\quad+2938336 X + 14414464,\\
Q(X)&:=3 X^{6} + 351 X^{5} + 17061 X^{4}\\
&\quad+440961 X^{3} + 6391368 X^{2} + 49253520 X + 157651200
\end{align*}
so $P,Q>0$ on the entire half-line.
The same exact coefficient test applies to each entry below; the two degrees count the complete stored numerator and denominator.
\begin{center}\small
\begin{tabular}{@{}lrr@{}}\toprule
Test & Numerator degree & Denominator degree\\\midrule
$-a$&4&6\\
$b$&11&13\\
$\Delta$&32&36\\
$A_{-E}$&11&13\\
$C_{-E}$&32&36\\
$A_K$&11&13\\
$C_K$&32&36\\
$A_{T_0}$&10&12\\
$C_{T_0}$&29&33\\
$A_{T_{1/2}}$&10&12\\
$C_{T_{1/2}}$&29&33\\
\bottomrule\end{tabular}\end{center}
Every shifted numerator and denominator has nonnegative coefficients and a positive constant.
Thus Lemma~\ref{lem:eleven-signs} applies for every real $N\ge22$, without evaluating a root at $N=22$ or at any other individual dimension.
The complete arrays and their reverse reconstruction are supplied with the uniform certificates.

As a specialization of the same uniform family, consider dimension twenty-one.
Take $N=21$ and the balanced generating matrix $A_{\rm b}=\operatorname{diag}(I_{10},-I_{10})/\sqrt{20}$, so that $A_{\rm b}^2=I_{20}/20$.
Use all twenty-three rows of Table~\ref{tab:umbilic-coefficients}, with only the $(0,2)$ row multiplied by $\tau$.
Its source includes both $h_{ij}\widetilde U_{ij}$ and $(\partial_i h_{ij})\widetilde U_j$, with its double-divergence contribution in the normalized source $\widetilde S[h]$ of Proposition~\ref{sec:certificate-normalization}.
Both divergence contributions in the raw energy are retained as well.
A divergence-free replacement would change these source and energy terms.

In the notation of \eqref{eq:coefficient-degree-identities}, write the coefficient vector as $\mathbf c_0+\tau\mathbf c_1$.
Evaluating the stationarity form from Lemma~\ref{lem:endpoint-hessian-structure} at $N=21$ gives the exact quadratic
\begin{align*}
p_{21}(\tau)&=-\frac{8c_{02}^2}{95}\tau^2+b_{21}\tau+c_{21},
&c_{02}&=-1.2158794046,\\
b_{21}&:=2M_{21}^{-1}\mathbf c_0^{\mathsf T}
(\mathbf D\mathbf M+\mathbf M\mathbf D)\mathbf c_1,\\
c_{21}&:=M_{21}^{-1}\mathbf c_0^{\mathsf T}
(\mathbf D\mathbf M+\mathbf M\mathbf D)\mathbf c_0.
\end{align*}
The table of defining rows and the exact corrector-pairing formulas specify $\mathbf c_0,\mathbf c_1,\mathbf D,\mathbf M$ without numerical choices; all are evaluated at $N=21$, with $M_{21}=M_{21}^{\mathrm I}$.
Here $w_{21}=4096/(6435\pi)$.

The coefficient is the larger root \eqref{eq:larger-coefficient-root}.
Its required signs follow from the same recurrence and polynomial identities used for every $N\ge21$; there is no separate root verification at $N=21$.
For this generating matrix every translation eigenvalue equals
\[
\lambda_{\rm tr}
:=T_0+\frac{T_{1/2}-T_0}{10}
=\frac9{10}T_0+\frac1{10}T_{1/2}.
\]
The uniform proof gives $T_0,T_{1/2}>0$, hence this actual translation eigenvalue is positive.
Lemma~\ref{lem:endpoint-hessian-structure} supplies the zero gradient and mixed block, and Lemma~\ref{lem:source-kernel-orthogonality} and Proposition~\ref{prop:metric-bubble-expansion} give the degree-nine source and
metric tails $O(R^{-1/2})$ and $O(R^{-1})$, respectively.
\section{Corrector estimates and verification data}\label{app:verification}
\subsection{Angular coercivity and residual estimates}
Write $\Sigma^{(q)}$ for functions whose tangential angular dependence is a degree-$q$ spherical harmonic on $\Sph^{m-1}$.
At the centered unit bubble, $\cB^{\mathrm{min}}$ denotes the Jacobi form for the minimal-boundary equation $(\sigma,\eta)=(1,0)$, and $\cB^{\mathrm{sf}}$ that for $(0,2)$.
The divergence-free polynomial family generated by $A$ uses only $q=2,3$, which are orthogonal to the bubble kernel.
The enlarged scalar family also has a degree-one first translated source.
Its inverse is taken on the kernel complement, not on the whole degree-one sector.

\begin{lemma}\label{a:lem:angular}
At the centered unit bubble the following inequalities hold for $v\in\Sigma^{(q)}$, $q\ge2$:
\begin{align}
\cB^{\mathrm{sf}}(v,v)&\ge \frac{2(q-1)}{N+2q-2}\int|\nabla v|^2,
&\cB^{\mathrm{sf}}(v,v)&\ge2(q-1)\int_{\pa\R^N_+}\frac{v^2}{1+s},\label{a:eq:IIcoerc}\\
\cB^{\mathrm{min}}(v,v)&\ge c_q\int|\nabla v|^2,
&c_q&:=\frac{4(q-1)(N+q)}{(N+2q)(N+2q-2)}.\label{a:eq:Icoerc}
\end{align}
Moreover,
\begin{equation}\label{a:eq:weightedtrace}
\int|\nabla v|^2\ge(N+2q-2)\int_{\pa\R^N_+}\frac{v^2}{1+s}.
\end{equation}
\end{lemma}
\begin{proof}
At the scalar-flat unit bubble, set $f:=v/U^{\mathrm{sf}}_{0,1}$ and regard $f$ as a function on the Euclidean ball of radius $1/2$ after conformal inversion.
Conformal covariance gives
\begin{equation}\label{a:eq:ballforms}
\cB^{\mathrm{sf}}(v,v)=\int_{B_{1/2}}|\nabla f|^2-2\int_{\pa B_{1/2}}f^2,
\qquad
\|v\|_\Sigma^2=\int_{B_{1/2}}|\nabla f|^2+(N-2)\int_{\pa B_{1/2}}f^2.
\end{equation}
Moreover,
\[
\int_{\pa\R^N_+}\frac{v^2}{1+s}=\int_{\pa B_{1/2}}f^2.
\]
The ball has Steklov eigenvalues $2q$, and tangential angular degree $q$ contains no lower total spherical degree.
Its Dirichlet energy is therefore at least $2q$ times the boundary norm; the zero-boundary-trace part contributes nonnegative energy.
The three scalar-flat inequalities, including \eqref{a:eq:weightedtrace}, follow by substitution.
At $\kappa=0$, set $f:=(1+s+t^2)^{(N-2)/2}v$ and regard $f$ as a function on the hemisphere of radius $1/2$, with
\[
\cB(v,v)=\int(|\nabla f|^2-4Nf^2),\qquad
\|v\|_\Sigma^2=\int(|\nabla f|^2+N(N-2)f^2).
\]
The first spherical eigenvalue available in angular degree $q$ is $4q(q+N-1)$.
Taking the ratio of the two forms yields \eqref{a:eq:Icoerc}.
\end{proof}

\begin{lemma}
\label{lem:projected-degree-one}
At $(\sigma,\eta)=(1,0)$ and $z=(0,1)$, restrict to $\Sigma^{(1)}\cap\Sigma_z$, with precisely the constraints \eqref{eq:common-orthogonality-functionals}.
Then
\begin{equation}
\cB^{\mathrm{min}}(v,v)\ge c_1^\perp\int|\nabla v|^2,
\qquad c_1^\perp:=\frac4{N+4},\qquad v\in\Sigma^{(1)}\cap\Sigma_{(0,1)}.
\end{equation}
The weighted trace estimate \eqref{a:eq:weightedtrace} holds in $\Sigma^{(1)}$ with its factor $N$.
\end{lemma}
\begin{proof}
Under the hemisphere transformation in the preceding proof, the constraints are exactly $L^2$ orthogonality to the spherical degree-one Jacobi fields.
Tangential angular degree one is orthogonal to the constant spherical mode; its spherical degree-one part consists of the tangential translations, which the constraints remove.
The remaining Neumann spectrum on the round hemisphere of radius $1/2$ starts no lower than $8(N+1)$.
Therefore the ratio of the two forms is at least
\[
\frac{8(N+1)-4N}{8(N+1)+N(N-2)}=\frac4{N+4}.
\]
The trace inequality follows from the ball Dirichlet comparison in \eqref{a:eq:ballforms}, which is valid also for $q=1$; its proof does not use positivity of the scalar-flat Jacobi form.
\end{proof}

\begin{lemma}
Let $Z,X\in\Sigma^{(q)}$, $q\ge2$, and suppose their residual satisfies
\[
\cB(Z-X,\psi)=\int R\psi+\int_{\pa\R^N_+}R_b\psi
\qquad(\psi\in\Sigma^{(q)}).
\]
Assume the weighted square integrals on the right sides below are finite.
At the scalar-flat unit bubble, with $D=s+(1+t)^2$,
\begin{equation}
\|Z-X\|_{\cB}\le
\left[\frac4{(N-2)^2}\frac{N+2q-2}{2(q-1)}\int DR^2\right]^{1/2}
+\left[\frac1{2(q-1)}\int_{\pa\R^N_+}(1+s)R_b^2\right]^{1/2}.
\end{equation}
At the minimal unit bubble, with $D=1+s+t^2$ and $c_q$ from \eqref{a:eq:Icoerc},
\begin{equation}
\|Z-X\|_{\cB}\le
\left[\frac4{(N-2)^2c_q}\int DR^2\right]^{1/2}
+\left[\frac1{c_q(N+2q-2)}\int_{\pa\R^N_+}(1+s)R_b^2\right]^{1/2}.
\end{equation}
The latter estimate also holds for $Z,X\in\Sigma^{(1)}\cap\Sigma_{(0,1)}$, with the constraints of Lemma~\ref{lem:projected-degree-one}, a residual tested on that complement, and $c_q=c_1^\perp$, $N+2q-2=N$.
\end{lemma}
\begin{proof}
For the scalar-flat case, Hardy's inequality with pole $(0,-1)$ gives
\[
\int v^2/D\le\frac4{(N-2)^2}\int|\nabla v|^2.
\]
To obtain it, integrate the vector field $(x+e_N)/D$; its boundary flux is nonpositive.
At the minimal bubble use radial Hardy and $D\ge s+t^2$.
Apply weighted Cauchy--Schwarz to the residual identity with $v=Z-X$.
The interior term is controlled by Hardy and the appropriate coercivity constant.
The boundary term is controlled by \eqref{a:eq:IIcoerc} in the scalar-flat case and by \eqref{a:eq:Icoerc} and \eqref{a:eq:weightedtrace} in the minimal case.
Division by $\|v\|_{\cB}$ proves the estimates when this norm is nonzero; the zero case is immediate.
For the minimal degree-one complement, use Lemma~\ref{lem:projected-degree-one}, whose trace factor is $N$.
\end{proof}

For the minimal degree-one estimate, replace a trial by its projection \eqref{eq:explicit-Sigma-projection} unless it already satisfies the constraints.
Kernel compatibility leaves its source pairing unchanged.

\subsection{Verification of Proposition~\ref{prop:four-scalar-minima}}
\label{subsec:certificate-index}
The electronic supplement contains the program \texttt{verify\_proposition\_5\_8.py} for the computer-assisted part of Proposition~\ref{prop:four-scalar-minima}.
It contains the defining coefficients and all reconstruction and verification routines, with no project imports or previously computed models or sign certificates as inputs.
The program requires Python~3.10 or later and SymPy~1.14.0, with assertions enabled.
Place it beside \texttt{main.tex} and run
\begin{verbatim}
python verify_proposition_5_8.py --manuscript main.tex
\end{verbatim}
The program first compares its $69$ defining rows with the coefficient tables in Appendix~\ref{a:app:coefficients}.
The computations use the embedded exact rows; the manuscript is read only for this comparison.

The scalar-flat boundary responses are constructed from the finite sums in Lemma~\ref{lem:endpoint-boundary-response}, including the projected degree-one terms.
For the minimal responses, the program constructs both parity telescopers for the forty entries with $q=1,2$, $0\le a\le b\le4$, or $q=3$, $0\le a\le b\le3$.
In each of the eighty cases, it verifies the successive-term ratio and the polynomial identity \eqref{eq:minimal-polynomial-telescoper}, forms $r(k):=-B_0(k-1)P_*(k)/C_0(k)$, and checks $r(0)$.
It also checks symmetry of the responses and the two explicit values at $q=2$, $a=b=0$ in Lemma~\ref{lem:endpoint-boundary-response}.

The pole checks cover the whole convergence range.
Substitute $N=\max\{4,2(a+b+q)\}+X$ with $X>0$.
After positive normalization, the denominators of the coefficients of $P_*$ have nonnegative coefficients in $X$ and are positive for $X>0$.
After common powers of $k$ are canceled, the denominator of $r(k)$ has nonnegative coefficients in $(X,k)$ and a constant coefficient in $k$ that is positive for $X>0$.
Thus $r(k)$ is defined for every $k\ge0$ in this range.
The program checks that its numerator degree exceeds its denominator degree by at most one, so $r(k)=O(k)$ as $k\to\infty$.
Together with the decay established after \eqref{eq:minimal-telescoping-identity}, this bound proves that the terminal term tends to zero; no truncated spectral sum is used.

The program then reconstructs the four scalar-flat and minimal-boundary families from their defining rows.
The first and second translated sources retain all divergence terms in \eqref{eq:polynomial-source-functional}, with the normalization in Proposition~\ref{sec:certificate-normalization}.
Polynomial correctors, volume moments, local boundary terms, and exact boundary responses determine the five normalized forms $E,p,K,T_0,T_{1/2}$ of Proposition~\ref{prop:endpoint-matrix-assembly}.
Both inverse pairings in \eqref{eq:full-source-Hessian-pairing} are included, with the degree-one source projected off the Jacobi kernel.
The executable names are \texttt{E}, \texttt{p}, \texttt{K}, \texttt{T\_0}, and \texttt{T\_half}; its symbols \texttt{N}, \texttt{tau}, and \texttt{w} represent $N$, $\tau$, and $w_N$.
Only the distinguished row receives the multiplier $\tau$, which is held fixed under translation and dilation.

The eleven tests in \eqref{eq:eleven-signs} are derived from these reconstructed forms.
For the scalar-flat families, the program verifies positivity of the numerator and denominator coefficients after $N=15+X$ or $N=22+X$.
For the minimal families, it checks the rational constants and recurrence identities used in Lemma~\ref{lem:gamma-recurrence}, applies the resulting enclosure for $w_N$, and constructs the Bernstein coefficients in Appendix~\ref{app:sign-witnesses}.
Their numerator and denominator coefficients are nonnegative after the shifts in Table~\ref{tab:sign-verification-ranges}, and their constant coefficients are positive.
Every dimension shift and Bernstein conversion is reversed exactly and compared with the reconstructed expression.
The $11+11+43+26=91$ rational pairs therefore establish the required signs throughout the four integer dimension ranges, without numerical root selection or dimension sampling.

The output \texttt{verification\_results.json} records the reconstructed forms, telescopers, sign certificates, software versions, and manuscript and program hashes.
It is an execution record, not an input to the computation.
The program stops at a failed identity or sign check and refuses execution with assertions disabled.
The summation limits, angular decomposition, and root criterion are justified in the text; the nonlinear localization is proved in Section~\ref{sec:common-analytic}.

For the dimension-twenty-one specialization, substitute $N=21$ and $w=w_{21}=4096/(6435\pi)$ in the reconstructed form $p$.
Its coefficients in increasing powers of $\tau$ are $c_{21},b_{21},-8c_{02}^2/95$, as in Appendix~\ref{app:sign-witnesses}.

\section{Explicit bounds for negative mean curvature}
\label{app:normal-weyl-bounds}
We prove Corollary~\ref{cor:explicit-negative-bounds} using tensors $h(\bar x,t):=f(t)H^W(\bar x)$, where $H^W$ is the quadratic Weyl tensor defined below.
Throughout this appendix $\kappa=-L<0$, $m:=N-1$, $x=({\bar x},t)\in\R^m\times[0,\infty)$, and
\[
D_L:=1+|{\bar x}|^2+(t-L)^2,\qquad w_L:=D_L^{-(N-2)}.
\]
We use the normalization $\widehat F$ of \eqref{eq:negative-energy-normalization}.

\subsection{Weyl tensors and the translation Hessian}
Split the $m$ tangential coordinates into two sets of sizes $p,q\ge2$, $p+q=m$.
Prescribe the sectional components of an algebraic curvature tensor by
\[
W_{ijij}:=s_{ij}:=\begin{cases}
q/[(p-1)(m-1)],&i,j\text{ in the first set},\\
p/[(q-1)(m-1)],&i,j\text{ in the second set},\\
-1/(m-1),&i,j\text{ in different sets},
\end{cases}\qquad i\ne j.
\]
Impose the curvature symmetries and set the other components to zero.
Put
\[
H^W_{ab}:=W_{aibj}x_ix_j,
\qquad H^W_{ii}=\sum_{j\ne i}s_{ij}x_j^2,
\qquad H^W_{ij}=-s_{ij}x_ix_j\quad(i\ne j),
\]
with zero normal components.
\begin{lemma}
For the component construction above, with $m=p+q$ and $p,q\ge2$, $W$ is a nonzero algebraic Weyl tensor and
\begin{equation}\label{eq:weyl-polynomial-identities}
\tr H^W=0,\qquad \divg H^W=0,\qquad H^W{\bar x}=0,\qquad
\Delta_{\bar x}H^W=0.
\end{equation}
The map $v\mapsto\partial_vH^W$ is injective.
\end{lemma}
\begin{proof}
The curvature operator is diagonal on $\mathbf e_i\wedge\mathbf e_j$, so Bianchi holds; each row sum is zero, so its Ricci contraction vanishes.
Thus $W$ is a nonzero algebraic Weyl tensor.
The four identities in \eqref{eq:weyl-polynomial-identities} follow by contraction and differentiation of $H^W_{ab}=W_{aibj}x_ix_j$, using the curvature symmetries and vanishing Ricci contraction.
For $i\ne j$, $\partial_vH^W_{ij}=-s_{ij}(v_ix_j+v_jx_i)$.
Since every $s_{ij}$ is nonzero, $\partial_vH^W=0$ implies $v=0$.

\end{proof}

\begin{lemma}\label{lem:weyl-normal-translation}
Let $h({\bar x},t):=f(t)H^W({\bar x})$, where $f$ is a polynomial and $N>2(2+\deg f)+2$.
At the centered unit bubble,
\begin{align}
\widehat F_h(0,1)
&=-\frac14\int_{\R^N_+}w_L
\bigl(f^2|\nabla_{\bar x}H^W|^2+(f')^2|H^W|^2\bigr),
\label{eq:weyl-normal-energy}\\
\partial_{\xi_a\xi_b}\widehat F_h(0,1)
&=-\frac12\int_{\R^N_+}w_L(f')^2
\langle\partial_aH^W,\partial_bH^W\rangle.
\label{eq:weyl-normal-translation}
\end{align}
If $f'\not\equiv0$, the translation Hessian is negative definite.
The translation gradient and the mixed translation--scale block vanish.
\end{lemma}
\begin{proof}
By \eqref{eq:weyl-polynomial-identities}, the centered kinetic term and scalar source vanish.
This proves \eqref{eq:weyl-normal-energy}.
For $h_\xi({\bar x},t):=f(t)H^W({\bar x}+\xi)$, radial transversality implies $h_\xi {\bar x}=-h_\xi\xi$ and ${\bar x}^Th_\xi {\bar x}=\xi^Th_\xi\xi$.
Hence the translated source is $O(|\xi|^2)$, and its inverse pairing has zero Hessian at $\xi=0$.
Differentiating the raw terms, we obtain
\begin{align*}
\partial_{\xi_a\xi_b}\widehat F_h(0,1)
={}&-\frac14\int w_L f^2\partial_{ab}|\nabla_{\bar x}H^W|^2
+4(N-1)(N-2)\int w_L D_L^{-2}f^2((H^W)^2)_{ab}\\
&-\frac14\int w_L(f')^2\partial_{ab}|H^W|^2.
\end{align*}
The first two integrals cancel.
To see this, write $H^W_{ij}=A_{ijrs}x_rx_s$, where $A_{ijrs}:=(W_{irjs}+W_{isjr})/2$, and use
\[
\int_{\R^m}w_LD_L^{-2}x_rx_sx_ux_v\dd {\bar x}
=\frac{\int_{\R^m}w_L\dd {\bar x}}{4(N-1)(N-2)}
(\delta_{rs}\delta_{uv}+\delta_{ru}\delta_{sv}+\delta_{rv}\delta_{su}).
\]
This identity follows from the radial beta integral.
The cancellation follows from the pair symmetries and zero traces of $A$.
Finally,
\[
\partial_{ab}|H^W|^2
=2\langle\partial_aH^W,\partial_bH^W\rangle
+2\langle H^W,\partial_{ab}H^W\rangle.
\]
The last term has zero radial average because $H^W$ is harmonic quadratic and $\partial_{ab}H^W$ is constant.
This proves \eqref{eq:weyl-normal-translation}; definiteness follows from injectivity.
Tangential parity at every scale proves the remaining assertions.
\end{proof}

\subsection{Proof of the explicit bounds}
\begin{proof}[Proof of Corollary~\ref{cor:explicit-negative-bounds}]
Both cases below have $N\ge9$, so the tangential moments used here converge, including the ratio with denominator $N-7$.
Define $C_W>0$ by
\begin{equation}\label{eq:weyl-normal-tangential-moments}
\int_{\R^m}w_L|\nabla_{\bar x}H^W|^2\dd {\bar x}
=:C_W[1+(t-L)^2]^{-(N-5)/2}.
\end{equation}
The degree-two harmonic identity and the radial beta integral also imply
\[
\frac{\int_{\R^m}w_L|H^W|^2\dd {\bar x}}
{\int_{\R^m}w_L|\nabla_{\bar x}H^W|^2\dd {\bar x}}
=\frac{1+(t-L)^2}{2(N-7)}.
\]
In each of the following two cases, $C_W\begin{pmatrix}A&B\\B&C\end{pmatrix}$ is the weighted gradient Gram matrix of the two displayed polynomials.
It is positive definite: a linear combination with zero gradient would be constant, whereas its positive tangential degree forces both coefficients to vanish.

\emph{Nonumbilic backgrounds.}
Let $N\ge9$, $d:=N-8$, and
\[
I_j:=\int_{-L}^{\infty}{\widehat t}^j(1+{\widehat t}^2)^{-(d+3)/2}\dd {\widehat t},
\qquad I:=I_0,\quad J:=I_1.
\]
Direct integration shows that
\[
J=\frac{(1+L^2)^{-(d+1)/2}}{d+1}>0,
\qquad dI_2=I-(d+1)LJ.
\]
For the polynomials $H^W,tH^W$, \eqref{eq:weyl-normal-tangential-moments} therefore yields
\[
A=I,\qquad B=LI+J,\qquad
C=\left(L^2+\frac3{2d}\right)I
+\left(1-\frac3{2d}\right)LJ.
\]
In particular,
\begin{equation}
25B^2-24AC
=\left(L^2-\frac{36}{d}\right)I^2
+\left(26+\frac{36}{d}\right)LIJ+25J^2>0
\quad\text{if }L\ge\frac6{\sqrt d}.
\end{equation}
Choose $h:=(1-t/L)H^W$, and set $b:=B/L$, $c:=C/L^2$ and $\Delta:=25b^2-24Ac$.
Homogeneous scaling gives
\[
\widehat F_h(0,\eps)
=-\frac{C_W}{4}\bigl(A\eps^4-2b\eps^5+c\eps^6\bigr).
\]
Since $b,c>0$ and $0<\Delta<25b^2$, its two positive critical scales are
\[
\eps_\pm:=\frac{5b\pm\sqrt\Delta}{6c}.
\]
At the upper root, with $\ell=\log\eps$,
\[
\partial_{\ell\ell}\widehat F_h(0,\eps_+)
=-\frac{C_W}{2}\eps_+^5\sqrt\Delta<0.
\]
Lemma~\ref{lem:weyl-normal-translation}, applied to the scaled polynomial, gives a negative-definite translation block.
The Gram matrix is positive, so the value is negative.
Thus $(0,\eps_+)$ is a negative nondegenerate full maximum.

To obtain nonzero trace-free second fundamental form at the concentration center, replace $h$ by $h-2\theta tA_0$, where $A_0\ne0$ is a fixed symmetric trace-free tangential matrix and $\theta>0$ is sufficiently small.
The corrected functional depends $C^2$ on the bubble parameters and quadratically on the tensor coefficients.
The implicit-function theorem preserves the negative full maximum.
Tangential parity keeps its center at zero.
Since $H^W(0)=0$, the second fundamental form of $\exp(\mu h)$ there is $\mu\theta A_0\ne0$.
The degree-two Weyl term is unchanged, so the curvature of $\exp(\mu h)$ remains nonzero for sufficiently small $\mu\ne0$.
The degree is three and $N>8$, so Theorem~\ref{thm:fixed-metric-transfer} proves the nonumbilic assertion of Corollary~\ref{cor:explicit-negative-bounds}.

\emph{Umbilic backgrounds.}
Let $N\ge11$, $d:=N-10$, and now set
\[
I_j:=\int_{-L}^{\infty}{\widehat t}^j(1+{\widehat t}^2)^{-(d+5)/2}\dd {\widehat t},
\qquad I:=I_0,\quad J:=I_1.
\]
By successive integration by parts,
\begin{align*}
J&=\frac{(1+L^2)^{-(d+3)/2}}{d+3},&
I_2&=\frac{I-(d+3)LJ}{d+2},\\
I_3&=\frac{2+(d+3)L^2}{d+1}J,&
I_4&=\frac{3I_2-(d+3)L^3J}{d}.
\end{align*}
For $H^W,t^2H^W$, the Gram entries are
\begin{align*}
A={}&I,\qquad B=\left(L^2+\frac1{d+2}\right)I+\frac{d+1}{d+2}LJ,\\
C={}&\left(L^4+\frac{8L^2}{d+2}+\frac5{d(d+2)}\right)I\\
&+\frac{L\{(d^3+2d^2+3d-10)L^2+7d^2+4d-15\}}
{d(d+1)(d+2)}J.
\end{align*}
Their scale discriminant is
\begin{equation}\label{eq:weyl-even-discriminant}
9B^2-8AC=c_{20}I^2+c_{11}IJ+c_{02}J^2,
\end{equation}
where
\begin{align*}
c_{20}&:=L^4-\frac{46L^2}{d+2}-\frac{31d+80}{d(d+2)^2},\\
c_{11}&:=\frac{2L}{d(d+1)(d+2)^2}
\bigl[(5d^4+20d^3+17d^2+34d+80)L^2
-19d^3-54d^2+37d+120\bigr],\\
c_{02}&:=\frac{9L^2(d+1)^2}{(d+2)^2}.
\end{align*}
Put $c_*:=23+4\sqrt{35}$ and suppose $L^2\ge c_*/d$.
Then $46<c_*<47$ and $c_*^2=46c_*+31$.
For $X:=L^2-c_*/d\ge0$,
\[
c_{20}=X^2+\left(\frac{2c_*}{d}-\frac{46}{d+2}\right)X
+\frac{(92c_*+44)d+4c_*^2}{d^2(d+2)^2}>0.
\]
The bracket in $c_{11}$ equals
\begin{align*}
&(5d^4+20d^3+17d^2+34d+80)X
+(5c_*-19)d^3+(20c_*-54)d^2\\
&\hspace{12mm}+(17c_*+37)d+34c_*+120+80c_*/d>0.
\end{align*}
Thus all three coefficients in \eqref{eq:weyl-even-discriminant} are positive, and $9B^2-8AC>0$.

Choose $h:=(1-t^2/L^2)H^W$.
With $r:=\eps^2$, $b:=B/L^2$, $c:=C/L^4$ and $\Delta:=9b^2-8Ac$, its centered energy is
\[
\widehat F_h(0,\eps)=-\frac{C_W}{4}(Ar^2-2br^3+cr^4).
\]
The upper critical value and its scale Hessian are
\[
r_+:=\frac{3b+\sqrt\Delta}{4c}>0,\qquad
\partial_{\ell\ell}\widehat F_h(0,\sqrt{r_+})
=-2C_Wr_+^3\sqrt\Delta<0.
\]
The translation block is negative definite by Lemma~\ref{lem:weyl-normal-translation}, and positivity of the Gram matrix makes the value negative.
The boundary is totally geodesic because $\partial_t h({\bar x},0)=0$, and the degree-two part has nonzero linearized Weyl curvature at the origin.
Since $\deg h=4$ and $N>10$, Theorem~\ref{thm:fixed-metric-transfer} proves the umbilic assertion of Corollary~\ref{cor:explicit-negative-bounds}.
\end{proof}


\begin{thebibliography}{99}
\bibitem{AlmarazCompactness}
S.~Almaraz,
\emph{A compactness theorem for scalar-flat metrics on manifolds with boundary},
Calc. Var. Partial Differential Equations \textbf{41} (2011), 341--386.

\bibitem{AlmarazBlowup}
\bysame,
\emph{Blow-up phenomena for scalar-flat metrics on manifolds with boundary},
J. Differential Equations \textbf{251} (2011), 1813--1840.

\bibitem{AlmarazWang}
S.~Almaraz and S.~Wang,
\emph{A compactness theorem for conformal metrics with constant scalar curvature and constant boundary mean curvature in dimension three},
Calc. Var. Partial Differential Equations \textbf{64} (2025), Article No.~35.

\bibitem{Brendle}
S.~Brendle,
\emph{Blow-up phenomena for the Yamabe equation},
J. Amer. Math. Soc. \textbf{21} (2008), 951--979.

\bibitem{BrendleMarques}
S.~Brendle and F.~C. Marques,
\emph{Blow-up phenomena for the Yamabe equation II},
J. Differential Geom. \textbf{81} (2009), 225--250.

\bibitem{ChenRuanSun}
X.~Chen, Y.~Ruan, and L.~Sun,
\emph{The Han--Li conjecture in constant scalar curvature and constant
boundary mean curvature problem on compact manifolds},
Adv. Math. \textbf{358} (2019), Paper No.~106854, 56 pp.

\bibitem{ChenWu}
X.~Chen and N.~Wu,
\emph{Blow-up phenomena for the constant scalar curvature and constant
boundary mean curvature equation},
J. Differential Equations \textbf{269} (2020), 9432--9470.

\bibitem{Cherrier}
P.~Cherrier,
\emph{Probl\`emes de Neumann non lin\'eaires sur les vari\'et\'es riemanniennes},
J. Funct. Anal. \textbf{57} (1984), 154--206.

\bibitem{DisconziKhuri}
M.~M. Disconzi and M.~A. Khuri,
\emph{Compactness and non-compactness for the Yamabe problem on manifolds
with boundary},
J. Reine Angew. Math. \textbf{724} (2017), 145--201.

\bibitem{EscobarAnnals}
J.~F. Escobar,
\emph{Conformal deformation of a Riemannian metric to a scalar flat metric
with constant mean curvature on the boundary},
Ann. of Math. (2) \textbf{136} (1992), 1--50; Addendum, \textbf{139} (1994), 749--750.

\bibitem{EscobarJDG}
\bysame,
\emph{The Yamabe problem on manifolds with boundary},
J. Differential Geom. \textbf{35} (1992), 21--84.

\bibitem{EscobarIndiana}
\bysame,
\emph{Conformal deformation of a Riemannian metric to a constant scalar curvature metric with constant mean curvature on the boundary},
Indiana Univ. Math. J. \textbf{45} (1996), 917--943.

\bibitem{GongKimMussoWeiCompactness}
L.~Gong, S.~Kim, M.~Musso, and J.~Wei,
\emph{On Critical Dimensions for Compactness in the Boundary Yamabe Problem, II},
in preparation.

\bibitem{HanLiDuke}
Z.-C.~Han and Y.~Y.~Li,
\emph{The Yamabe problem on manifolds with boundary: Existence and
compactness results},
Duke Math. J. \textbf{99} (1999), 489--542.

\bibitem{HanLiCAG}
\bysame,
\emph{The existence of conformal metrics with constant scalar curvature and
constant boundary mean curvature},
Comm. Anal. Geom. \textbf{8} (2000), 809--869.

\bibitem{HoShin}
P.~T. Ho and J.~Shin,
\emph{Blow-up phenomena for the constant scalar curvature and constant
boundary mean curvature equation (after Chen and Wu)},
NoDEA Nonlinear Differential Equations Appl. \textbf{32} (2025), Article No.~130.

\bibitem{KhuriMarquesSchoen}
M.~A. Khuri, F.~C. Marques, and R.~M. Schoen,
\emph{A compactness theorem for the Yamabe problem},
J. Differential Geom. \textbf{81} (2009), 143--196.

\bibitem{KimMussoWei}
S.~Kim, M.~Musso, and J.~Wei,
\emph{Compactness of scalar-flat conformal metrics on low-dimensional
manifolds with constant mean curvature on boundary},
Ann. Inst. H. Poincar\'e C Anal. Non Lin\'eaire \textbf{38} (2021), 1763--1793.
\end{thebibliography}
\end{document}